\documentclass[oneside, a4paper, 11pt]{amsart}
\usepackage[utf8]{inputenc}
\usepackage[T1]{fontenc}

\usepackage{amsmath,amsthm,amssymb,bbm,comment,mathrsfs}

\pdfoutput=1
\usepackage[ttscale=.85]{libertine}
\usepackage{libertinust1math}
\DeclareSymbolFont{cmarrows}{OMS}{cmsy}{m}{n}
\DeclareMathSymbol{\mapstochar}{\mathrel}{cmarrows}{"37}

\usepackage[tracking=true]{microtype}
\SetTracking
  {encoding=T1, shape=sc}
  {10}
\SetProtrusion
  {encoding=T1,size={7,8}}
  {1={ ,750},2={ ,500},3={ ,500},4={ ,500},5={ ,500},
    6={ ,500},7={ ,600},8={ ,500},9={ ,500},0={ ,500}}

\newcommand*{\cat}[1]{\ensuremath{\mathcal{#1}}} 

\newcommand*{\id}{\mathrm{id}}
\renewcommand{\colon}{\!\nobreak\mskip2mu\mathpunct{}\nonscript%
\mkern-\thinmuskip{:}\mskip6muplus1mu\relax%
}
\newcommand{\from}{\colon}
\newcommand*{\defeq}{\mathrel{\vcenter{\baselineskip0.5ex \lineskiplimit0pt
    \hbox{\scriptsize.}\hbox{\scriptsize.}}}%
=}

\NewDocumentCommand{\tetramod}{O{} O{} O{} O{}}{%
  \leftidx{^{#1}_{#2}}{\mathbf{Vec}}{_{#3}^{#4}}%
}
\NewDocumentCommand{\tetramodfd}{O{} O{} O{} O{}}{%
  \leftidx{^{#1}_{#2}}{\mathbf{vecfd}}{_{#3}^{#4}}%
}
\newcommand*{\blank}{{-}}

\newcommand*{\kVect}{\tetramod{}}

\newcommand*{\op}{\ensuremath{\mathrm{op}}}

\newcommand*{\ld}[1]{\leftidx{^\vee}{\!#1}{}}     
\newcommand*{\rd}[1]{{#1}^\vee}                   
\newcommand*{\ev}{\mathrm{ev}}
\newcommand*{\coev}{\mathrm{coev}}

\renewcommand*{\to}{\longrightarrow}
\renewcommand*{\mapsto}{\longmapsto}

\renewcommand*{\hom}[1]{\ensuremath{\lfloor#1\rfloor}}

\usepackage{scalerel}   
\makeatletter
\let\@tl\triangleleft
\let\@tr\triangleright
\newcommand*{\@smallTriangle}[2]{\vcenter{\hbox{\scalebox{0.75}{\ensuremath{#1#2}}}}}
\newcommand*{\@medTriangle}[2]{\vcenter{\hbox{\scalebox{0.90}{\ensuremath{#1#2}}}}}
\newcommand*{\lact}{\mathbin{\mathpalette\@smallTriangle\@tr}}
\newcommand*{\ract}{\mathbin{\mathpalette\@smallTriangle\@tl}}
\let\@btl\blacktriangleleft%
\let\@btr\blacktriangleright%
\newcommand*{\blact}{\mathbin{\mathpalette\@medTriangle\@btr}}
\newcommand*{\bract}{\mathbin{\mathpalette\@medTriangle\@btl}}
\renewcommand*{\triangleright}{\lact}
\renewcommand*{\triangleleft}{\ract}
\renewcommand*{\blacktriangleright}{\blact}
\renewcommand*{\blacktriangleleft}{\bract}

\makeatother

\DeclareRobustCommand{\SkipTocEntry}[5]{}

\numberwithin{equation}{section}

\numberwithin{figure}{section}

\usepackage{thmtools,mathtools}
\usepackage{graphicx,enumitem,stmaryrd}
\usepackage{multicol}
\usepackage{xparse}
\usepackage{leftidx}    

\usepackage{anyfontsize}
\allowdisplaybreaks[4]
\usepackage[all]{xy}
\usepackage[table]{xcolor}
\usepackage[normalem]{ulem}
\usepackage{bm, tensor}
\usepackage{tikz,tikz-cd}
\usetikzlibrary{matrix}
\usetikzlibrary{fit}
\usetikzlibrary{calc}
\tikzcdset{
  arrow style=tikz,
  diagrams={>={Straight Barb}}
}

\definecolor{blue}{rgb}{0.38, 0.51, 0.71}
\definecolor{red}{RGB}{175, 49, 39}
\definecolor{green}{RGB}{146, 227, 95}
\usepackage[english]{babel}%
\usepackage[colorlinks, linktoc=page, linkcolor={red}, citecolor={green!65!black}, urlcolor={blue}]{hyperref}
\addto\extrasenglish{%
}

\newcommand{\xtwoheadrightarrow}[2][]{%
  \xrightarrow[#1]{#2}\mathrel{\mkern-14mu}\rightarrow%
}

\makeatletter
\def\slashedarrowfill@#1#2#3#4#5{%
$\m@th\thickmuskip0mu\medmuskip\thickmuskip\thinmuskip\thickmuskip
  \relax#5#1\mkern-7mu%
  \cleaders\hbox{$#5\mkern-2mu#2\mkern-2mu$}\hfill
  \mathclap{#3}\mathclap{#2}%
  \cleaders\hbox{$#5\mkern-2mu#2\mkern-2mu$}\hfill
  \mkern-7mu#4$%
}
\def\rightslashedarrowfill@{%
\slashedarrowfill@\relbar\relbar\mapstochar\rightarrow}
\newcommand\xslashedrightarrow[2][]{%
\ext@arrow 0055{\rightslashedarrowfill@}{#1}{#2}}
\makeatother

\makeatletter
\newcommand{\ostar}{\mathbin{\mathpalette\make@circled\star}}
\newcommand{\make@circled}[2]{%
\ooalign{$\m@th#1\smallbigcirc{#1}$\cr\hidewidth$\m@th#1#2$\hidewidth\cr}%
}
\newcommand{\smallbigcirc}[1]{%
\vcenter{\hbox{\scalebox{0.77778}{$\m@th#1\bigcirc$}}}%
}
\makeatother

\usepackage{adjustbox}

\usepackage{CJKutf8}
\newcommand*{\yo}{\text{\begin{CJK}{UTF8}{min}よ\end{CJK}}}

\tikzcdset{scale cd/.style={every label/.append style={scale=#1},
  cells={nodes={scale=#1}}}}

\calclayout

\usetikzlibrary{calc}
\usetikzlibrary{decorations.pathmorphing}
\tikzset{curve/.style={settings={#1},to path={(\tikztostart)
  .. controls ($(\tikztostart)!\pv{pos}!(\tikztotarget)!\pv{height}!270:(\tikztotarget)$)
  and ($(\tikztostart)!1-\pv{pos}!(\tikztotarget)!\pv{height}!270:(\tikztotarget)$)
  .. (\tikztotarget)\tikztonodes}},
  settings/.code={\tikzset{quiver/.cd,#1}
      \def\pv##1{\pgfkeysvalueof{/tikz/quiver/##1}}},
  quiver/.cd,pos/.initial=0.35,height/.initial=0}

\tikzset{tail reversed/.code={\pgfsetarrowsstart{tikzcd to}}}
\tikzset{2tail/.code={\pgfsetarrowsstart{Implies[reversed]}}}
\tikzset{2tail reversed/.code={\pgfsetarrowsstart{Implies}}}
\tikzset{no body/.style={/tikz/dash pattern=on 0 off 1mm}}

\definecolor{blue(pigment)}{rgb}{0.2, 0.2, 0.6}
\definecolor{americanrose}{rgb}{1.0, 0.01, 0.24}
\definecolor{nicegreen}{rgb}{0.0, 0.5, 0.0}
\definecolor{deepmagenta}{rgb}{0.8, 0.0, 0.8}
\definecolor{deepcarrotorange}{rgb}{0.91, 0.41, 0.17}
\definecolor{cadetgrey}{rgb}{0.57, 0.64, 0.69}

\newtheorem{theoremm}{Theorem}[section]

\declaretheorem[style=plain,name=Theorem,numberlike=theoremm]{theorem}
\declaretheorem[style=plain,name=Theorem,numbered=no]{theorem*}

\declaretheorem[style=plain,name=Lemma,numberlike=theoremm]{lemma}
\declaretheorem[style=plain,name=Proposition,numberlike=theoremm]{proposition}
\declaretheorem[style=plain,name=Corollary,numberlike=theoremm]{corollary}

\declaretheorem[style=plain,name=Conjecture,numberlike=theoremm]{conjecture}
\declaretheorem[style=definition,name=Definition,numberlike=theorem]{definition}

\declaretheorem[style=remark,name=Example,numberlike=theorem]{example}
\declaretheorem[style=remark,name=Remark,numberlike=theorem]{remark}
\declaretheorem[style=remark,name=Notation,numberlike=theorem]{notation}

\usepackage[cal=boondoxo, calscaled=0.96, bb=dsserif]{mathalfa}

\newcommand{\csym}[1]{\mathcal{#1}}

\newcommand{\on}[1]{\operatorname{#1}}
\newcommand{\setj}[1]{\left\{ #1 \right\}}

\DeclareFontFamily{U}{DSSerif}{\skewchar \font =45}
\DeclareFontShape{U}{DSSerif}{m}{n}{<-> s*[1]  DSSerif}{}
\DeclareFontSubstitution{U}{DSSerif}{m}{n}
\DeclareMathAlphabet{\mathbbbb}{U}{DSSerif}{m}{n}

\DeclareMathAlphabet\EuRoman{U}{eur}{m}{n}
\SetMathAlphabet\EuRoman{bold}{U}{eur}{b}{n}
\newcommand{\euler}{\EuRoman}
\newcommand{\kotimes}{\otimes_{\Bbbk}}

\newcommand{\projca}{\on{proj}_{\mathcal{C}}(A)}
\newcommand{\modca}{\on{mod}_{\mathcal{C}}(A)}
\newcommand{\modw}[1]{\on{mod}_{\mathcal{C}}(#1)}

\newcommand{\radca}{\on{Rad}_{A}^{\mathcal{C}}}

\newcommand{\pfa}{\on{Kl}_{\cp\otimes A}}
\newcommand{\apf}{\on{Kl}_{A\otimes \cp}}
\newcommand{\cp}{\mathcal{C}_{p}}

\newcommand{\presh}{[\mathcal{C}^{\on{op}},\mathbf{Vec}]}
\newcommand{\preshcp}{[\cp^{\on{op}},\mathbf{vec}]}
\newcommand{\homa}{\on{Hom}_{\blank A}}
\newcommand{\ahom}{\on{Hom}_{A\blank}}
\DeclareMathOperator{\topp}{top}

\newcommand{\cC}{\mathcal{C}}
\newcommand{\cD}{\mathcal{D}}
\newcommand{\cF}{\mathcal{F}}
\newcommand{\maP}{\mathfrak{P}}
\newcommand{\HHom}{\on{Hom}}
\newcommand{\Vecc}{\mathbf{vecfd}}
\newcommand{\Ivecc}{\mathbf{Vec}}
\newcommand{\End}{\on{End}}
\newcommand{\bk}{\Bbbk}
\newcommand{\unit}{\mathbb{1}}
\newcommand{\ot}{\otimes}
\newcommand{\FPdim}{\on{FPdim}}
\newcommand{\Sym}{\on{Sym}}
\newcommand{\iradc}{\mathtt{Rad}^{\mathcal{C}}}

\usetikzlibrary{backgrounds}
\usetikzlibrary{arrows}
\usetikzlibrary{shapes,shapes.geometric,shapes.misc}
\tikzstyle{tikzfig}=[baseline=-0.25em,scale=0.5]
\pgfkeys{/tikz/tikzit fill/.initial=0}
\pgfkeys{/tikz/tikzit draw/.initial=0}
\pgfkeys{/tikz/tikzit shape/.initial=0}
\pgfkeys{/tikz/tikzit category/.initial=0}
\pgfdeclarelayer{edgelayer}
\pgfdeclarelayer{nodelayer}
\pgfsetlayers{background,edgelayer,nodelayer,main}
\tikzstyle{none}=[inner sep=0mm]
\tikzstyle{every loop}=[]
\tikzstyle{whitedot}=[fill=white, draw, shape=circle, scale=0.3, tikzit draw=black, tikzit shape=circle, tikzit fill=white]
\tikzstyle{blackdot}=[fill=black, draw, shape=circle, scale=0.3, tikzit draw=black, tikzit shape=circle, tikzit fill=black]
\tikzstyle{box}=[fill=white, draw=black, shape=rectangle, tikzit fill=white]
\tikzstyle{BL}=[draw=black, shape=circle, fill=black, scale=0.3]
\tikzstyle{PP}=[draw={rgb,255:red,102;green,41;blue,163}, shape=circle, fill={rgb,255:red,102;green,41;blue,163}, scale=0.3]
\tikzstyle{morphism-edge}=[-, draw=black, thick]
\tikzstyle{cotensor}=[-, draw=gray]
\tikzstyle{braid-over}=[-, draw=white, thick, double=black, double distance=0.8pt, tikzit draw={rgb,255: red,128; green,0; blue,128}]
\tikzstyle{purple-over}=[-, draw=white, thick, double={rgb,255:red,102;green,41;blue,163}, double distance=0.8pt, tikzit draw={rgb,255:red,102;green,41;blue,163}]
\tikzstyle{purple}=[-, draw={rgb,255:red,102;green,41;blue,163}, thick]
\tikzstyle{blue-under}=[-, draw={rgb,255:red,0;green,128;blue,128}, thick]
\tikzstyle{ddd}=[-, draw=black, dash dot dot]
\tikzstyle{unit}=[-, draw=black, densely dotted]
\tikzstyle{Front}=[-, draw=black, fill={rgb,255 :red,255; green,255; blue,255}, opacity=0.8]
\tikzstyle{Hidden}=[-, draw=black, fill={rgb,255 :red,255; green,255; blue,255}, opacity=0.2]
\tikzstyle{directed}=[-, thick, black, decoration={markings, mark=at position 0.5 with {\arrow{>}}}, postaction=decorate]
\newcommand{\tikzfig}[1]{{%
    \tikzstyle{every picture}=[tikzfig]
    \IfFileExists{#1.tikz}
    {\input{#1.tikz}}
    {%
      \IfFileExists{figures/#1.tikz}
      {\input{figures/#1.tikz}}
      {\tikz[baseline=-0.5em]{\node[draw=red,font=\color{red},fill=red!10!white] {\textit{#1}};}}%
    }%
}}

\usepackage[a4paper, margin=3cm]{geometry}
\begin{document}

\title{Simple algebras and exact module categories}

\author{Kevin Coulembier}
\address{K.C., School of Mathematics and Statistics, University of Sydney, NSW 2006, Australia}
\email{kevin.coulembier@sydney.edu.au}

\author{Mateusz Stroiński}
\address{M.S., Department of Mathematics, Uppsala University, Box. 480, SE-75106, Uppsala, Sweden}
\email{mateusz.stroinski@math.uu.se}

\author{Tony Zorman}
\address{T.Z., TU Dresden, Institut für Geometrie, Zellescher Weg 12--14, 01062 Dresden, Germany}
\email{tony.zorman@tu-dresden.de}


\begin{abstract}
 We verify a conjecture of Etingof and Ostrik, stating that an algebra object in a finite tensor category is exact if and only if it is a finite direct product of simple algebras. Towards that end, we introduce an analogue of the Jacobson radical of an algebra object, similar to the Jacobson radical of a finite-dimensional algebra. We give applications of our main results in the context of incompressible finite symmetric tensor categories.
\end{abstract}

\maketitle

\vspace{-0.5cm}
\tableofcontents

\setlength{\parskip}{0.25em}

\addtocontents{toc}{\SkipTocEntry}
\subsection*{Acknowledgements}\label{sec:acknowledgements}
K.C.~is supported by
ARC grants DP210100251 and FT220100125.
T.Z.~is supported by \textsc{dfg} grant \textsc{kr} \oldstylenums{5036/2--1}.
M.S.~would like to thank Victor Ostrik for explaining~\cite[Conjecture~B.6]{EOf} to him. The authors would like to also thank Johannes Flake and Edmund Heng for their valuable comments.

\section{Introduction}

Let $\bk$ be an arbitrary field.
Finite tensor categories are abelian $\bk$-linear rigid monoidal categories satisfying finiteness assumptions, generalising fusion categories, see~\cite{EO} for an overview. They arise naturally from representation theory and logarithmic conformal field theory.

In the theory of module categories, the notion of exact algebras in finite tensor categories is central. It is known that any exact algebra is a product of simple algebras. Here we prove the converse, as conjectured in~\cite[Conjecture~B.6]{EOf} by Etingof and Ostrik, that every product of simple algebras is exact.  In the special case that the finite tensor category is the module category of a finite dimensional Hopf algebra, the statement was known, see~\cite[Proposition~B.7]{EOf}, since it can be seen to be equivalent to a result of Skryabin in~\cite[Theorem~3.5]{Sk}.

\addtocontents{toc}{\SkipTocEntry}
\subsection*{Mixed subfunctors}

To state our results in more detail, let $\cC$ be a finite tensor category over $\bk$, with category of projective objects $\cp$. We fix an algebra object $A$ in $\cC$. We consider the categories $\apf$ and $\pfa$ of free projective left and right $A$-modules in $\cC$. By slight abuse we can say that these categories have the \emph{same} set of objects as $\cp$.
Proving~\cite[Conjecture~B.6]{EOf} can, via results in~\cite{EOf, St2}, be reduced to the following result, see \autoref{RSlattice}:
\begin{quote}
  There exists a canonical bijection between the set of ideal objects in $A$ and the set of ideals in the category $\pfa$ that are stable under the left action of $\cp$.
\end{quote}

A seemingly new notion we introduce to establish this bijection is that of a \emph{mixed subfunctor}
of the presheaf $\cC(\blank,A)$ on $\cp$, see \autoref{mixedsubfunctor}. These are subfunctors that are simultaneously subfunctors for the interpretation of
$$\cC(\blank,A)\cong \ahom(A\otimes \blank,A)\cong \homa(\blank\otimes A,A)$$ as a presheaf on $\apf$ and on $\pfa$, see \autoref{mixedsubfunctor} for details. By formal manipulations we can see that mixed subfunctors are in bijection with $\cp$-stable ideals in $\pfa$, see \autoref{sec:mixedsubfunctorscpideals}.
We then give two proofs for the fact that mixed subfunctors are in bijection with ideal objects in $A$. One proof uses the idea from~\cite{EOf} to realise a finite tensor category via a pseudo-Hopf algebra. The other takes a more categorical approach, using the monoidal structure on the presheaf category $\presh$ on $\mathcal{C}$ given by \emph{Day convolution}, and establishing a correspondence between mixed subfunctors and representable ideal objects in the algebra object $\mathcal{C}(\blank,A)$, which in turn correspond to ideals in $A$.

Using mixed subfunctors as a bridge between ideals in $A$ and $\cp$-stable ideals in $\pfa$, we connect the latter kind of ideals with exactness of $\modca$ as a $\mathcal{C}$-module category, by expanding on~\cite[Lemma~12]{MM}, and its generalisations~\cite[Theorem~11]{KMMZ},~\cite[Theorem~4.17]{St2}. Namely, we find that the minimal projective presentation of an object of the form $Q \triangleright M$, for $M \in \modca$ and $Q \in \cp$, generates a $\cp$-stable ideal in $\pfa$. Minimality implies that this generator, and hence also the generated ideal, are contained in the Jacobson radical of $\pfa$. It also implies that this generator is zero if and only if $Q \triangleright M$ is projective.
This allows us to develop a theory of Jacobson radicals for algebra objects in finite tensor categories analogous to that for finite-dimensional algebras.

In particular, to an algebra object $A$ we associate an ideal $\iradc(A)$ in $A$, which, under \autoref{RSlattice} corresponds to the greatest nilpotent $\cp$-stable ideal in $\pfa$, and hence $\iradc(A)$ itself is the greatest nilpotent ideal object in $A$. We can now formulate our main result:

\begin{theorem*}[\autoref{thm:mainresult}]
 Let $A$ be an algebra object in a finite tensor category $\mathcal{C}$. The following are equivalent:
\begin{enumerate}[label=(\roman*)]
 \item $A$ is exact.
 \item $\iradc(A) = 0$.
 \item $A$ has no non-zero nilpotent ideals.
 \item $A$ is a finite direct product of simple algebras.
\end{enumerate}
\end{theorem*}

This is again analogous with the theory of finite-dimensional algebras, since $A$ is exact if and only if every $M \in \modca$ is \emph{$\mathcal{C}$-projective}, i.e.\ its internal hom functor $\hom{M,\blank}$ is exact, so exactness of $A$ is analogous, and, as we show, equivalent, to semisimplicity.
Finally, in \autoref{thm:semisimplequotient}, we show that the category of $A/\iradc(A)$-modules, viewed as a full subcategory of $\modca$, consists precisely of the subquotients of objects of the form $Q \triangleright L$, for $Q \in \cp$ and $L$ a semisimple object of $\modca$. In order to do this, we give a description of the submodule $\iradc(M) = M\iradc(A)$ of $M$ in \autoref{prop:DayRadicalModule}.

In the presence of a fibre functor, where $\modca$ is the category of modules over an $H$-module algebra, for a finite-dimensional Hopf algebra $H$, the properties of $\iradc(A)$, given by the largest $H$-stable ideal in the radical of $A$, were studied previously in~\cite{LMS}, and, in a more general Hopf-theoretic setting, in~\cite{Fi}.

\addtocontents{toc}{\SkipTocEntry}
\subsection*{Motivation and applications}

The base example of a finite tensor category is the category $\Vecc$ of finite dimensional vector spaces. In $\Vecc$ every simple algebra is a Frobenius algebra. It is known that this statement does not extend to general finite tensor categories, see~\cite[Example~5.24]{Sh}. In~\cite{Sh}, Shimizu introduced the notion of a quasi-Frobenius algebra (which are algebras Morita equivalent to Frobenius algebras), and conjectured that all simple algebras are quasi-Frobenius, see~\cite[Conjecture~5.25]{Sh}. As noted \textit{loc cit.}\ the conjecture is equivalent to~\cite[Conjecture~B.6]{EOf}, so we also prove Shimizu's conjecture.

If $\cC$ is braided and $A$ is an indecomposable commutative algebra in $A$, then the category of $A$-modules in $\cC$ is again a (finite) tensor category if and only if $A$ is exact. It thus follows from the result in this paper that any simple commutative algebra gives rise to another tensor category. This is of interest for instance in modular and symmetric cases:

In the theory of vertex operator algebras, for a vertex operator algebra extension $V\subset W$, it is so that $W$ becomes a commutative algebra in the braided monoidal representation category of $V$ (under suitable assumptions) such that the representation category of $W$ is obtained as the category of (local) modules over this algebra. Our result can thus be used to get broader results on when representation categories of vertex operator algebras are rigid. For instance, see~\cite[Theorem~5.21]{SY} and~\cite[Theorem~2.24]{CMSY} for results on these questions that can now be extended.

For symmetric (not necessarily finite) tensor categories, it is an interesting question in the ongoing classification of incompressible categories whether simple commutative (ind-)algebras are exact, see~\cite[Question~6.1.4]{CEO}. We thus answer this question in the affirmative for finite symmetric tensor categories. We give some more detailed applications of our results to finite symmetric tensor categories in \autoref{SecSymm}.

\addtocontents{toc}{\SkipTocEntry}
\subsection*{Structure} The remainder of the paper is organised as follows. In \autoref{sec:prelim}, we recall some preliminaries and introduce notation for finite tensor categories, their categories of projective objects, their presheaf categories, module categories and ideals therein. Further, we introduce the notion of a mixed subfunctor (\autoref{mixedsubfunctor}) and establish the notation for all the kinds of ideals we will consider. In \autoref{sec:1mix}, we establish a bijection (\autoref{mix1main}) between ideals in an algebra object $A$ in a finite tensor category $\mathcal{C}$, and mixed subfunctors of ${\mathcal{C}(\blank,A)}_{|\cp}$, using a realisation of $\mathcal{C}$ via a pseudo-Hopf algebra. We establish the same bijection using Day convolution and representable presheaves in \autoref{sec:2mix}. We then establish a bijection between mixed subfunctors of ${\mathcal{C}(\blank,A)}_{|\cp}$ and $\cp$-stable ideals in $\pfa$ in \autoref{sec:mixedsubfunctorscpideals}, resulting in a bijection between the latter kind of ideals, and ideals in $A$, \autoref{RSlattice}. In \autoref{sec:nilpotent}, we define and study products of various kinds of ideals, introduce the $\mathcal{C}$-module radical of an algebra object (\autoref{def:cmoduleradical}), and show that $A$ is exact if and only if $\iradc(A) = 0$ (\autoref{thm:semisimplecharacterization}), which implies~\cite[Conjecture~B.6]{EOf}.
In \autoref{sec:mainresult}, we first give a module-theoretic proof of the fact that an exact algebra is a finite direct product of simple algebras (\autoref{prop:splitalg}). We then use the results of preceding sections to state our main result, \autoref{thm:mainresult}, and give a short proof thereof.
In \autoref{sec:radmod2}, we define the radical of a module object (\autoref{def:MI}), using which we characterise $A/\iradc(A)$-modules (\autoref{thm:semisimplequotient}).
\autoref{sec:Skryabin} explains how our results specialize to Skryabin's result,~\cite[Theorem~3.5]{Sk}, in the presence of a fibre functor on $\mathcal{C}$.
The paper concludes with \autoref{SecSymm}, which contains the applications to finite symmetric tensor categories.

\section{Preliminaries}\label{sec:prelim}

\subsection{Finite tensor categories}

Let $\bk$ be a field. We refer to~\cite{EGNO} for the standard notions of (finite) tensor categories over $\bk$. However, contrary to~\cite{EGNO} we do not assume that the base field $\bk$ is algebraically closed. A $\bk$-linear monoidal category $(\cC,\otimes,\unit)$ (with $\bk$-bilinear tensor product) is a finite tensor category if
\begin{enumerate}
    \item $(\cC,\otimes,\unit)$ is rigid, meaning every $X\in \cC$ has a left dual $\ld{X}$ and a right dual $\rd{X}$ (denoted by $X^\ast$ and ${}^\ast X$ respectively in~\cite{EGNO}).
    Given an object $V \in \cC$, we denote the
    \emph{evaluation} and \emph{coevaluation} (sometimes referred to as \emph{counit} and \emph{unit}) morphisms as follows:
    \[
    \ev_{V}\colon \ld{V} \otimes V \to \unit, \; \coev_{V}\colon \unit \to V \otimes \ld{V}
    \]
    \item $\cC$ is equivalent to the category of finite dimensional modules of a finite dimensional associative $\bk$-algebra, meaning that $\cC$ has enough projective objects, finite dimensional morphism spaces, objects of finite length and only finitely many simple objects up to isomorphism.
    \item $\bk\to\End(\unit)$ is an isomorphism.
\end{enumerate}
Useful properties that follow from these assumptions is that the tensor product is exact and $\unit$ is a simple object. As a consequence of the latter, for any non-zero object $V\in \cC$, the coevaluation morphism $\coev_{{V}}$ is mono and the evaluation morphism $\ev_{{V}}$ is epic.

Throughout the paper, let $\mathcal{C}$ be a finite tensor category over $\bk$, and let $\cp$ denote the full subcategory of projective objects in $\mathcal{C}$. Further, throughout we let
$A=(A,\mu,\eta)$ be an algebra object in $\mathcal{C}$, let $\modca$ denote the category of right $A$-module objects, and let $\projca$ denote the full subcategory of projective objects in $\modca$. For $M,N \in \modca$, we write $\homa(M,N)$ for $\modca(M,N)$ and similarly the notation $\ahom$ refers to morphisms of left $A$-modules.

Recall that $\modca$ is a $\mathcal{C}$-module category, where, for $M \in \modca$ and $V \in \mathcal{C}$, the action of $V$ on $M$ is defined as the following $A$-module:
\begin{equation}
(V \triangleright M, \nabla_{V \triangleright M}) = (V \otimes M, V \otimes \nabla_{M}).
\end{equation}
We use the notation $V \triangleright M$, rather than the common alternative $V \otimes M$, both to make our choice of $A$-action unambiguous, but also to indicate that many of our definitions and results can be translated into the setting of abstract $\mathcal{C}$-module categories. For a left $A$-module $N$, we find a similarly defined left $A$-module $N \triangleleft V$.

\begin{definition}
 Let $\pfa$ denote the full subcategory of $\modca$ whose objects are those of the form $P \triangleright A$ for $P \in \cp$. The notation indicates that this is the Kleisli category for the monad $\blank \otimes A$ on $\cp$, viewed as a subcategory of $\modca$, which itself is the Eilenberg--Moore category for $\blank \otimes A$ viewed as a monad on $\mathcal{C}$. In particular, being a Kleisli category, the object set of $\pfa$ canonically identifies with the object set of $\cp$, on which the monad $\blank \otimes A$ operates.
\end{definition}

In \autoref{def:cpstableideal} below, we only consider the cases $\mathcal{M} = \projca$ and $\mathcal{M} = \pfa$.
\begin{definition}\label{def:cpstableideal}
    Let $\mathcal{M}$ be a $\mathcal{C}$-module subcategory of $\modca$. A \emph{$\cp$-stable ideal } $\csym{J}$ in $\mathcal{M}$ consists of a collection of subspaces
$
\csym{J}(M,N) \subseteq \homa(M,N)
$
such that for $f \in \homa(M,M')$ and $g\in \homa(M',M'')$, we have $g \circ f \in \csym{J}(M,M'')$ if $f \in \csym{J}(M,M')$ or $g \in \csym{J}(M',M'')$, and $f \in \csym{J}(M,M')$ implies $Q \triangleright f \in \csym{J}(Q \triangleright M, Q\triangleright M')$, for $Q \in \cp$.
\end{definition}

\begin{notation}
    In the sequel, we will use four different kinds of structure, between which we will establish a chain of bijections. In order to maintain clear distinction between these, we record the notation we will use for each here:
    \begin{enumerate}[label=(N\arabic*)]
        \item
        We write the \emph{mixed subfunctors} of $\mathcal{C}(\blank,A)\colon \cp^{\on{op}} \to \Ivecc$ (see \autoref{mixedsubfunctor}) in a monospaced font: $\mathtt{I,J}$, etc.
        \item
        We denote the \emph{ideals in the algebra $\mathcal{C}(\blank,A)$} (see \autoref{idealinDay}), in particular also the representable ideals in $\mathcal{C}(\blank,A)$, by upright Roman letters: $\euler{I,J}$, etc.
        \item\label{idealnotation}
        We denote the \emph{$\cp$-stable ideals in $\projca$}, as well as ideals in categories more generally, by calligraphic letters: $\csym{I,J}$, etc.
        \item
        As an exception to~\ref{idealnotation}, we write the \emph{$\cp$-stable ideals in $\pfa$} (see \autoref{def:cpstableideal}) in Fraktur: $\mathfrak{I,J}$, etc.
        For $P,P' \in \cp$, we write $\mathfrak{J}(P,Q)$ for $\mathfrak{J}(P\triangleright A, Q\triangleright A)$.
    \end{enumerate}
\end{notation}

\begin{lemma}
  Any object of $\,\projca$ is a direct summand of an object in $\pfa$. Thus, $\projca$ is the idempotent closure of $\,\pfa$ in $\modca$.
\end{lemma}

\begin{notation}\label{projectivenotation}
   We use symbols $P,Q,P',Q'$ and the like to denote the objects of $\cp$. We use symbols $R,S,R',S'$ and the like to denote the objects of $\projca$. When associating a projective object to a not necessarily projective object, we use the latter as a superscript for the former. Thus, for example, using the standard subscript notation for projective resolutions, we may write the projective resolution of an object $V \in \mathcal{C}$ as
\[\begin{tikzcd}
	{\cdots } & {P_{2}^{V}} & {P_{1}^{V}} & {P_{0}^{V}} & V
	\arrow[from=1-1, to=1-2]
	\arrow["{p_{2}^{V}}", from=1-2, to=1-3]
	\arrow["{p_{1}^{V}}", from=1-3, to=1-4]
	\arrow["{p_{0}^{V}}", two heads, from=1-4, to=1-5]
\end{tikzcd}\]
\end{notation}

The following is a well-known property of idempotent closures, see e.g.\ \cite[Corollary~10.15]{St1}.
\begin{lemma}\label{lem:CauchyCorr}
 Let $\csym{J}$ be a $\cp$-stable ideal in $\projca$. Then $\csym{J}_{|\pfa}$ defines a $\cp$-stable ideal in $\pfa$. Further, a $\cp$-stable ideal $\mathfrak{J}$ in $\pfa$ admits a unique extension $\mathfrak{J}^{\mathsf{c}}$ to a $\cp$-stable ideal in $\projca$, establishing a bijection between $\cp$-stable ideals in $\pfa$ and in $\projca$.
\end{lemma}

Let $\mathbf{Vec}$ denote the category of (not necessarily finite-dimensional) vector spaces over $\Bbbk$, and let $\mathbf{vecfd}$ denote its full subcategory of finite-dimensional vector spaces.
Recall that the category $\presh$ of $\Bbbk$-linear presheaves on $\mathcal{C}$ is monoidal under so-called \emph{Day convolution}. For $\Phi, \Psi \in \presh$, their tensor product is given by
\[
 (\Phi \oast \Psi)(X) = \int^{U,V \in \mathcal{C}} \mathcal{C}(X, U \otimes V) \kotimes \Phi(U) \kotimes \Psi(V).
\]
In particular, the Yoneda embedding $\yo\colon \mathcal{C} \to \presh$ is a fully faithful, left exact, strong monoidal $\Bbbk$-linear functor. We can characterise its image as follows:
\begin{lemma}
 A presheaf $\Phi \in \presh$ is representable if and only if it is left exact and $\Phi(V)$ is finite-dimensional for all $V \in \mathcal{C}$.
\end{lemma}
\begin{proof}
 The Ind-completion $\mathbf{Ind}(\mathcal{C})$ is a locally finitely presentable category, whose category of compact objects is precisely $\mathcal{C}$.
 Since a presheaf $\Phi \in \presh$ is left exact if and only if its category of elements is filtered,
 in that case it is of the form $\Phi \cong \mathrm{fcolim}_i \cat{C}(\blank, V_{i})$.
 Thus, the functor
 $\mathbf{Ind}(\mathcal{C})(\blank, \Phi)\colon \mathbf{Ind}(\cat{C}) \to \kVect$ factors through the Yoneda embedding, and there is an isomorphism $\mathbf{Ind}(\cat{C})(\blank, \Phi)|_{\yo(\cat{C})} \cong \Phi$.

 If $W \in \mathcal{C}$ then $\mathbf{Ind}(\mathcal{C})(V,W) \cong \mathcal{C}(V,W)$ is finite-dimensional. On the other hand, if $\,\mathcal{C}(P,W)$ is finite-dimensional for all $P \in \cp$ then $W$ is of finite length, and so $W \in \mathcal{C}$.
\end{proof}

Since $\yo$ is a monoidal functor, it sends algebra and (bi)module objects in $\mathcal{C}$ to algebra and (bi)module objects in $\presh$ endowed with Day convolution. In particular, for an algebra object $A$ in $\mathcal{C}$, the multiplication on $\mathcal{C}(\blank,A)$ is given by the map
\begin{equation}\label{DayMultiplication}
 \mu_{\mathcal{C}(\blank,A)}\colon \int^{V,W \in \mathcal{C}} \mathcal{C}(\blank, V \otimes W) \kotimes \mathcal{C}(V,A) \kotimes \mathcal{C}(W,A) \to \mathcal{C}(\blank,A)
\end{equation}
corresponding to the collection
\begin{equation}\label{DayMultiply}
\begin{aligned}
 \mathcal{C}(X, V \otimes W) \kotimes \mathcal{C}(V,A) \kotimes \mathcal{C}(W,A) \to \mathcal{C}(X,A),\\
 \varphi \otimes g \otimes h \mapsto \mu_{A} \circ (g \otimes h) \circ \varphi
\end{aligned}
\end{equation}
dinatural in $V,W$, under the universal property of the coend.
Henceforth, $\mathcal{C}(\blank,A)$ will always be considered as an algebra object in the presheaf category $\presh$.

\begin{corollary}
 Let $I$ be an ideal in $A$, then the presheaf $\mathcal{C}(\blank,I)$ is an ideal in the algebra $\mathcal{C}(\blank,A)$.
\end{corollary}

\begin{corollary}\label{idealinDay}
 A subpresheaf $\euler{I}$ of $\mathcal{C}(\blank,A)$ is an ideal in $\mathcal{C}(\blank,A)$ if for all $U,V,W \in \mathcal{C}$, all $\varphi \in \mathcal{C}(U,V\otimes W)$ and morphisms $f \in \mathcal{C}(V,A), g \in \mathcal{C}(W,A)$, we have $\mu_{A} \circ (f \otimes g) \circ \varphi \in \euler{I}(U)$ provided that either $f \in \euler{I}(V)$ or $g \in \euler{I}(W)$.
\end{corollary}

\begin{proof}
    Being an ideal in the algebra object $\mathcal{C}(\blank,A)$ implies that $\euler{I}$ is a left $\mathcal{C}(\blank,A)$-submodule of $\mathcal{C}(\blank,A)$, i.e.\ that for every $U \in \mathcal{C}$, the image of the map
 \begin{equation}\label{leftsubmodule}
 \begin{aligned}
 &\int^{V,W \in \mathcal{C}} \mathcal{C}(U, V \otimes W) \kotimes \mathcal{C}(V,A) \kotimes \euler{I}(W) \\
 &\to \int^{V,W \in \mathcal{C}} \mathcal{C}(U, V \otimes W) \kotimes \mathcal{C}(V,A) \kotimes \mathcal{C}(W,A) \xrightarrow{{(\mu_{\mathcal{C}(\blank,A)})}_{U}} \mathcal{C}(U,A)
 \end{aligned}
 \end{equation}
 is contained in $\euler{I}(U)$. Secondly, $\euler{I}$ is a right $\mathcal{C}(\blank,A)$-submodule of $\mathcal{C}(\blank,A)$, so that also the image of the map
 \begin{equation}\label{rightsubmodule}
 \begin{aligned}
 &\int^{V,W \in \mathcal{C}} \mathcal{C}(U, V \otimes W) \kotimes  \euler{I}(V) \kotimes \mathcal{C}(W,A) \\
 &\to \int^{V,W \in \mathcal{C}} \mathcal{C}(U, V \otimes W) \kotimes \mathcal{C}(V,A) \kotimes \mathcal{C}(W,A) \xrightarrow{{(\mu_{\mathcal{C}(\blank,A)})}_{U}} \mathcal{C}(U,A)
 \end{aligned}
 \end{equation}
 is contained in $\euler{I}(U)$. Using the description of the multiplication on $\mathcal{C}(\blank,A)$ following \autoref{DayMultiplication}, we can reformulate these conditions as stating that
 \begin{equation}\label{idealcondition}
 \mu_{A} \circ (g \otimes h) \circ \varphi \in \euler{I}(U),
 \end{equation}
 for any $\varphi\colon U \to V \otimes W, g\colon V \to A$ and $h\colon W \to A$, whenever $g$ or $h$ lie in $\euler{I}(V)$ or $\euler{I}(W)$, respectively.
\end{proof}

We will denote the restriction of a presheaf $\Phi$ on a category $\mathcal{A}$ to a subcategory $\mathcal{B}$ of $\mathcal{A}$ by $\Phi_{|\mathcal{B}}$. Our cases of interest will be ${\mathcal{C}(\blank,A)}_{|\cp}$ in \autoref{mixedsubsection} below, and ${\homa(\blank,M)}_{|\projca}$, for $M \in \modca$, in \autoref{sec:radmod2}.

\subsection{Mixed subfunctors}\label{mixedsubsection}

Let $M \in \mathcal{C}$, let $N \in \modca$ and let $N'$ be a left $A$-module in $\cC$. Given $f \in \mathcal{C}(M,N)$ and $h \in \mathcal{C}(M,N')$, let $f^{A} = \nabla_{N} \circ (f\otimes A) \in \homa(M \triangleright A, N)$ and we define ${}^{A}h \in \ahom(A\triangleleft M,N')$ similarly.

Given $g \in \homa(M\triangleright A,N)$, let $g^{1} = g \circ (M \otimes \eta) \in \mathcal{C}(M,N)$ and given $h' \in \ahom(A \triangleleft M,N')$, define ${}^{1}h' \in \mathcal{C}(M,N')$ similarly.
In particular ${(f^{A})}^{1} = f$ and ${(g^{1})}^{A} = g$ and similarly for the left module structures, yielding familiar isomorphisms
\begin{equation}\label{indres}
\mathcal{C}(M,N) \cong \homa(M \triangleright A, N)\text{ and }\mathcal{C}(M,N') \cong \ahom(A \triangleleft M,N').
\end{equation}

\begin{definition}\label{mixedsubfunctor}
  A \emph{mixed subfunctor of $\,{\mathcal{C}(\blank,A)}_{|\cp}$} consists of, for $P \in \cp$, a collection of subspaces $\mathtt{J}(P) \subseteq \mathcal{C}(P,A)$, that satisfy the following conditions:
  \begin{enumerate}[label=(M\arabic*)]
  \item\label{mixed1} $\mathtt{J}$ is a subpresheaf of ${\mathcal{C}(\blank,A)}_{|\cp}$: for any $P,P' \in \cp$ and any $g \in \mathtt{J}(P)$ and $f \in \mathcal{C}(P', P)$, the morphism $g\circ f \in \mathcal{C}(P',A)$ lies in $\mathtt{J}(P')$.
  \item\label{mixed2} The collections $\{\,g^{A} \in \homa(P \triangleright A, A) \mid g \in \mathtt{J}(P)\,\}$ define a subpresheaf of the presheaf ${\homa(\blank,A)}_{|\pfa}$: for $g \in \mathtt{J}(P)$ and $f \in \homa(Q \triangleright A, P \triangleright A)$, the morphism $g^{A} \circ f$ is of the form $h^{A}$ for $h \in \mathtt{J}(Q)$. Equivalently, ${(g^{A} \circ f)}^{1} \in \mathtt{J}(Q)$.
  \item\label{mixed3} The collections $\{\,{}^{A}g \in \ahom(A \triangleleft P, A) \mid g \in \mathtt{J}(P)\,\}$ define a subpresheaf of the presheaf ${\ahom(\blank,A)}_{|\apf}$: for $g \in \mathtt{J}(P)$ and $f \in \ahom(A \triangleleft Q, A \triangleleft P)$, the morphism ${}^{A}g \circ f$ is of the form ${}^{A}h$ for $h \in \mathtt{J}(Q)$. Equivalently, ${}^{1}({}^{A}g \circ f) \in \mathtt{J}(Q)$.
  \end{enumerate}
\end{definition}

Note that that condition~\ref{mixed1} is implied by either~\ref{mixed2} or~\ref{mixed3}.

\section{Algebra ideals and mixed subfunctors, first approach}\label{sec:1mix}

Fix a projective generator $\maP$ of $\cC$.
Since the category of presheaves on $\pfa$ is equivalent to the category of right modules over $\End_{\blank A}(\maP\triangleright A)$, and similarly for right $\End_{A\blank}(A\triangleleft\maP)$-modules, we have an immediate bijection between mixed subfunctors and subspaces of
$$\HHom(\maP,A)\cong\homa(\maP\triangleright A,A)\cong \ahom(A\triangleleft \maP,A)$$ that are submodules for both of these algebras. We will derive explicit descriptions of the actions of both algebras in \autoref{SecSU} and then use them to link mixed subfunctors to algebra ideals.

\subsection{The set-up}\label{SecSU}

\subsubsection{The pseudo-Hopf algebra}
Consider the faithful exact functor
$$\omega \defeq \cC(\maP,-)\colon\cC\to\Vecc.$$ We thus obtain an equivalence between $\cC$ and the category of finite dimensional (over~$\bk$) left $H$-modules, for
$$H\defeq\End(\omega)={\End(\maP)}^{\op}.$$

Under this equivalence, the projective generator $\maP$ of $\cC$ is sent to the left regular $H$-module and $\unit$ is sent to some simple $H$-module $L$. For the rest of this section, we will \emph{redefine} $\cC$ as the category of $H$-modules, justifying notation as $\unit=L$. Furthermore, $\omega$ thus becomes the functor forgetting the $H$-action, which we often leave out of notation.

As explained in~\cite[\S~4]{EOf}, the tensor category structure on $\cC$ equips $H$ with the structure of a `pseudo-Hopf algebra'. We do not need to review all features, but the tensor product on $\cC$ is encoded in an $(H,H\otimes_{\bk}H)$-bimodule $T$, so that
$$X\ot Y\;=\; T\otimes_{H\otimes_{\bk} H}(X\otimes_{\bk}Y)$$
for objects ($H$-modules) $X,Y\in\cC$. Here, and for the rest of this section, in order to avoid confusion, the undecorated symbols $\otimes, \triangleright$ and $\triangleleft$
refers to the tensor product inside $\cC$ and the actions of $\cC$ on left and right $A$-modules, while for tensor products over $\bk$-algebras $R$ we will always use $\otimes_R$.
We also have natural isomorphisms
\begin{equation}\label{unittransfo}
T\otimes_{H\otimes_{\bk} H}(L\otimes_{\bk}-)\;\cong\;\id\;\cong\;T\otimes_{H\otimes_{\bk} H}(-\otimes_{\bk}L),
\end{equation}
where $\id$ is the identity functor on $\cC$.

\subsubsection{Action maps}
Regarding $A$ as an object in $\cC$ (\textit{i.e.\ }an $H$-module) gives an action map
$$\rho\colon H\otimes_{\bk} A\to A,$$
which is again a morphism of $H$-modules (for $A$ regarded as a vector space in the source).
We can also write the multiplication map as
$$\mu\colon T\otimes_{H\otimes_{\bk} H}(A\otimes_{\bk}A)\to A,$$
and the unit as $\eta\colon L\to A$.

We have isomorphisms of vector spaces
$$\End_{A\blank}(A\triangleleft \maP)\cong \cC(\maP,A\ot \maP)=\omega(A\ot\maP)\cong T\otimes_{H\otimes_{\bk} H}(A\otimes_{\bk}H)\cong T\otimes_{H\blank}A.$$
Here, we write subscripts like $H\blank$ to indicate that we take the tensor product over $H\otimes_{\bk}\bk\subset H\otimes_{\bk}H$.
This thus equips $T\otimes_{H\blank}A$ with the structure of a $\bk$-algebra, which we will not need to describe explicitly. By construction, it boasts a right module structure on
$$\ahom(A\triangleleft\maP,A)\cong\cC(\maP,A)=\omega(A).$$
It will be more convenient to write the algebra on the left in the corresponding action map, \textit{i.e.\ }as
$$(T\otimes_{H\blank}A)\otimes_{\bk}\omega(A)\;\to\;\omega(A).$$

By direct verification, up to composition with an isomorphism, this action map is
\begin{equation}\label{action1}
T\otimes_{H\otimes_{\bk} H}(A\otimes_{\bk} (H\otimes_{\bk}\omega(A)))\xrightarrow{T_{H\otimes_{\bk} H}(A\otimes_{\bk} \rho)} T\otimes_{H\otimes_{\bk} H}(A\otimes_{\bk} A)\xrightarrow{\mu}\omega(A).
\end{equation}
We included all $\omega$'s in the formula that make sense, to stress that $H$ does not act on the second copy of $A$ in the source, and to indicate that $\omega(A)$ is the module on which we act.

Similarly, the action of
$$\End_{\blank A}(\maP\triangleright A)\;\cong\; T\otimes_{H\otimes_{\bk} H}(H\otimes_{\bk}A)$$
on $\homa(\maP\ot A,A)\cong\omega(A)$
is captured by the action map
\begin{equation}\label{action2}T\otimes_{H\otimes_{\bk} H}( (H\otimes_{\bk}\omega(A))\otimes_{\bk}A)\xrightarrow{T_{H\otimes_{\bk} H}(\rho\otimes_{\bk} A)} T\otimes_{H\otimes_{\bk} H}(A\otimes_{\bk} A)\xrightarrow{\mu}\omega(A).\end{equation}

\subsection{The bijection}

\begin{proposition}\label{PropId}
A subspace
$$V\subset\omega(A)=\cC(\maP,A)$$
corresponds to a two-sided ideal in the algebra $A\in\cC$ if and only if it is a submodule for both the $\End_{\blank A}(\maP\triangleright A)$-module and the $\End_{A\blank}(A\triangleleft\maP)$-module structure on $\cC(\maP,A)$.
\end{proposition}
\begin{proof}
Assume first that $V\subset\omega(A)$ represents a two-sided ideal in $A$. That means that it is an $H$-submodule, or in other words that $\rho$ restricts to
$$H\otimes_{\bk}V\to V$$
and similarly that $\mu$ restricts to
$$T\otimes_{H\otimes H}(A\otimes_{\bk}V)\to V\quad\mbox{and}\quad T\otimes_{H\otimes H}(V\otimes_{\bk}A)\to V.$$
Hence $V\subset \omega(A)$ is stable under actions~\eqref{action1} and~\eqref{action2}, making it a submodule for both actions.

Conversely, assume that a subspace $V\subset\omega(A)$ is stable for both action maps~\eqref{action1} and~\eqref{action2}. Using either action, the unit of $A$ and~\eqref{unittransfo}, shows that $V$ is an $H$-submodule. Since $\rho$ is an epimorphism of $H$-modules and the tensor product is exact, we find a commutative diagram
\[
  \begin{tikzcd}[ampersand replacement=\&]
    {T \otimes_{H \otimes H} (A \kotimes (H \kotimes V))} \&\& {T \otimes_{H \otimes H} (A \kotimes V)} \\
    \&\& {\omega(A)}
    \arrow["{T \otimes_{H \otimes H}(A \otimes \rho)}", two heads, from=1-1, to=1-3]
    \arrow[from=1-1, to=2-3]
    \arrow["\mu", from=1-3, to=2-3]
  \end{tikzcd}
\]
where the diagonal is the action~\eqref{action1} restricted to $V\subset\omega(A)$. By assumption this action takes values in $V\subset\omega(A)$. Since the horizontal arrow is surjective, also $\mu$ must take values in $V$ implying that $V$ is a left ideal. Using~\eqref{action2} shows that $V$ is a right ideal too.
\end{proof}

Passing throught the equivalence between mixed subfunctors and `double' submodules of $\cC(\maP,A)$ then yields the following consequence of \autoref{PropId}.

\begin{corollary}\label{mix1main}
    The assignment that sends an ideal $I<A$ to the presheaf
    $$\cp^{\op}\to\Vecc,\quad P\mapsto \cC(P,I),$$
    yields a bijection between ideals in $A$ and mixed subfunctors of $\,\cC(\blank,A)|_{\cp}$.
\end{corollary}

\section{Algebra ideals and mixed subfunctors, second approach}\label{sec:2mix}
\subsection{Reduction of presheaves to projective objects}

\begin{lemma}\label{lexinj}
 A left exact presheaf $\Phi \in \presh$ is uniquely determined by its restriction to the subcategory of projective objects in $\mathcal{C}$, i.e.\ by
 \[
  \cp^{\on{op}} \hookrightarrow \mathcal{C}^{\on{op}} \xrightarrow{\Phi} \mathbf{Vec}_{\Bbbk}.
 \]
 Furthermore, this restriction yields an equivalence $\mathbf{Lex}(\mathcal{C}^{\on{op}},\mathbf{Vec}_{\Bbbk}) \simeq \mathbf{Fun}_{\Bbbk}(\cp^{\on{op}},\mathbf{Vec}_{\Bbbk})$, which also restricts to an equivalence $\mathbf{Lex}(\mathcal{C}^{\on{op}},\mathbf{vecfd}_{\Bbbk}) \simeq \mathbf{Fun}_{\Bbbk}(\cp^{\on{op}},\mathbf{vecfd}_{\Bbbk})$.
\end{lemma}

\begin{proof}
 Since $\mathcal{C}$ has enough projectives, $\mathcal{C}^{\on{op}}$ has enough injectives, which are given precisely by $\cp^{\on{op}}$. The claim follows by a standard result about abelian categories with enough injectives.
\end{proof}

Before proving the result of this section, we need to recall a slight variant of a result of~\cite{KL}:
\begin{proposition}[{\cite[Proposition~5.1.7]{KL}}]\label{KerlerLyubashenko}
 Let $\mathcal{C}$ be an abelian $\Bbbk$-linear category and $\mathcal{P}$ a collection of objects in $\mathcal{C}$ such that for any $X \in \mathcal{C}$, there is an epimorphism $\bigoplus_{i=1}^{n} P_{i} \twoheadrightarrow X$, for some $n$ and $P_{i} \in \mathcal{P}$. Let $\Phi\colon \mathcal{C}^{\on{op}} \kotimes \mathcal{C} \to \mathbf{Vec}$ be a functor such that $\Phi(-,Y)$ is left exact and $\Phi(X,-)$ is exact, for all $X,Y \in \mathcal{C}$. The morphism
 \[
 \int^{P \in \mathcal{P}} \Phi(P,P) \to \int^{X \in \mathcal{C}} \Phi(X,X)
 \]
 induced by the inclusion $\setj{\Phi(P,P)}_{P \in \mathcal{P}} \subseteq \setj{\Phi(X,X)}_{X \in \mathcal{C}}$, is an isomorphism.
\end{proposition}

\begin{proof}
 While the statement of~\cite[Proposition~5.1.7]{KL} assumes also $\Phi(-,Y)$ to be exact, the assumption about the right exactness of $\Phi(-,Y)$ is not used in the proof given therein.
\end{proof}

\begin{remark}
 In our application of \autoref{KerlerLyubashenko} in the proof of \autoref{DayCollection} below, we could instead also directly invoke~\cite[Remark~5.1.10]{KL}.
\end{remark}

\begin{proposition}\label{DayCollection}
 A representable ideal $\euler{I}$ in $\mathcal{C}(\blank,A)$ is determined by the collection of subspaces $\euler{I}(P) \subseteq \mathcal{C}(P,A)$, for $P \in \cp$.

 Further, a collection $\{\,\euler{I}(P) \subseteq \mathcal{C}(P,A) \mid P \in \cp\,\}$ is the restriction of a representable ideal in $\mathcal{C}(\blank,A)$ if and only if the following conditions are satisfied:
 \begin{enumerate}[label=(D\arabic*)]
 \item\label{Day1}
     $\euler{I}$ is a subpresheaf of $\,{\mathcal{C}(\blank,A)}_{|\cp}$: for any $P,P' \in \cp$ and any $g \in \euler{I}(P)$ and $f \in \mathcal{C}(P', P)$, the morphism $g\circ f \in \mathcal{C}(P',A)$ lies in $\euler{I}(P')$.
 \item\label{Day2}
    $\euler{I}$ is the restriction of a left $\mathcal{C}(\blank,A)$-submodule object of $\,\mathcal{C}(\blank,A)$: for any $Q,P,P' \in \cp$, any $\varphi \in \mathcal{C}(Q,P\otimes P')$, any $g \in \mathcal{C}(P,A)$ and $h \in \euler{I}(P')$, the morphism $\mu_{A} \circ (g \otimes h) \circ \varphi$ lies in $\euler{I}(Q)$.
 \item\label{Day3}
   $\euler{I}$ is the restriction a right $\mathcal{C}(\blank,A)$-submodule object of $\,\mathcal{C}(\blank,A)$: for any $Q,P,P' \in \cp$, any $\varphi \in \mathcal{C}(Q,P\otimes P')$, any $g \in \euler{I}(P)$ and $h \in \mathcal{C}(P',A)$, the morphism $\mu_{A} \circ (g \otimes h) \circ \varphi$ lies in $\euler{I}(Q)$.
 \end{enumerate}
\end{proposition}

\begin{proof}
 By definition, an ideal $\mathrm{I}$ in $\mathcal{C}(\blank,A)$ consists of subspaces $\mathrm{I}(V) \subseteq \mathcal{C}(V,A)$ which give a subpresheaf, thus in particular the collection $\setj{\mathcal{C}(P,A)}_{P \in \cp}$ satisfies condition~\ref{Day1}. Further, this collection determines $\mathrm{I}$, by \autoref{lexinj}.
 Since representable presheaves are closed under Day convolution, we find that $\mathrm{I} \oast \mathcal{C}(\blank,A)$ and $\mathcal{C}(\blank,A) \oast \mathrm{I}$ are both left exact, so we may assume $U$ to be projective in \autoref{idealcondition}. This implies that the profunctors whose coends we consider in \autoref{leftsubmodule} and \autoref{rightsubmodule} above are left exact in the contravariant variable and exact in the covariant variables, so by \autoref{KerlerLyubashenko}, we may also assume $V,W$ to be projective.
\end{proof}

\subsection{Mixed subfunctors and representable presheaf ideals}\label{mixidealcorr}

\begin{lemma}\label{idealtomixed}
 A collection of subspaces $\euler{J}(P) \subseteq \mathcal{C}(P,A)$, for $P \in \cp$, determining a representable ideal in $\mathcal{C}(\blank,A)$ is a mixed subfunctor.
\end{lemma}

\begin{proof}
 By definition, $\euler{J}$ is a subpresheaf of ${\mathcal{C}(\blank,A)}_{|\cp}$, and so condition~\ref{Day1} and condition~\ref{mixed1} are identical.
 Hence, it suffices to verify conditions~\ref{mixed2} and~\ref{mixed3}.

 To verify condition~\ref{mixed2}, let $g \in \euler{J}(P)$ and $f \in \homa(Q \triangleright A, P \triangleright A)$. We have
 \[
 {(g^{A} \circ f)}^{1} = \mu_{A} \circ (g \otimes A) \circ (f \circ Q \otimes \eta_{A}),
 \]
 which lies in $\euler{J}(Q)$ by applying condition~\ref{Day2} to morphisms $g,h$, and $\varphi$, by setting $h = \on{id}_{A}$ and $\varphi = f\circ (Q \otimes \eta_{A})$.
 A string-diagrammatic reformulation of this is given in \autoref{fig:fig1}.
 \begin{figure}[htbp]
   \begin{tikzpicture}
	\begin{pgfonlayer}{nodelayer}
		\node [style=none] (0) at (-10.5, -1) {};
		\node [style=none] (2) at (-10.5, 1) {};
		\node [style=box] (3) at (-10.5, 0) {\scriptsize $h$};
		\node [style=none] (4) at (-9.25, 0) {\scriptsize $\in \cat{J}$};
		\node [style=none] (5) at (-8, 0.25) {\scriptsize$\overset{\ref{Day2}}{\implies}$};
		\node [style=none] (6) at (-5.75, -2.25) {};
		\node [style=box] (7) at (-5.75, -1.25) {\scriptsize$\varphi$};
		\node [style=box] (8) at (-6.75, 0) {\scriptsize$g$};
		\node [style=box] (9) at (-4.75, 0) {\scriptsize$h$};
		\node [style=none] (10) at (-5.75, 1.25) {};
		\node [style=none] (11) at (-5.75, 2) {};
		\node [style=none] (12) at (-2.75, 0) {\scriptsize $\in \cat{J},\ \forall \varphi,g$};
		\node [style=none] (13) at (1.5, -1) {};
		\node [style=none] (14) at (1.5, 1) {};
		\node [style=box] (15) at (1.5, 0) {\scriptsize $g$};
		\node [style=none] (16) at (2.75, 0) {\scriptsize $\in \cat{J}$};
		\node [style=none] (17) at (4, 0.25) {\scriptsize$\overset{\ref{mixed2}}{\implies}$};
		\node [style=none] (18) at (5.25, -2) {};
		\node [style=blackdot] (19) at (5.75, -1.5) {};
		\node [style=box] (20) at (5.5, -0.5) {\scriptsize$f$};
		\node [style=box] (21) at (5.5, 1) {\scriptsize$g^A$};
		\node [style=none] (22) at (5.25, -0.75) {};
		\node [style=none] (23) at (5.75, -0.75) {};
		\node [style=none] (24) at (5.75, -0.25) {};
		\node [style=none] (25) at (5.25, -0.25) {};
		\node [style=none] (26) at (5.25, 0.75) {};
		\node [style=none] (27) at (5.75, 0.75) {};
		\node [style=none] (28) at (5.5, 2.5) {};
		\node [style=none] (29) at (7.75, -2) {};
		\node [style=blackdot] (30) at (8.5, -1.5) {};
		\node [style=box] (31) at (8.125, -0.5) {\scriptsize$\hspace{0.2em}f\hspace{0.2em}$};
		\node [style=none] (33) at (7.75, -0.75) {};
		\node [style=none] (34) at (8.5, -0.75) {};
		\node [style=none] (35) at (8.5, -0.25) {};
		\node [style=none] (36) at (7.75, -0.25) {};
		\node [style=none] (38) at (8.5, 0.75) {};
		\node [style=none] (39) at (8.125, 2.5) {};
		\node [style=none] (40) at (6.75, 0) {\scriptsize$=$};
		\node [style=box] (41) at (7.75, 0.75) {\scriptsize$g$};
		\node [style=none] (42) at (8.125, 1.75) {};
		\node [style=none] (43) at (9.75, 0) {\scriptsize $\in \cat{J}$};
		\node [style=none] (44) at (-0.25, 0) {\scriptsize{and}};
		\node [style=none] (45) at (11.5, 0) {\scriptsize$\implies$};
		\node [style=none] (46) at (12.75, -2.25) {};
		\node [style=blackdot] (47) at (14, -1.75) {};
		\node [style=box] (48) at (13.375, -0.75) {\scriptsize$\hspace{0.75em}f\hspace{0.75em}$};
		\node [style=none] (49) at (12.75, -1) {};
		\node [style=none] (50) at (14, -1) {};
		\node [style=none] (51) at (14, -0.5) {};
		\node [style=none] (52) at (12.75, -0.5) {};
		\node [style=box] (53) at (14, 0.5) {\scriptsize$\mathrm{id}$};
		\node [style=none] (54) at (13.375, 2.25) {};
		\node [style=box] (55) at (12.75, 0.5) {\scriptsize$g$};
		\node [style=none] (56) at (13.375, 1.5) {};
		\node [style=none] (57) at (15.5, 0) {\scriptsize $\in \cat{J}$};
	\end{pgfonlayer}
	\begin{pgfonlayer}{edgelayer}
		\draw [style=morphism-edge] (0.center) to (3);
		\draw [style=morphism-edge] (3) to (2.center);
		\draw [style=morphism-edge] (6.center) to (7);
		\draw [style=morphism-edge, in=270, out=0] (7) to (9);
		\draw [style=morphism-edge, in=270, out=180] (7) to (8);
		\draw [style=morphism-edge, in=180, out=90] (8) to (10.center);
		\draw [style=morphism-edge, in=90, out=0] (10.center) to (9);
		\draw [style=morphism-edge] (10.center) to (11.center);
		\draw [style=morphism-edge] (13.center) to (15);
		\draw [style=morphism-edge] (15) to (14.center);
		\draw [style=morphism-edge] (18.center) to (22.center);
		\draw [style=morphism-edge] (19) to (23.center);
		\draw [style=morphism-edge] (24.center) to (27.center);
		\draw [style=morphism-edge] (25.center) to (26.center);
		\draw [style=morphism-edge] (21) to (28.center);
		\draw [style=morphism-edge] (29.center) to (33.center);
		\draw [style=morphism-edge] (30) to (34.center);
		\draw [style=morphism-edge] (35.center) to (38.center);
		\draw [style=morphism-edge] (36.center) to (41);
		\draw [style=morphism-edge, in=180, out=90] (41) to (42.center);
		\draw [style=morphism-edge, in=90, out=0] (42.center) to (38.center);
		\draw [style=morphism-edge] (42.center) to (39.center);
		\draw [style=morphism-edge] (46.center) to (49.center);
		\draw [style=morphism-edge] (47) to (50.center);
		\draw [style=morphism-edge] (51.center) to (53);
		\draw [style=morphism-edge] (52.center) to (55);
		\draw [style=morphism-edge, in=180, out=90] (55) to (56.center);
		\draw [style=morphism-edge, in=90, out=0] (56.center) to (53);
		\draw [style=morphism-edge] (56.center) to (54.center);
	\end{pgfonlayer}
\end{tikzpicture}
   \caption{Verification of condition~\ref{mixed2}.}%
   \label{fig:fig1}
 \end{figure}

 To verify condition~\ref{mixed3}, let $g \in \euler{J}(P)$ and $f \in \ahom(A \triangleleft Q, A \triangleleft P)$. We have
 \[
 ({}^{A}g \circ f) = \mu_{A} \circ (\on{id}_{A} \otimes g) \circ (f \circ \eta_{A} \otimes Q),
 \]
 which lies in $\euler{J}(Q)$ by applying condition~\ref{Day3} to the morphisms $\varphi,g'$, and $h'$, by setting $h' = g$, $g' = \on{id}_{A}$, and $\varphi = f\circ (\eta_{A} \otimes Q)$.
\end{proof}

\begin{lemma}
  A mixed subfunctor $\mathtt{J}$ gives a collection of subspaces $\mathtt{J}(P) \subseteq \mathcal{C}(P,A)$, for $P \in \cp$ determining a representable ideal in $\mathcal{C}(\blank,A)$
\end{lemma}

\begin{proof}
 Since condition~\ref{mixed1} and condition~\ref{Day1} are identical, it suffices to verify~\ref{Day2} and~\ref{Day3}.

 Let $\varphi \in \mathcal{C}(Q, P \otimes P'), \; g \in \mathcal{C}(P,A)$ and $h \in \mathtt{J}(P')$.
 Then condition~\ref{Day2} follows from \autoref{fig:fig2},
 where ${}^{A}h \circ ({}^{A}g \otimes P') \in \mathtt{J}$ by condition~\ref{mixed2} and further ${}^{A}h \circ ({}^{A}g \otimes P') \circ (A \otimes \varphi) \circ (\eta_{A} \otimes Q) \in \mathtt{J}$ by applying condition~\ref{mixed1} to this latter morphism.
 \begin{figure}[htbp]
   \begin{tikzpicture}
	\begin{pgfonlayer}{nodelayer}
		\node [style=none] (0) at (-8.25, -2.5) {};
		\node [style=box] (1) at (-8.25, -1.5) {\scriptsize$\hspace{0.5em}\varphi\hspace{0.5em}$};
		\node [style=box] (2) at (-9, 0) {\scriptsize$g$};
		\node [style=box] (3) at (-7.5, 0) {\scriptsize$h$};
		\node [style=none] (4) at (-8.25, 1.25) {};
		\node [style=none] (5) at (-8.25, 2) {};
		\node [style=none] (6) at (-8.75, -1.25) {};
		\node [style=none] (7) at (-7.75, -1.25) {};
		\node [style=none] (8) at (-4.5, -2.75) {};
		\node [style=box] (9) at (-4.5, -1.75) {\scriptsize$\hspace{0.5em}\varphi\hspace{0.5em}$};
		\node [style=box] (10) at (-5.25, 0) {\scriptsize${}^A g$};
		\node [style=box] (11) at (-3.75, 0) {\scriptsize${}^A h$};
		\node [style=none] (12) at (-4.5, 1.25) {};
		\node [style=none] (13) at (-4.5, 2) {};
		\node [style=none] (14) at (-5, -1.5) {};
		\node [style=none] (15) at (-4, -1.5) {};
		\node [style=none] (16) at (-5, -0.25) {};
		\node [style=none] (17) at (-5.5, -0.25) {};
		\node [style=none] (18) at (-5, -0.25) {};
		\node [style=blackdot] (19) at (-5.5, -1) {};
		\node [style=none] (20) at (-4, -0.25) {};
		\node [style=none] (21) at (-3.5, -0.25) {};
		\node [style=blackdot] (22) at (-4, -0.875) {};
		\node [style=none] (23) at (-6.5, 0) {\scriptsize$=$};
		\node [style=none] (24) at (-0.5, -2.75) {};
		\node [style=box] (25) at (-0.5, -1.75) {\scriptsize$\hspace{0.5em}\varphi\hspace{0.5em}$};
		\node [style=box] (26) at (-1.25, 0) {\scriptsize${}^A g$};
		\node [style=box] (27) at (0.25, 0.75) {\scriptsize${}^A h$};
		\node [style=none] (29) at (0.25, 1.75) {};
		\node [style=none] (30) at (-1, -1.5) {};
		\node [style=none] (31) at (0, -1.5) {};
		\node [style=none] (32) at (-1, -0.25) {};
		\node [style=none] (33) at (-1.5, -0.25) {};
		\node [style=none] (34) at (-1, -0.25) {};
		\node [style=blackdot] (35) at (-1.5, -1) {};
		\node [style=none] (36) at (0, 0.5) {};
		\node [style=none] (37) at (0.5, 0.5) {};
		\node [style=blackdot] (38) at (0, -0.25) {};
		\node [style=none] (39) at (-2.5, 0) {\scriptsize$=$};
		\node [style=none] (40) at (0.25, 2.5) {};
		\node [style=none] (41) at (3.5, -2.75) {};
		\node [style=box] (42) at (3.5, -1.75) {\scriptsize$\hspace{0.5em}\varphi\hspace{0.5em}$};
		\node [style=box] (43) at (2.75, 0) {\scriptsize${}^A g$};
		\node [style=box] (44) at (4.25, 1.75) {\scriptsize${}^A h$};
		\node [style=none] (46) at (3, -1.5) {};
		\node [style=none] (47) at (4, -1.5) {};
		\node [style=none] (48) at (3, -0.25) {};
		\node [style=none] (49) at (2.5, -0.25) {};
		\node [style=none] (50) at (3, -0.25) {};
		\node [style=blackdot] (51) at (2.5, -1) {};
		\node [style=none] (52) at (4, 1.5) {};
		\node [style=none] (53) at (4.5, 1.5) {};
		\node [style=blackdot] (54) at (4, 0.25) {};
		\node [style=none] (55) at (1.5, 0) {\scriptsize$=$};
		\node [style=none] (56) at (4.25, 2.5) {};
		\node [style=none] (57) at (4, 1) {};
		\node [style=none] (58) at (7.5, -2.75) {};
		\node [style=box] (59) at (7.5, -1.75) {\scriptsize$\hspace{0.5em}\varphi\hspace{0.5em}$};
		\node [style=box] (60) at (6.75, 0) {\scriptsize${}^A g$};
		\node [style=box] (61) at (7.75, 1.75) {\scriptsize${}^A h$};
		\node [style=none] (62) at (7, -1.5) {};
		\node [style=none] (63) at (8, -1.5) {};
		\node [style=none] (64) at (7, -0.25) {};
		\node [style=none] (65) at (6.5, -0.25) {};
		\node [style=none] (66) at (7, -0.25) {};
		\node [style=blackdot] (67) at (6.5, -1) {};
		\node [style=none] (68) at (7.5, 1.5) {};
		\node [style=none] (69) at (8, 1.5) {};
		\node [style=none] (71) at (5.5, 0) {\scriptsize$=$};
		\node [style=none] (72) at (7.75, 2.5) {};
	\end{pgfonlayer}
	\begin{pgfonlayer}{edgelayer}
		\draw [style=morphism-edge] (0.center) to (1);
		\draw [style=morphism-edge, in=270, out=90] (7.center) to (3);
		\draw [style=morphism-edge, in=270, out=90] (6.center) to (2);
		\draw [style=morphism-edge, in=180, out=90] (2) to (4.center);
		\draw [style=morphism-edge, in=90, out=0] (4.center) to (3);
		\draw [style=morphism-edge] (4.center) to (5.center);
		\draw [style=morphism-edge] (8.center) to (9);
		\draw [style=morphism-edge, in=180, out=90] (10) to (12.center);
		\draw [style=morphism-edge, in=90, out=0] (12.center) to (11);
		\draw [style=morphism-edge] (12.center) to (13.center);
		\draw [style=morphism-edge] (18.center) to (14.center);
		\draw [style=morphism-edge] (19) to (17.center);
		\draw [style=morphism-edge, in=270, out=90] (15.center) to (21.center);
		\draw [style=morphism-edge] (22) to (20.center);
		\draw [style=morphism-edge] (24.center) to (25);
		\draw [style=morphism-edge] (34.center) to (30.center);
		\draw [style=morphism-edge] (35) to (33.center);
		\draw [style=morphism-edge, in=270, out=90] (31.center) to (37.center);
		\draw [style=morphism-edge] (38) to (36.center);
		\draw [style=morphism-edge] (40.center) to (27);
		\draw [style=morphism-edge, in=180, out=90] (26) to (29.center);
		\draw [style=morphism-edge] (41.center) to (42);
		\draw [style=morphism-edge] (50.center) to (46.center);
		\draw [style=morphism-edge] (51) to (49.center);
		\draw [style=morphism-edge, in=270, out=90] (47.center) to (53.center);
		\draw [style=morphism-edge] (54) to (52.center);
		\draw [style=morphism-edge] (56.center) to (44);
		\draw [style=morphism-edge, in=180, out=90] (43) to (57.center);
		\draw [style=morphism-edge] (58.center) to (59);
		\draw [style=morphism-edge] (66.center) to (62.center);
		\draw [style=morphism-edge] (67) to (65.center);
		\draw [style=morphism-edge, in=270, out=90] (63.center) to (69.center);
		\draw [style=morphism-edge] (72.center) to (61);
		\draw [style=morphism-edge, in=270, out=90] (60) to (68.center);
	\end{pgfonlayer}
\end{tikzpicture}
   \caption{Verification of condition~\ref{Day2}.}%
   \label{fig:fig2}
 \end{figure}
\end{proof}

\section{Mixed subfunctors and \texorpdfstring{$\cp$}{Cp}-stable ideals}\label{sec:mixedsubfunctorscpideals}

\subsection{From \texorpdfstring{$\cp$}{Cp}-stable ideals to mixed subfunctors}
Recall that $\cp$ is closed under taking duals.
Since the bijection we aim to establish is invariant under monoidal equivalence, we may assume $\mathcal{C}$ to be strict, in particular that $X \otimes \mathbb{1} = X = \mathbb{1} \otimes X$ for $X \in \mathcal{C}$, and $(P \otimes Q) \triangleright M = P \triangleright (Q \triangleright M)$ for $M \in \modca$.

\begin{definition}\label{def:Smap}
 Let $\mathfrak{J}$ be a $\cp$-stable ideal in $\pfa$. Let $Q \xrightarrow{q} \mathbb{1}$ be an epimorphism in $\mathcal{C}$, for $Q \in \cp$.
 We define a mixed subfunctor $\mathbb{S}(\mathfrak{J})$ of $\mathcal{C}(\blank,A)$ by
 \begin{equation}\label{SDefined}
 \mathbb{S}(\mathfrak{J})(P) = \setj{(q\otimes A) \circ g^{1} \; | \; g \in \mathfrak{J}(P,Q) \subseteq \homa(P\triangleright A,Q\triangleright A)}.
 \end{equation}
 In string diagrams, we shall write an element in $\mathbb{S}(\mathfrak{J})(P)$ as
 
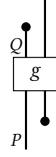
\[\begin{tikzpicture}[tikzfig]
	\begin{pgfonlayer}{nodelayer}
		\node [style=none] (0) at (-0.25, -2) {};
		\node [style=blackdot] (1) at (0.25, -1.25) {};
		\node [style=blackdot] (2) at (-0.25, 1.25) {};
		\node [style=none] (3) at (0.25, 2) {};
		\node [style=box] (4) at (0, 0) {\scriptsize$\hspace{0.3em}g\hspace{0.3em}$};
		\node [style=none] (5) at (-0.25, -0.25) {};
		\node [style=none] (6) at (0.25, -0.25) {};
		\node [style=none] (7) at (-0.25, 0.25) {};
		\node [style=none] (8) at (0.25, 0.25) {};
		\node [style=none] (9) at (-0.5, -1.75) {\tiny $P$};
		\node [style=none] (10) at (-0.5, 0.75) {\tiny$Q$};
	\end{pgfonlayer}
	\begin{pgfonlayer}{edgelayer}
		\draw [style=morphism-edge] (0.center) to (5.center);
		\draw [style=morphism-edge] (1) to (6.center);
		\draw [style=morphism-edge] (8.center) to (3.center);
		\draw [style=morphism-edge] (7.center) to (2);
	\end{pgfonlayer}
      \end{tikzpicture}
    \]
\end{definition}

\begin{lemma}
 The definition of $\mathbb{S}(\mathfrak{J})(P)$ is independent of the choice of epimorphism $Q \xtwoheadrightarrow{q} \mathbb{1}$.
\end{lemma}

\begin{proof}
 Let $Q \xtwoheadrightarrow{q} \mathbb{1}$ and $Q' \xtwoheadrightarrow{q'} \mathbb{1}$ be two choices of epimorphism from a projective to $\mathbb{1}$, and let ${\mathbb{S}(\mathfrak{J})(P)}_{Q}$ and ${\mathbb{S}(\mathfrak{J})(P)}_{Q'}$ denote the resulting subspaces of $\mathcal{C}(P,A)$. We obtain a commutative diagram
\[\begin{tikzcd}
	& {Q'} \\
	Q & {\mathbb{1}}
	\arrow["{h'}"', shift right, curve={height=12pt}, dashed, from=1-2, to=2-1]
	\arrow["{q'}", two heads, from=1-2, to=2-2]
	\arrow["h"', shift right, curve={height=-12pt}, dashed, from=2-1, to=1-2]
	\arrow["q"', two heads, from=2-1, to=2-2]
\end{tikzcd}\]
Then
\[
{\mathbb{S}(\mathfrak{J})(P)}_{Q}
= \{\,(q\otimes A) \circ f^{1} \mid f \in \mathfrak{J}(P,Q)\,\}
= \{\,(q'\otimes A) \circ (h\otimes A) \circ f^{1} \mid f \in \mathfrak{J}(P,Q)\,\}
\subseteq {\mathbb{S}(\mathfrak{J})(P)}_{Q'},
\]
and similarly ${\mathbb{S}(\mathfrak{J})(P)}_{Q'} \subseteq {\mathbb{S}(\mathfrak{J})(P)}_{Q}$.
\end{proof}

\begin{lemma}
 The collection $\{\,\mathbb{S}(\mathfrak{J})(P) \mid P \in \cp\,\}$ defines a mixed subfunctor of $\,{\mathcal{C}(\blank,A)}_{|\cp}$.
\end{lemma}

\begin{proof}
  Throughout, let $g \in \mathfrak{J}(P,Q) \subseteq \homa(P \triangleright A, Q \triangleright A)$.
  For condition~\ref{mixed2},
  suppose that  $f\colon P' \triangleright A \to P\triangleright A$.
  We have to show that
  \[
    {((q \otimes A) \circ g^1)}^A \circ f \in {\mathbb{S}(\mathfrak{J})(P')}^A,
  \]
  which is the case if and only if
  \[
    {({((q \otimes A)\circ g^1)}^A \circ f)}^1 \in \mathbb{S}(\mathfrak{J})(P').
  \]
  This follows by
  \[\begin{tikzpicture}[tikzfig]
	\begin{pgfonlayer}{nodelayer}
		\node [style=none] (0) at (-0.25, -2.5) {};
		\node [style=blackdot] (1) at (1.25, -1.5) {};
		\node [style=none] (3) at (1.25, 2.5) {};
		\node [style=box] (4) at (0.5, -0.5) {\scriptsize$\hspace{1em}f\hspace{1em}$};
		\node [style=none] (5) at (-0.25, -0.75) {};
		\node [style=none] (6) at (1.25, -0.75) {};
		\node [style=none] (7) at (-0.25, -0.25) {};
		\node [style=none] (8) at (1.25, -0.25) {};
		\node [style=box] (9) at (0, 1.75) {\scriptsize$\hspace{.5em}g\hspace{.5em}$};
		\node [style=none] (10) at (-0.25, 1.5) {};
		\node [style=none] (11) at (0.25, 1.5) {};
		\node [style=blackdot] (12) at (0.25, 0.75) {};
		\node [style=none] (13) at (-0.25, 2) {};
		\node [style=none] (14) at (0.25, 2) {};
		\node [style=blackdot] (15) at (-0.25, 3) {};
		\node [style=none] (16) at (0.75, 3.25) {};
		\node [style=none] (17) at (0.75, 4) {};
		\node [style=none] (18) at (2.25, 0.75) {$=$};
		\node [style=none] (19) at (3.25, -2.5) {};
		\node [style=blackdot] (20) at (4.75, -1.5) {};
		\node [style=none] (21) at (4.75, 1.5) {};
		\node [style=box] (22) at (4, -0.5) {\scriptsize$\hspace{1em}f\hspace{1em}$};
		\node [style=none] (23) at (3.25, -0.75) {};
		\node [style=none] (24) at (4.75, -0.75) {};
		\node [style=none] (25) at (3.25, -0.25) {};
		\node [style=none] (26) at (4.75, -0.25) {};
		\node [style=box] (27) at (4, 1.75) {\scriptsize$\hspace{1em}g\hspace{1em}$};
		\node [style=none] (28) at (3.25, 1.5) {};
		\node [style=none] (31) at (3.25, 2) {};
		\node [style=blackdot] (33) at (3.25, 3) {};
		\node [style=none] (34) at (4.75, 3.75) {};
		\node [style=none] (35) at (4.75, 2) {};
		\node [style=none] (36) at (10.75, 0.75) {$= (q\otimes A)\circ {(g\circ f)}^1 \in \mathbb{S}(\mathfrak{J})(P')$};
	\end{pgfonlayer}
	\begin{pgfonlayer}{edgelayer}
		\draw [style=morphism-edge] (1) to (6.center);
		\draw [style=morphism-edge] (8.center) to (3.center);
		\draw [style=morphism-edge] (0.center) to (5.center);
		\draw [style=morphism-edge] (7.center) to (10.center);
		\draw [style=morphism-edge] (11.center) to (12);
		\draw [style=morphism-edge, in=180, out=90] (14.center) to (16.center);
		\draw [style=morphism-edge, in=90, out=0] (16.center) to (3.center);
		\draw [style=morphism-edge] (15) to (13.center);
		\draw [style=morphism-edge] (17.center) to (16.center);
		\draw [style=morphism-edge] (20) to (24.center);
		\draw [style=morphism-edge] (26.center) to (21.center);
		\draw [style=morphism-edge] (19.center) to (23.center);
		\draw [style=morphism-edge] (25.center) to (28.center);
		\draw [style=morphism-edge] (33) to (31.center);
		\draw [style=morphism-edge] (35.center) to (34.center);
	\end{pgfonlayer}
      \end{tikzpicture}
    \]

    It is left to prove~\ref{mixed3}.
    Suppose that $f \colon A \triangleleft P' \to A \triangleleft P$.
    We have to show that
    \[
      {}^A((q \otimes A) \circ g^1) \circ {}^1\!f \in \mathbb{S}(\mathfrak{J})(P'),
    \]
    which is the case if and only if
    \[
      {}^1({}^A((q \otimes A) \circ g^1) \circ f) \in \mathbb{S}(\mathfrak{J})(P').
    \]
    By \autoref{fig:fig6}, the morphism $\xi$ defined therein is a right $A$-module morphism,
    and so it is left to show that
    \[
      \xi \in \mathfrak{J}(P' \otimes Q, P' \otimes \ld{P'})
      \subseteq \homa(P' \otimes Q \triangleright A, P' \otimes \ld{P'} \triangleright A).
    \]
    This follows by
    \[\begin{tikzpicture}[tikzfig]
	\begin{pgfonlayer}{nodelayer}
		\node [style=none] (0) at (-4.75, -0.75) {};
		\node [style=none] (1) at (-4.25, -0.75) {};
		\node [style=box] (2) at (-4.5, 0.75) {$\hspace{0.5em}g\hspace{0.5em}$};
		\node [style=none] (3) at (-4.75, 2.25) {};
		\node [style=none] (4) at (-4.25, 2.25) {};
		\node [style=none] (5) at (-4.75, 0.5) {};
		\node [style=none] (6) at (-4.25, 0.5) {};
		\node [style=none] (7) at (-4.75, 1) {};
		\node [style=none] (8) at (-4.25, 1) {};
		\node [style=none] (9) at (-1.75, 0.75) {$\in \mathfrak{J} \implies$};
		\node [style=none] (10) at (0.5, -0.75) {};
		\node [style=none] (11) at (1, -0.75) {};
		\node [style=box] (12) at (0.75, 0.75) {$\hspace{0.5em}g\hspace{0.5em}$};
		\node [style=blackdot] (13) at (0.5, 1.75) {};
		\node [style=none] (14) at (1, 2.25) {};
		\node [style=none] (15) at (0.5, 0.5) {};
		\node [style=none] (16) at (1, 0.5) {};
		\node [style=none] (17) at (0.5, 1) {};
		\node [style=none] (18) at (1, 1) {};
		\node [style=none] (19) at (3.5, 0.75) {$\in \mathfrak{J} \implies$};
		\node [style=none] (20) at (7.25, -0.75) {};
		\node [style=none] (21) at (7.75, -0.75) {};
		\node [style=box] (22) at (7.5, 0.75) {$\hspace{0.5em}g\hspace{0.5em}$};
		\node [style=blackdot] (23) at (7.25, 1.75) {};
		\node [style=none] (24) at (7.75, 2.25) {};
		\node [style=none] (25) at (7.25, 0.5) {};
		\node [style=none] (26) at (7.75, 0.5) {};
		\node [style=none] (27) at (7.25, 1) {};
		\node [style=none] (28) at (7.75, 1) {};
		\node [style=none] (29) at (10.25, 0.75) {$\in \mathfrak{J} \implies$};
		\node [style=none] (30) at (5.25, -0.75) {};
		\node [style=none] (31) at (5.75, -0.75) {};
		\node [style=none] (32) at (5.25, 2.25) {};
		\node [style=none] (33) at (5.75, 2.25) {};
		\node [style=none] (34) at (6.25, -0.75) {};
		\node [style=none] (36) at (6.25, 2.25) {};
		\node [style=none] (37) at (14, -0.75) {};
		\node [style=none] (38) at (14.5, -0.75) {};
		\node [style=box] (39) at (14.25, 0.75) {$\hspace{0.5em}g\hspace{0.5em}$};
		\node [style=blackdot] (40) at (14, 1.75) {};
		\node [style=none] (41) at (14.5, 2.25) {};
		\node [style=none] (42) at (14, 0.5) {};
		\node [style=none] (43) at (14.5, 0.5) {};
		\node [style=none] (44) at (14, 1) {};
		\node [style=none] (45) at (14.5, 1) {};
		\node [style=none] (47) at (12, -0.75) {};
		\node [style=none] (48) at (12.5, -0.75) {};
		\node [style=none] (49) at (12, 3.5) {};
		\node [style=none] (50) at (12.5, 3.5) {};
		\node [style=none] (51) at (13, -0.75) {};
		\node [style=none] (52) at (13, 2.25) {};
		\node [style=none] (53) at (13.75, 3) {};
		\node [style=none] (54) at (13.75, 3.5) {};
	\end{pgfonlayer}
	\begin{pgfonlayer}{edgelayer}
		\draw [style=morphism-edge] (0.center) to (5.center);
		\draw [style=morphism-edge] (6.center) to (1.center);
		\draw [style=morphism-edge] (7.center) to (3.center);
		\draw [style=morphism-edge] (8.center) to (4.center);
		\draw [style=morphism-edge] (10.center) to (15.center);
		\draw [style=morphism-edge] (16.center) to (11.center);
		\draw [style=morphism-edge] (17.center) to (13);
		\draw [style=morphism-edge] (18.center) to (14.center);
		\draw [style=morphism-edge] (20.center) to (25.center);
		\draw [style=morphism-edge] (26.center) to (21.center);
		\draw [style=morphism-edge] (27.center) to (23);
		\draw [style=morphism-edge] (28.center) to (24.center);
		\draw [style=morphism-edge] (30.center) to (32.center);
		\draw [style=morphism-edge] (33.center) to (31.center);
		\draw [style=morphism-edge] (34.center) to (36.center);
		\draw [style=morphism-edge] (37.center) to (42.center);
		\draw [style=morphism-edge] (43.center) to (38.center);
		\draw [style=morphism-edge] (44.center) to (40);
		\draw [style=morphism-edge] (45.center) to (41.center);
		\draw [style=morphism-edge] (47.center) to (49.center);
		\draw [style=morphism-edge] (50.center) to (48.center);
		\draw [style=morphism-edge] (51.center) to (52.center);
		\draw [style=morphism-edge, in=90, out=-180] (53.center) to (52.center);
		\draw [style=morphism-edge, in=360, out=90] (41.center) to (53.center);
		\draw [style=morphism-edge] (54.center) to (53.center);
	\end{pgfonlayer}
      \end{tikzpicture}
    \]
    which is the top part of $\xi$.
    \begin{figure}[htbp]
      \input{fig6.tikz}
      \caption{Verification of the first part of condition~\ref{mixed3}.}%
      \label{fig:fig6}
    \end{figure}
\end{proof}

\subsection{From mixed subfunctors to \texorpdfstring{$\cp$}{Cp}-stable ideals}

\begin{definition}
 Let $\mathtt{I}$ be a mixed subfunctor of ${\mathcal{C}(\blank,A)}_{|\cp}$. We define a $\cp$-stable ideal $\mathbb{R}(\mathtt{I})$ in $\pfa$ by
 \begin{equation}\label{RDefined}
 \mathbb{R}(\mathtt{I})(P,Q) = \setj{{((Q\otimes f) \circ (\coev_{{Q}} \otimes P))}^{A} \; | \; f \in \mathtt{I}(\ld{Q}\otimes P)}.
 \end{equation}
\end{definition}

\begin{lemma}
 The assignment of \autoref{RDefined} defines a $\cp$-stable ideal $\mathbb{R}(\mathtt{I})$ in $\pfa$.
\end{lemma}

\begin{proof}
 Let $f \in \mathtt{I}(\ld{Q} \otimes P)$, and let $\mathbb{R}(f) \defeq {((Q\otimes f) \circ (\coev_{{Q}} \otimes P))}^{A}$.
 We first show that $\mathbb{R}(\mathtt{I})$ is stable under precomposition with morphisms of $\pfa$. Thus, let $g \in \homa(P' \triangleright A, P \triangleright A)$, and we will show that $\mathbb{R}(f) \circ g \in \mathbb{R}(\mathtt{I})(P',Q)$.
    In the following equation, we write $\mathbb{R}(f) \circ g$ on the left-hand side, and on the right-hand side we have the morphism ${\mathbb{R}(f^A \circ (Q \otimes g))}^1$:
    
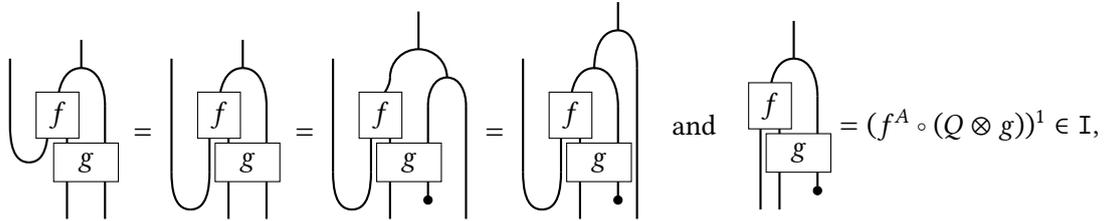
\[\begin{tikzpicture}[tikzfig]
	\begin{pgfonlayer}{nodelayer}
		\node [style=none] (0) at (-0.25, -2.5) {};
		\node [style=none] (1) at (0.75, -2.5) {};
		\node [style=box] (2) at (0.25, -1) {$\hspace{0.5em}g\hspace{0.5em}$};
		\node [style=none] (4) at (0.75, 0.5) {};
		\node [style=none] (5) at (-0.25, -1.25) {};
		\node [style=none] (6) at (0.75, -1.25) {};
		\node [style=none] (7) at (-0.25, -0.75) {};
		\node [style=none] (8) at (0.75, -0.75) {};
		\node [style=box] (9) at (-0.5, 0.25) {$\hspace{0.1em}f\hspace{0.1em}$};
		\node [style=none] (10) at (-0.25, 0) {};
		\node [style=none] (11) at (-0.75, 0) {};
		\node [style=none] (12) at (-1.25, -1) {};
		\node [style=none] (13) at (-1.75, 0) {};
		\node [style=none] (14) at (-0.5, 0.5) {};
		\node [style=none] (15) at (0.125, 1.5) {};
		\node [style=none] (16) at (0.125, 2.25) {};
		\node [style=none] (17) at (-1.75, 1.75) {};
		\node [style=none] (18) at (1.75, -0.25) {$=$};
		\node [style=none] (19) at (4, -2.5) {};
		\node [style=none] (20) at (5, -2.5) {};
		\node [style=box] (21) at (4.5, -1) {$\hspace{0.5em}g\hspace{0.5em}$};
		\node [style=none] (22) at (5, 0.5) {};
		\node [style=none] (23) at (4, -1.25) {};
		\node [style=none] (24) at (5, -1.25) {};
		\node [style=none] (25) at (4, -0.75) {};
		\node [style=none] (26) at (5, -0.75) {};
		\node [style=box] (27) at (3.75, 0.25) {$\hspace{0.1em}f\hspace{0.1em}$};
		\node [style=none] (28) at (4, 0) {};
		\node [style=none] (29) at (3.5, -1.25) {};
		\node [style=none] (30) at (3, -2.25) {};
		\node [style=none] (31) at (2.5, -1.25) {};
		\node [style=none] (32) at (3.75, 0.5) {};
		\node [style=none] (33) at (4.375, 1.5) {};
		\node [style=none] (34) at (4.375, 2.25) {};
		\node [style=none] (35) at (2.5, 1.75) {};
		\node [style=none] (36) at (6, -0.25) {$=$};
		\node [style=none] (37) at (3.5, 0) {};
		\node [style=none] (38) at (8.25, -2.5) {};
		\node [style=blackdot] (39) at (9.25, -2) {};
		\node [style=box] (40) at (8.75, -1) {$\hspace{0.5em}g\hspace{0.5em}$};
		\node [style=none] (41) at (9.25, 0.5) {};
		\node [style=none] (42) at (8.25, -1.25) {};
		\node [style=none] (43) at (9.25, -1.25) {};
		\node [style=none] (44) at (8.25, -0.75) {};
		\node [style=none] (45) at (9.25, -0.75) {};
		\node [style=box] (46) at (8, 0.25) {$\hspace{0.1em}f\hspace{0.1em}$};
		\node [style=none] (47) at (8.25, 0) {};
		\node [style=none] (48) at (7.75, -1.25) {};
		\node [style=none] (49) at (7.25, -2.25) {};
		\node [style=none] (50) at (6.75, -1.25) {};
		\node [style=none] (51) at (8.25, 1.25) {};
		\node [style=none] (52) at (9, 2) {};
		\node [style=none] (54) at (6.75, 1.75) {};
		\node [style=none] (55) at (7.75, 0) {};
		\node [style=none] (56) at (10.25, -2.5) {};
		\node [style=none] (59) at (9, 3) {};
		\node [style=none] (60) at (11, -0.25) {$=$};
		\node [style=none] (61) at (13.25, -2.5) {};
		\node [style=blackdot] (62) at (14.25, -2) {};
		\node [style=box] (63) at (13.75, -1) {$\hspace{0.5em}g\hspace{0.5em}$};
		\node [style=none] (64) at (14.25, 0.5) {};
		\node [style=none] (65) at (13.25, -1.25) {};
		\node [style=none] (66) at (14.25, -1.25) {};
		\node [style=none] (67) at (13.25, -0.75) {};
		\node [style=none] (68) at (14.25, -0.75) {};
		\node [style=box] (69) at (13, 0.25) {$\hspace{0.1em}f\hspace{0.1em}$};
		\node [style=none] (70) at (13.25, 0) {};
		\node [style=none] (71) at (12.75, -1.25) {};
		\node [style=none] (72) at (12.25, -2.25) {};
		\node [style=none] (73) at (11.75, -1.25) {};
		\node [style=none] (74) at (13, 0.5) {};
		\node [style=none] (75) at (13.625, 1.5) {};
		\node [style=none] (76) at (13.625, 1.5) {};
		\node [style=none] (77) at (11.75, 1.75) {};
		\node [style=none] (78) at (12.75, 0) {};
		\node [style=none] (79) at (14.75, -2.5) {};
		\node [style=none] (80) at (14.75, 1.5) {};
		\node [style=none] (81) at (14.25, 2.5) {};
		\node [style=none] (82) at (14.25, 3.25) {};
		\node [style=none] (83) at (10.25, 0.5) {};
		\node [style=none] (84) at (9.75, 1.25) {};
		\node [style=none] (85) at (8, 0.5) {};
		\node [style=none] (87) at (16.25, 0) {and};
		\node [style=blackdot] (89) at (19.5, -1.75) {};
		\node [style=box] (90) at (19, -0.75) {$\hspace{0.5em}g\hspace{0.5em}$};
		\node [style=none] (91) at (19.5, 0.75) {};
		\node [style=none] (92) at (18.5, -1) {};
		\node [style=none] (93) at (19.5, -1) {};
		\node [style=none] (94) at (18.5, -0.5) {};
		\node [style=none] (95) at (19.5, -0.5) {};
		\node [style=box] (96) at (18.25, 0.5) {$\hspace{0.1em}f\hspace{0.1em}$};
		\node [style=none] (97) at (18.5, 0.25) {};
		\node [style=none] (98) at (18, -2.25) {};
		\node [style=none] (102) at (18.875, 1.75) {};
		\node [style=none] (104) at (18, 0.25) {};
		\node [style=none] (105) at (18.875, 2.75) {};
		\node [style=none] (108) at (18.25, 0.75) {};
		\node [style=none] (109) at (18.5, -2.25) {};
		\node [style=none] (110) at (23.5, 0) {$= (f^A \circ (Q \otimes g))^1 \in \mathtt{I}$,};
	\end{pgfonlayer}
	\begin{pgfonlayer}{edgelayer}
		\draw [style=morphism-edge] (0.center) to (5.center);
		\draw [style=morphism-edge] (6.center) to (1.center);
		\draw [style=morphism-edge] (8.center) to (4.center);
		\draw [style=morphism-edge] (7.center) to (10.center);
		\draw [style=morphism-edge, in=180, out=90] (14.center) to (15.center);
		\draw [style=morphism-edge, in=90, out=0] (15.center) to (4.center);
		\draw [style=morphism-edge] (16.center) to (15.center);
		\draw [style=morphism-edge, in=0, out=-90] (11.center) to (12.center);
		\draw [style=morphism-edge, in=270, out=180] (12.center) to (13.center);
		\draw [style=morphism-edge] (17.center) to (13.center);
		\draw [style=morphism-edge] (19.center) to (23.center);
		\draw [style=morphism-edge] (24.center) to (20.center);
		\draw [style=morphism-edge] (26.center) to (22.center);
		\draw [style=morphism-edge] (25.center) to (28.center);
		\draw [style=morphism-edge, in=180, out=90] (32.center) to (33.center);
		\draw [style=morphism-edge, in=90, out=0] (33.center) to (22.center);
		\draw [style=morphism-edge] (34.center) to (33.center);
		\draw [style=morphism-edge, in=0, out=-90] (29.center) to (30.center);
		\draw [style=morphism-edge, in=270, out=180] (30.center) to (31.center);
		\draw [style=morphism-edge] (35.center) to (31.center);
		\draw [style=morphism-edge] (37.center) to (29.center);
		\draw [style=morphism-edge] (38.center) to (42.center);
		\draw [style=morphism-edge] (43.center) to (39);
		\draw [style=morphism-edge] (45.center) to (41.center);
		\draw [style=morphism-edge] (44.center) to (47.center);
		\draw [style=morphism-edge, in=180, out=90] (51.center) to (52.center);
		\draw [style=morphism-edge, in=0, out=-90] (48.center) to (49.center);
		\draw [style=morphism-edge, in=270, out=180] (49.center) to (50.center);
		\draw [style=morphism-edge] (54.center) to (50.center);
		\draw [style=morphism-edge] (55.center) to (48.center);
		\draw [style=morphism-edge] (61.center) to (65.center);
		\draw [style=morphism-edge] (66.center) to (62);
		\draw [style=morphism-edge] (68.center) to (64.center);
		\draw [style=morphism-edge] (67.center) to (70.center);
		\draw [style=morphism-edge, in=180, out=90] (74.center) to (75.center);
		\draw [style=morphism-edge, in=90, out=0] (75.center) to (64.center);
		\draw [style=morphism-edge] (76.center) to (75.center);
		\draw [style=morphism-edge, in=0, out=-90] (71.center) to (72.center);
		\draw [style=morphism-edge, in=270, out=180] (72.center) to (73.center);
		\draw [style=morphism-edge] (77.center) to (73.center);
		\draw [style=morphism-edge] (78.center) to (71.center);
		\draw [style=morphism-edge, in=90, out=-180] (81.center) to (76.center);
		\draw [style=morphism-edge, in=90, out=0] (81.center) to (80.center);
		\draw [style=morphism-edge] (80.center) to (79.center);
		\draw [style=morphism-edge] (82.center) to (81.center);
		\draw [style=morphism-edge, in=180, out=90] (41.center) to (84.center);
		\draw [style=morphism-edge, in=90, out=0] (84.center) to (83.center);
		\draw [style=morphism-edge, in=360, out=90] (84.center) to (52.center);
		\draw [style=morphism-edge] (83.center) to (56.center);
		\draw [style=morphism-edge] (52.center) to (59.center);
		\draw [style=morphism-edge, in=270, out=90] (85.center) to (51.center);
		\draw [style=morphism-edge] (93.center) to (89);
		\draw [style=morphism-edge] (95.center) to (91.center);
		\draw [style=morphism-edge] (94.center) to (97.center);
		\draw [style=morphism-edge] (104.center) to (98.center);
		\draw [style=morphism-edge] (102.center) to (105.center);
		\draw [style=morphism-edge] (109.center) to (92.center);
		\draw [style=morphism-edge, in=180, out=90] (108.center) to (102.center);
		\draw [style=morphism-edge, in=90, out=0] (102.center) to (91.center);
	\end{pgfonlayer}
      \end{tikzpicture}
    \]
    so $\mathbb{R}(f) \circ g \in \mathbb{R}(P',Q)$.

    Now, we show that $\mathbb{R}(\mathtt{I})$ is stable under postcomposition with morphisms in $\pfa$. Let $h \in \homa(Q \triangleright A, Q' \triangleright A)$. Similarly to the previous part, we first show that there is $\widehat{h\circ f}$ such that $h \circ \mathbb{R}(f) = \mathbb{R}(\widehat{h \circ f})$, and then that $\widehat{h \circ f}$ belongs to $\mathtt{I}$. The morphism $\mathbb{R}(\widehat{h \circ f})$ is displayed as right-hand side below:
    
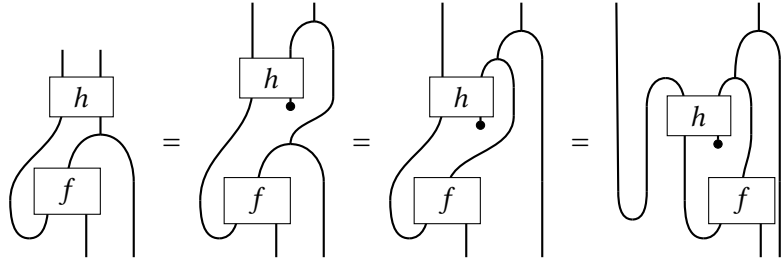
\[\begin{tikzpicture}[tikzfig]
	\begin{pgfonlayer}{nodelayer}
		\node [style=none] (0) at (-8.75, -2.5) {};
		\node [style=none] (1) at (-7.25, -3) {};
		\node [style=box] (2) at (-7.75, -1.25) {$\hspace{0.5em}f\hspace{0.5em}$};
		\node [style=none] (3) at (-8.25, -1.5) {};
		\node [style=none] (4) at (-7.25, -1.5) {};
		\node [style=none] (5) at (-9.25, -1.5) {};
		\node [style=none] (6) at (-6, -3) {};
		\node [style=none] (7) at (-6.875, 0.25) {};
		\node [style=none] (8) at (-7.75, -1) {};
		\node [style=none] (9) at (-6, -1) {};
		\node [style=none] (10) at (-6.875, 1) {};
		\node [style=box] (11) at (-7.375, 1.25) {$\hspace{0.5em}h\hspace{0.5em}$};
		\node [style=none] (12) at (-7.875, 1) {};
		\node [style=none] (13) at (-7.875, 1.5) {};
		\node [style=none] (14) at (-6.875, 1.5) {};
		\node [style=none] (15) at (-7.875, 2.5) {};
		\node [style=none] (16) at (-6.875, 2.5) {};
		\node [style=none] (17) at (-5, 0) {$=$};
		\node [style=none] (18) at (-3.75, -2.5) {};
		\node [style=none] (19) at (-2.25, -3) {};
		\node [style=box] (20) at (-2.75, -1.5) {$\hspace{0.5em}f\hspace{0.5em}$};
		\node [style=none] (21) at (-3.25, -1.75) {};
		\node [style=none] (22) at (-2.25, -1.75) {};
		\node [style=none] (23) at (-4.25, -1.5) {};
		\node [style=none] (24) at (-1, -3) {};
		\node [style=none] (26) at (-2.75, -1.25) {};
		\node [style=none] (27) at (-1, -1) {};
		\node [style=none] (28) at (-1.875, 1.5) {};
		\node [style=box] (29) at (-2.375, 1.75) {$\hspace{0.5em}h\hspace{0.5em}$};
		\node [style=none] (30) at (-2.875, 1.5) {};
		\node [style=none] (31) at (-2.875, 2) {};
		\node [style=none] (32) at (-1.875, 2) {};
		\node [style=none] (33) at (-2.875, 3.75) {};
		\node [style=none] (34) at (-1.875, 2.5) {};
		\node [style=none] (35) at (0, 0) {$=$};
		\node [style=blackdot] (36) at (-1.875, 1) {};
		\node [style=none] (37) at (-0.75, 2.5) {};
		\node [style=none] (38) at (-0.75, 1.25) {};
		\node [style=none] (39) at (-1.25, 3.25) {};
		\node [style=none] (40) at (-1.25, 3.75) {};
		\node [style=none] (41) at (-1.875, 0) {};
		\node [style=none] (42) at (1.25, -2.5) {};
		\node [style=none] (43) at (2.75, -3) {};
		\node [style=box] (44) at (2.25, -1.5) {$\hspace{0.5em}f\hspace{0.5em}$};
		\node [style=none] (45) at (1.75, -1.75) {};
		\node [style=none] (46) at (2.75, -1.75) {};
		\node [style=none] (47) at (0.75, -1.5) {};
		\node [style=none] (48) at (4.75, -3) {};
		\node [style=none] (49) at (2.25, -1.25) {};
		\node [style=none] (50) at (4.75, -1) {};
		\node [style=none] (51) at (3.125, 1) {};
		\node [style=box] (52) at (2.625, 1.25) {$\hspace{0.5em}h\hspace{0.5em}$};
		\node [style=none] (53) at (2.125, 1) {};
		\node [style=none] (54) at (2.125, 1.5) {};
		\node [style=none] (55) at (3.125, 1.5) {};
		\node [style=none] (56) at (2.125, 3.75) {};
		\node [style=none] (57) at (3.125, 1.5) {};
		\node [style=none] (58) at (5.75, 0) {$=$};
		\node [style=blackdot] (59) at (3.125, 0.5) {};
		\node [style=none] (60) at (4, 1.5) {};
		\node [style=none] (61) at (4, 0.75) {};
		\node [style=none] (63) at (3.575, 2.25) {};
		\node [style=none] (64) at (4.2, 3) {};
		\node [style=none] (65) at (4.75, 2.25) {};
		\node [style=none] (66) at (4.2, 3.75) {};
		\node [style=none] (67) at (9, -2.5) {};
		\node [style=none] (68) at (10.5, -3) {};
		\node [style=box] (69) at (10, -1.5) {$\hspace{0.5em}f\hspace{0.5em}$};
		\node [style=none] (70) at (9.5, -1.75) {};
		\node [style=none] (71) at (10.5, -1.75) {};
		\node [style=none] (72) at (8.5, -1.5) {};
		\node [style=none] (73) at (11, -3) {};
		\node [style=none] (74) at (10, -1.25) {};
		\node [style=none] (75) at (11, -1) {};
		\node [style=none] (76) at (9.375, 0.5) {};
		\node [style=box] (77) at (8.875, 0.75) {$\hspace{0.5em}h\hspace{0.5em}$};
		\node [style=none] (78) at (8.5, 0.5) {};
		\node [style=none] (79) at (8.5, 1) {};
		\node [style=none] (80) at (9.375, 1) {};
		\node [style=none] (81) at (8, 1.75) {};
		\node [style=none] (82) at (9.375, 1) {};
		\node [style=blackdot] (84) at (9.375, 0) {};
		\node [style=none] (85) at (10.25, 1) {};
		\node [style=none] (86) at (10.25, 0.25) {};
		\node [style=none] (87) at (9.825, 1.75) {};
		\node [style=none] (88) at (10.45, 3) {};
		\node [style=none] (89) at (11, 2.25) {};
		\node [style=none] (90) at (10.45, 3.75) {};
		\node [style=none] (91) at (7.5, 1) {};
		\node [style=none] (92) at (7.5, -1.25) {};
		\node [style=none] (93) at (7.125, -2) {};
		\node [style=none] (94) at (6.75, -1.25) {};
		\node [style=none] (95) at (6.7, 3.75) {};
	\end{pgfonlayer}
	\begin{pgfonlayer}{edgelayer}
		\draw [style=morphism-edge, in=270, out=0] (0.center) to (3.center);
		\draw [style=morphism-edge] (4.center) to (1.center);
		\draw [style=morphism-edge, in=180, out=-90] (5.center) to (0.center);
		\draw [style=morphism-edge, in=360, out=90] (9.center) to (7.center);
		\draw [style=morphism-edge, in=180, out=90] (8.center) to (7.center);
		\draw [style=morphism-edge] (9.center) to (6.center);
		\draw [style=morphism-edge, in=270, out=90] (7.center) to (10.center);
		\draw [style=morphism-edge, in=270, out=90] (5.center) to (12.center);
		\draw [style=morphism-edge] (16.center) to (14.center);
		\draw [style=morphism-edge] (13.center) to (15.center);
		\draw [style=morphism-edge, in=270, out=0] (18.center) to (21.center);
		\draw [style=morphism-edge] (22.center) to (19.center);
		\draw [style=morphism-edge, in=180, out=-90] (23.center) to (18.center);
		\draw [style=morphism-edge] (27.center) to (24.center);
		\draw [style=morphism-edge, in=270, out=90] (23.center) to (30.center);
		\draw [style=morphism-edge] (34.center) to (32.center);
		\draw [style=morphism-edge] (31.center) to (33.center);
		\draw [style=morphism-edge] (36) to (28.center);
		\draw [style=morphism-edge, in=180, out=90] (34.center) to (39.center);
		\draw [style=morphism-edge, in=90, out=0] (39.center) to (37.center);
		\draw [style=morphism-edge] (40.center) to (39.center);
		\draw [style=morphism-edge, in=90, out=-90] (37.center) to (38.center);
		\draw [style=morphism-edge, in=180, out=90] (26.center) to (41.center);
		\draw [style=morphism-edge, in=90, out=0] (41.center) to (27.center);
		\draw [style=morphism-edge, in=270, out=90] (41.center) to (38.center);
		\draw [style=morphism-edge, in=270, out=0] (42.center) to (45.center);
		\draw [style=morphism-edge] (46.center) to (43.center);
		\draw [style=morphism-edge, in=180, out=-90] (47.center) to (42.center);
		\draw [style=morphism-edge] (50.center) to (48.center);
		\draw [style=morphism-edge, in=270, out=90] (47.center) to (53.center);
		\draw [style=morphism-edge] (57.center) to (55.center);
		\draw [style=morphism-edge] (54.center) to (56.center);
		\draw [style=morphism-edge] (59) to (51.center);
		\draw [style=morphism-edge, in=90, out=-90] (60.center) to (61.center);
		\draw [style=morphism-edge, in=270, out=90] (49.center) to (61.center);
		\draw [style=morphism-edge, in=360, out=90] (60.center) to (63.center);
		\draw [style=morphism-edge, in=180, out=90] (57.center) to (63.center);
		\draw [style=morphism-edge, in=180, out=90] (63.center) to (64.center);
		\draw [style=morphism-edge, in=90, out=0] (64.center) to (65.center);
		\draw [style=morphism-edge] (65.center) to (50.center);
		\draw [style=morphism-edge] (66.center) to (64.center);
		\draw [style=morphism-edge, in=270, out=0] (67.center) to (70.center);
		\draw [style=morphism-edge] (71.center) to (68.center);
		\draw [style=morphism-edge, in=180, out=-90] (72.center) to (67.center);
		\draw [style=morphism-edge] (75.center) to (73.center);
		\draw [style=morphism-edge, in=270, out=90] (72.center) to (78.center);
		\draw [style=morphism-edge] (82.center) to (80.center);
		\draw [style=morphism-edge, in=360, out=90] (79.center) to (81.center);
		\draw [style=morphism-edge] (84) to (76.center);
		\draw [style=morphism-edge, in=90, out=-90] (85.center) to (86.center);
		\draw [style=morphism-edge, in=270, out=90] (74.center) to (86.center);
		\draw [style=morphism-edge, in=360, out=90] (85.center) to (87.center);
		\draw [style=morphism-edge, in=180, out=90] (82.center) to (87.center);
		\draw [style=morphism-edge, in=180, out=90] (87.center) to (88.center);
		\draw [style=morphism-edge, in=90, out=0] (88.center) to (89.center);
		\draw [style=morphism-edge] (89.center) to (75.center);
		\draw [style=morphism-edge] (90.center) to (88.center);
		\draw [style=morphism-edge, in=180, out=90] (91.center) to (81.center);
		\draw [style=morphism-edge] (91.center) to (92.center);
		\draw [style=morphism-edge] (95.center) to (94.center);
		\draw [style=morphism-edge, in=180, out=-90] (94.center) to (93.center);
		\draw [style=morphism-edge, in=270, out=0] (93.center) to (92.center);
	\end{pgfonlayer}
      \end{tikzpicture}
    \]
    In the next equation, we display $\widehat{h \circ f}$ as the left-hand side, and rewrite it to the right-hand side which is easily seen to belong to $\mathtt{I}$ using~\ref{mixed3}:
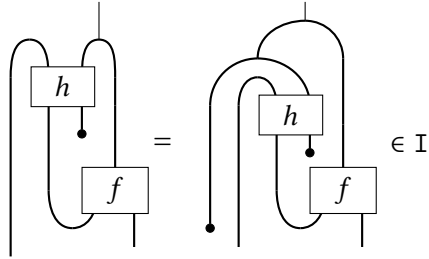
    \[\begin{tikzpicture}[tikzfig]
	\begin{pgfonlayer}{nodelayer}
		\node [style=none] (0) at (-2.875, -2.5) {};
		\node [style=none] (1) at (-1.25, -3) {};
		\node [style=box] (2) at (-1.75, -1.5) {$\hspace{0.5em}f\hspace{0.5em}$};
		\node [style=none] (3) at (-2.25, -1.75) {};
		\node [style=none] (4) at (-1.25, -1.75) {};
		\node [style=none] (5) at (-3.5, -1.5) {};
		\node [style=none] (7) at (-1.75, -1.25) {};
		\node [style=none] (9) at (-2.625, 1) {};
		\node [style=box] (10) at (-3.125, 1.25) {$\hspace{0.5em}h\hspace{0.5em}$};
		\node [style=none] (11) at (-3.5, 1) {};
		\node [style=none] (12) at (-3.5, 1.5) {};
		\node [style=none] (13) at (-2.625, 1.5) {};
		\node [style=none] (14) at (-4, 2.5) {};
		\node [style=none] (15) at (-2.625, 1.5) {};
		\node [style=blackdot] (16) at (-2.625, 0) {};
		\node [style=none] (17) at (-1.75, 1.5) {};
		\node [style=none] (18) at (-1.75, 0.75) {};
		\node [style=none] (19) at (-2.175, 2.5) {};
		\node [style=none] (23) at (-4.5, 1.5) {};
		\node [style=none] (24) at (-4.5, -3.25) {};
		\node [style=none] (25) at (-2.175, 3.5) {};
		\node [style=none] (26) at (-0.5, -0.25) {$=$};
		\node [style=none] (27) at (3.125, -2.5) {};
		\node [style=none] (28) at (4.75, -3) {};
		\node [style=box] (29) at (4.25, -1.5) {$\hspace{0.5em}f\hspace{0.5em}$};
		\node [style=none] (30) at (3.75, -1.75) {};
		\node [style=none] (31) at (4.75, -1.75) {};
		\node [style=none] (32) at (2.5, -1.5) {};
		\node [style=none] (33) at (4.25, -1.25) {};
		\node [style=none] (34) at (3.375, 0.25) {};
		\node [style=box] (35) at (2.875, 0.5) {$\hspace{0.5em}h\hspace{0.5em}$};
		\node [style=none] (36) at (2.5, 0.25) {};
		\node [style=none] (37) at (2.5, 0.75) {};
		\node [style=none] (38) at (3.375, 0.75) {};
		\node [style=none] (40) at (3.375, 0.75) {};
		\node [style=blackdot] (41) at (3.375, -0.5) {};
		\node [style=none] (42) at (4.25, 2) {};
		\node [style=none] (43) at (4.25, 0.25) {};
		\node [style=none] (44) at (3.25, 3) {};
		\node [style=none] (45) at (0.75, 0.75) {};
		\node [style=blackdot] (46) at (0.75, -2.5) {};
		\node [style=none] (47) at (3.25, 3.5) {};
		\node [style=none] (48) at (6, -0.25) {$\in \mathtt{I}$};
		\node [style=none] (49) at (2, 2) {};
		\node [style=none] (50) at (1.5, -3) {};
		\node [style=none] (51) at (2, 1.5) {};
		\node [style=none] (52) at (1.5, 0.75) {};
	\end{pgfonlayer}
	\begin{pgfonlayer}{edgelayer}
		\draw [style=morphism-edge, in=270, out=0] (0.center) to (3.center);
		\draw [style=morphism-edge] (4.center) to (1.center);
		\draw [style=morphism-edge, in=180, out=-90] (5.center) to (0.center);
		\draw [style=morphism-edge, in=270, out=90] (5.center) to (11.center);
		\draw [style=morphism-edge] (15.center) to (13.center);
		\draw [style=morphism-edge, in=360, out=90] (12.center) to (14.center);
		\draw [style=morphism-edge] (16) to (9.center);
		\draw [style=morphism-edge, in=90, out=-90] (17.center) to (18.center);
		\draw [style=morphism-edge, in=270, out=90] (7.center) to (18.center);
		\draw [style=morphism-edge, in=360, out=90] (17.center) to (19.center);
		\draw [style=morphism-edge, in=180, out=90] (15.center) to (19.center);
		\draw [style=morphism-edge, in=180, out=90] (23.center) to (14.center);
		\draw [style=morphism-edge] (23.center) to (24.center);
		\draw (19.center) to (25.center);
		\draw [style=morphism-edge, in=270, out=0] (27.center) to (30.center);
		\draw [style=morphism-edge] (31.center) to (28.center);
		\draw [style=morphism-edge, in=180, out=-90] (32.center) to (27.center);
		\draw [style=morphism-edge, in=270, out=90] (32.center) to (36.center);
		\draw [style=morphism-edge] (40.center) to (38.center);
		\draw [style=morphism-edge] (41) to (34.center);
		\draw [style=morphism-edge, in=90, out=-90] (42.center) to (43.center);
		\draw [style=morphism-edge, in=270, out=90] (33.center) to (43.center);
		\draw [style=morphism-edge, in=360, out=90] (42.center) to (44.center);
		\draw [style=morphism-edge] (45.center) to (46);
		\draw (44.center) to (47.center);
		\draw [style=morphism-edge, in=180, out=90] (45.center) to (49.center);
		\draw [style=morphism-edge, in=90, out=0] (49.center) to (40.center);
		\draw [style=morphism-edge, in=180, out=90] (49.center) to (44.center);
		\draw [style=morphism-edge, in=360, out=90] (37.center) to (51.center);
		\draw [style=morphism-edge, in=90, out=-180] (51.center) to (52.center);
		\draw [style=morphism-edge] (52.center) to (50.center);
	\end{pgfonlayer}
      \end{tikzpicture}
    \]

    Lastly, for $Q'' \in \cp$, denoting the morphism in the upper dotted box of the right-hand side by $\widehat{Q'' \triangleright f}$, the equation below shows that $Q'' \triangleright \mathbb{R}(f) = \mathbb{R}(Q''\triangleright f)$, so $Q'' \triangleright \mathbb{R}(f) \in \mathbb{R}(\mathtt{I})$.
    \[\begin{tikzpicture}[tikzfig]
	\begin{pgfonlayer}{nodelayer}
		\node [style=none] (0) at (-0.25, -3.25) {};
		\node [style=none] (1) at (0.75, -3.25) {};
		\node [style=none] (4) at (0.75, 0.5) {};
		\node [style=none] (7) at (-0.25, -1.25) {};
		\node [style=none] (8) at (0.75, -1.25) {};
		\node [style=box] (9) at (-0.5, 0.25) {$\hspace{0.1em}f\hspace{0.1em}$};
		\node [style=none] (10) at (-0.25, 0) {};
		\node [style=none] (11) at (-0.75, 0) {};
		\node [style=none] (12) at (-1.25, -1) {};
		\node [style=none] (13) at (-1.75, 0) {};
		\node [style=none] (14) at (-0.5, 0.5) {};
		\node [style=none] (15) at (0.125, 1.75) {};
		\node [style=none] (16) at (0.125, 3.25) {};
		\node [style=none] (17) at (-1.75, 3.25) {};
		\node [style=none] (18) at (-2.5, -3.25) {};
		\node [style=none] (19) at (-2.5, 3.25) {};
		\node [style=none] (20) at (-2.5, -3.5) {\tiny$Q''$};
		\node [style=none] (21) at (1.75, -0.75) {$=$};
		\node [style=none] (22) at (7, -3.25) {};
		\node [style=none] (23) at (8, -3.25) {};
		\node [style=none] (25) at (7, 1.75) {};
		\node [style=none] (26) at (7, -1.25) {};
		\node [style=none] (27) at (8, -1.25) {};
		\node [style=none] (29) at (8, -1.25) {};
		\node [style=box] (30) at (5.75, 1.25) {$\hspace{0.1em}f\hspace{0.1em}$};
		\node [style=none] (32) at (4.625, -0.5) {};
		\node [style=none] (33) at (4.125, -2.25) {};
		\node [style=none] (34) at (3.625, -0.5) {};
		\node [style=none] (35) at (5.75, 1.75) {};
		\node [style=none] (36) at (6.375, 2.5) {};
		\node [style=none] (37) at (6.375, 3.25) {};
		\node [style=none] (38) at (3.625, 2.25) {};
		\node [style=none] (39) at (3, -1.25) {};
		\node [style=none] (40) at (3, 2.25) {};
		\node [style=none] (41) at (4.125, -2.75) {};
		\node [style=none] (42) at (5.25, -1.25) {};
		\node [style=none] (43) at (5.75, 0) {};
		\node [style=none] (44) at (6.25, -1.25) {};
		\node [style=none] (45) at (6.25, -3.25) {};
		\node [style=none] (46) at (9.5, -0.75) {$\in \mathtt{I}$};
		\node [style=none] (47) at (5.5, 1) {};
		\node [style=none] (48) at (6, 1) {};
		\node [style=none] (49) at (2.5, -3) {};
		\node [style=none] (50) at (8.5, -3) {};
		\node [style=none] (51) at (2.5, -1.75) {};
		\node [style=none] (52) at (8.5, -1.75) {};
		\node [style=none] (53) at (4.25, -0.5) {};
		\node [style=none] (54) at (8.25, -0.5) {};
		\node [style=none] (55) at (8.25, 2.25) {};
		\node [style=none] (56) at (4.25, 2.25) {};
	\end{pgfonlayer}
	\begin{pgfonlayer}{edgelayer}
		\draw [style=morphism-edge] (8.center) to (4.center);
		\draw [style=morphism-edge] (7.center) to (10.center);
		\draw [style=morphism-edge, in=180, out=90] (14.center) to (15.center);
		\draw [style=morphism-edge, in=90, out=0] (15.center) to (4.center);
		\draw [style=morphism-edge] (16.center) to (15.center);
		\draw [style=morphism-edge, in=0, out=-90] (11.center) to (12.center);
		\draw [style=morphism-edge, in=270, out=180] (12.center) to (13.center);
		\draw [style=morphism-edge] (17.center) to (13.center);
		\draw [style=morphism-edge] (19.center) to (18.center);
		\draw [style=morphism-edge] (22.center) to (26.center);
		\draw [style=morphism-edge] (27.center) to (23.center);
		\draw [style=morphism-edge, in=270, out=90] (29.center) to (25.center);
		\draw [style=morphism-edge, in=180, out=90] (35.center) to (36.center);
		\draw [style=morphism-edge, in=90, out=0] (36.center) to (25.center);
		\draw [style=morphism-edge] (37.center) to (36.center);
		\draw [style=morphism-edge, in=0, out=-90] (32.center) to (33.center);
		\draw [style=morphism-edge, in=270, out=180] (33.center) to (34.center);
		\draw [style=morphism-edge] (38.center) to (34.center);
		\draw [style=morphism-edge] (40.center) to (39.center);
		\draw [style=morphism-edge, in=0, out=-90] (42.center) to (41.center);
		\draw [style=morphism-edge, in=270, out=180] (41.center) to (39.center);
		\draw [style=morphism-edge, in=180, out=90] (42.center) to (43.center);
		\draw [style=morphism-edge, in=90, out=0] (43.center) to (44.center);
		\draw [style=morphism-edge] (45.center) to (44.center);
		\draw [style=morphism-edge] (0.center) to (7.center);
		\draw [style=morphism-edge] (8.center) to (1.center);
		\draw [style=morphism-edge, in=270, out=90] (32.center) to (47.center);
		\draw [style=morphism-edge, in=90, out=-90] (48.center) to (26.center);
		\draw [style=ddd] (51.center) to (52.center);
		\draw [style=ddd] (52.center) to (50.center);
		\draw [style=ddd] (50.center) to (49.center);
		\draw [style=ddd] (49.center) to (51.center);
		\draw [style=ddd] (56.center) to (53.center);
		\draw [style=ddd] (56.center) to (55.center);
		\draw [style=ddd] (55.center) to (54.center);
		\draw [style=ddd] (54.center) to (53.center);
	\end{pgfonlayer}
      \end{tikzpicture}
    \]
\end{proof}

\subsection{The bijection}

\begin{proposition}
    For $\mathfrak{J}$ a $\cp$-stable ideal in $\pfa$, we have $\mathbb{RS}(\mathfrak{J}) = \mathfrak{J}$.
\end{proposition}
\begin{proof}
  We have
  \begin{align*}
    &\mathbb{R}\mathbb{S}(\mathfrak{J})(P,Q)
    = \left\{
    \scalebox{0.8}{\,\begin{tikzpicture}[tikzfig]
	\begin{pgfonlayer}{nodelayer}
		\node [style=none] (4) at (-0.75, -1.25) {};
		\node [style=none] (5) at (0.25, -1.5) {};
		\node [style=box] (6) at (0, -0.25) {\scriptsize$\hspace{0.2em}f\hspace{0.2em}$};
		\node [style=none] (7) at (0.5, 1) {};
		\node [style=none] (9) at (-1.25, -0.5) {};
		\node [style=none] (10) at (1, -0.25) {};
		\node [style=none] (11) at (-0.25, -0.5) {};
		\node [style=none] (12) at (0.25, -0.5) {};
		\node [style=none] (13) at (1, -1.5) {};
		\node [style=none] (14) at (0.5, 1.75) {};
		\node [style=none] (15) at (-1.25, 1.75) {};
	\end{pgfonlayer}
	\begin{pgfonlayer}{edgelayer}
		\draw [style=morphism-edge, in=360, out=90] (10.center) to (7.center);
		\draw [style=morphism-edge, in=180, out=90] (6) to (7.center);
		\draw [style=morphism-edge] (12.center) to (5.center);
		\draw [style=morphism-edge, in=0, out=-90] (11.center) to (4.center);
		\draw [style=morphism-edge, in=270, out=180] (4.center) to (9.center);
		\draw [style=morphism-edge] (13.center) to (10.center);
		\draw [style=morphism-edge] (15.center) to (9.center);
		\draw [style=morphism-edge] (14.center) to (7.center);
	\end{pgfonlayer}
      \end{tikzpicture}%
      }
    \;\Bigg|\;
    f \in \mathbb{S}(\mathfrak{J})(\ld{Q} \otimes P)
    \right\} \\
    &= \left\{
    \scalebox{0.8}{\,\begin{tikzpicture}[tikzfig]
	\begin{pgfonlayer}{nodelayer}
		\node [style=none] (4) at (-0.75, -1.25) {};
		\node [style=none] (5) at (0.25, -1.5) {};
		\node [style=box] (6) at (0, -0.25) {\scriptsize$\hspace{0.2em}f\hspace{0.2em}$};
		\node [style=none] (7) at (0.5, 1) {};
		\node [style=none] (9) at (-1.25, -0.5) {};
		\node [style=none] (10) at (1, -0.25) {};
		\node [style=none] (11) at (-0.25, -0.5) {};
		\node [style=none] (12) at (0.25, -0.5) {};
		\node [style=none] (13) at (1, -1.5) {};
		\node [style=none] (14) at (0.5, 1.75) {};
		\node [style=none] (15) at (-1.25, 1.75) {};
	\end{pgfonlayer}
	\begin{pgfonlayer}{edgelayer}
		\draw [style=morphism-edge, in=360, out=90] (10.center) to (7.center);
		\draw [style=morphism-edge, in=180, out=90] (6) to (7.center);
		\draw [style=morphism-edge] (12.center) to (5.center);
		\draw [style=morphism-edge, in=0, out=-90] (11.center) to (4.center);
		\draw [style=morphism-edge, in=270, out=180] (4.center) to (9.center);
		\draw [style=morphism-edge] (13.center) to (10.center);
		\draw [style=morphism-edge] (15.center) to (9.center);
		\draw [style=morphism-edge] (14.center) to (7.center);
	\end{pgfonlayer}
      \end{tikzpicture}%
      }
    \;\Bigg|\;
    f=\scalebox{0.8}{\begin{tikzpicture}[tikzfig]
	\begin{pgfonlayer}{nodelayer}
		\node [style=none] (0) at (-0.5, -1.75) {};
		\node [style=box] (1) at (0, -0.25) {$\hspace{0.5em}g\hspace{0.5em}$};
		\node [style=none] (2) at (-0.5, -0.5) {};
		\node [style=none] (3) at (0, -0.5) {};
		\node [style=none] (4) at (0.5, -0.5) {};
		\node [style=none] (5) at (-0.5, 0) {};
		\node [style=none] (6) at (0.5, 0) {};
		\node [style=none] (7) at (0, -1.75) {};
		\node [style=blackdot] (8) at (0.5, -1.25) {};
		\node [style=blackdot] (9) at (-0.5, 0.75) {};
		\node [style=none] (10) at (0.5, 1.75) {};
	\end{pgfonlayer}
	\begin{pgfonlayer}{edgelayer}
		\draw [style=morphism-edge] (10.center) to (6.center);
		\draw [style=morphism-edge] (5.center) to (9);
		\draw [style=morphism-edge] (0.center) to (2.center);
		\draw [style=morphism-edge] (3.center) to (7.center);
		\draw [style=morphism-edge] (4.center) to (8);
	\end{pgfonlayer}
      \end{tikzpicture}%
      }
    \ ,\ g \in \mathbb{}\mathfrak{J}(\ld{Q} \otimes P, Q^{\mathbb{1}}),
    \ \scalebox{0.8}{\begin{tikzpicture}[tikzfig]
	\begin{pgfonlayer}{nodelayer}
		\node [style=none] (0) at (0, -0.5) {};
		\node [style=blackdot] (1) at (0, 0.5) {};
	\end{pgfonlayer}
	\begin{pgfonlayer}{edgelayer}
		\draw [style=morphism-edge] (0.center) to (1);
	\end{pgfonlayer}
      \end{tikzpicture}%
      }\colon Q^{\mathbb{1}} \to \mathbb{1}
    \ \text{is epic}
    \right\}\\
    &= \left\{
    \scalebox{0.8}{\tikzstyle{every picture}=[tikzfig]\begin{tikzpicture}
	\begin{pgfonlayer}{nodelayer}
		\node [style=box] (1) at (0, -0.25) {$\hspace{0.75em}g\hspace{0.75em}$};
		\node [style=none] (2) at (-0.5, -0.5) {};
		\node [style=none] (3) at (0, -0.5) {};
		\node [style=none] (4) at (0.5, -0.5) {};
		\node [style=none] (5) at (-0.5, 0) {};
		\node [style=none] (6) at (0.5, 0) {};
		\node [style=none] (7) at (0, -1.75) {};
		\node [style=blackdot] (8) at (0.5, -1.25) {};
		\node [style=blackdot] (9) at (-0.5, 0.75) {};
		\node [style=none] (10) at (1, 1.75) {};
		\node [style=none] (11) at (1.5, 0) {};
		\node [style=none] (12) at (1, 1) {};
		\node [style=none] (13) at (1.5, -1.75) {};
		\node [style=none] (14) at (-1, -1.25) {};
		\node [style=none] (15) at (-1.5, -0.5) {};
		\node [style=none] (16) at (-1.5, 1.75) {};
	\end{pgfonlayer}
	\begin{pgfonlayer}{edgelayer}
		\draw [style=morphism-edge] (5.center) to (9);
		\draw [style=morphism-edge] (3.center) to (7.center);
		\draw [style=morphism-edge] (4.center) to (8);
		\draw [style=morphism-edge, in=90, out=0] (12.center) to (11.center);
		\draw [style=morphism-edge, in=180, out=90] (6.center) to (12.center);
		\draw [style=morphism-edge] (12.center) to (10.center);
		\draw [style=morphism-edge] (13.center) to (11.center);
		\draw [style=morphism-edge, in=270, out=0] (14.center) to (2.center);
		\draw [style=morphism-edge, in=270, out=180] (14.center) to (15.center);
		\draw [style=morphism-edge] (16.center) to (15.center);
	\end{pgfonlayer}
      \end{tikzpicture}%
      }
    \;\Bigg|\;
    g \in \mathbb{}\mathfrak{J}(\ld{Q} \otimes P, Q^{\mathbb{1}})
    \subseteq \homa(\ld{Q} \otimes P \triangleright A, Q^{\mathbb{1}} \triangleright A)
    \right\}\\
    &= \left\{
    \scalebox{0.8}{\tikzstyle{every picture}=[tikzfig]\begin{tikzpicture}
	\begin{pgfonlayer}{nodelayer}
		\node [style=box] (1) at (0, -0.25) {$\hspace{0.75em}g\hspace{0.75em}$};
		\node [style=none] (2) at (-0.5, -0.5) {};
		\node [style=none] (3) at (0, -0.5) {};
		\node [style=none] (5) at (-0.5, 0) {};
		\node [style=none] (6) at (0.5, 0) {};
		\node [style=none] (7) at (0, -1.75) {};
		\node [style=blackdot] (9) at (-0.5, 0.75) {};
		\node [style=none] (10) at (0.5, 1.75) {};
		\node [style=none] (11) at (0.5, -0.5) {};
		\node [style=none] (13) at (0.5, -1.75) {};
		\node [style=none] (14) at (-1, -1.25) {};
		\node [style=none] (15) at (-1.5, -0.5) {};
		\node [style=none] (16) at (-1.5, 1.75) {};
	\end{pgfonlayer}
	\begin{pgfonlayer}{edgelayer}
		\draw [style=morphism-edge] (5.center) to (9);
		\draw [style=morphism-edge] (3.center) to (7.center);
		\draw [style=morphism-edge] (13.center) to (11.center);
		\draw [style=morphism-edge, in=270, out=0] (14.center) to (2.center);
		\draw [style=morphism-edge, in=270, out=180] (14.center) to (15.center);
		\draw [style=morphism-edge] (16.center) to (15.center);
		\draw [style=morphism-edge] (10.center) to (6.center);
	\end{pgfonlayer}
      \end{tikzpicture}%
      }
    \;\Bigg|\;
    g \in \mathbb{}\mathfrak{J}(\ld{Q} \otimes P, Q^{\mathbb{1}})
    \right\}.
  \end{align*}
  Since $\mathfrak{J}$ is a $\cp$-stable ideal, one obtains that
  \(\ \begin{tikzpicture}[tikzfig]
	\begin{pgfonlayer}{nodelayer}
		\node [style=box] (1) at (2, -0.25) {$\hspace{0.75em}\hspace{0.75em}$};
		\node [style=none] (2) at (1.5, -0.5) {};
		\node [style=none] (3) at (2, -0.5) {};
		\node [style=none] (5) at (1.5, 0) {};
		\node [style=none] (6) at (2.5, 0) {};
		\node [style=none] (7) at (2, -1.75) {};
		\node [style=blackdot] (9) at (1.5, 0.75) {};
		\node [style=none] (10) at (2.5, 1.75) {};
		\node [style=none] (11) at (2.5, -0.5) {};
		\node [style=none] (13) at (2.5, -1.75) {};
		\node [style=none] (14) at (1, -1.25) {};
		\node [style=none] (15) at (0.5, -0.5) {};
		\node [style=none] (16) at (0.5, 1.75) {};
		\node [style=box] (17) at (-3.25, 0) {$\hspace{0.75em}\hspace{0.75em}$};
		\node [style=none] (19) at (-3.5, -0.25) {};
		\node [style=none] (20) at (-3.5, 0.25) {};
		\node [style=none] (21) at (-3, 0.25) {};
		\node [style=none] (22) at (-3.5, -1.5) {};
		\node [style=none] (24) at (-3, 1.5) {};
		\node [style=none] (25) at (-3, -0.25) {};
		\node [style=none] (26) at (-3, -1.5) {};
		\node [style=none] (27) at (-3.5, 1.5) {};
		\node [style=none] (28) at (-1.75, 0) {\scriptsize$\in \mathfrak{J}$};
		\node [style=none] (29) at (-0.5, 0) {\scriptsize$\implies$};
		\node [style=none] (30) at (3.75, 0) {\scriptsize$\in \mathfrak{J}$};
	\end{pgfonlayer}
	\begin{pgfonlayer}{edgelayer}
		\draw [style=morphism-edge] (5.center) to (9);
		\draw [style=morphism-edge] (3.center) to (7.center);
		\draw [style=morphism-edge] (13.center) to (11.center);
		\draw [style=morphism-edge, in=270, out=0] (14.center) to (2.center);
		\draw [style=morphism-edge, in=270, out=180] (14.center) to (15.center);
		\draw [style=morphism-edge] (16.center) to (15.center);
		\draw [style=morphism-edge] (10.center) to (6.center);
		\draw [style=morphism-edge] (19.center) to (22.center);
		\draw [style=morphism-edge] (26.center) to (25.center);
		\draw [style=morphism-edge] (24.center) to (21.center);
		\draw [style=morphism-edge] (27.center) to (20.center);
	\end{pgfonlayer}
      \end{tikzpicture}
      \),
    and hence $\mathbb{R}\mathbb{S}(\mathfrak{J})\subseteq\mathfrak{J}$.

    The inclusion $\mathfrak{J} \subseteq \mathbb{R}\mathbb{S}(\mathfrak{J})$ follows from the fact that
    \[
    \begin{tikzpicture}[tikzfig]
	\begin{pgfonlayer}{nodelayer}
		\node [style=box] (1) at (2.25, -0.75) {$\hspace{0.75em}\hspace{0.75em}$};
		\node [style=none] (2) at (2, 0.25) {};
		\node [style=none] (5) at (2, -0.5) {};
		\node [style=none] (6) at (2.5, -0.5) {};
		\node [style=none] (10) at (2.5, 0.25) {};
		\node [style=none] (11) at (2, -1) {};
		\node [style=none] (13) at (2, -2.25) {};
		\node [style=none] (14) at (1.5, 1) {};
		\node [style=none] (15) at (1, 0.25) {};
		\node [style=box] (17) at (-2.5, 0) {$\hspace{0.75em}\hspace{0.75em}$};
		\node [style=none] (19) at (-2.75, -0.25) {};
		\node [style=none] (20) at (-2.75, 0.25) {};
		\node [style=none] (21) at (-2.25, 0.25) {};
		\node [style=none] (22) at (-2.75, -2) {};
		\node [style=none] (24) at (-2.25, 1.75) {};
		\node [style=none] (25) at (-2.25, -0.25) {};
		\node [style=none] (26) at (-2.25, -2) {};
		\node [style=none] (27) at (-2.75, 1.75) {};
		\node [style=none] (29) at (-1, 0) {\scriptsize$=$};
		\node [style=none] (31) at (1, -1) {};
		\node [style=none] (32) at (0.5, -1.75) {};
		\node [style=none] (33) at (0, -1) {};
		\node [style=none] (34) at (0, 1.75) {};
		\node [style=none] (35) at (2.5, -1) {};
		\node [style=blackdot] (36) at (2.5, -1.75) {};
		\node [style=none] (37) at (3.5, 0.25) {};
		\node [style=none] (38) at (3, 1) {};
		\node [style=none] (39) at (3.5, -2.25) {};
		\node [style=none] (40) at (3, 1.75) {};
	\end{pgfonlayer}
	\begin{pgfonlayer}{edgelayer}
		\draw [style=morphism-edge] (13.center) to (11.center);
		\draw [style=morphism-edge, in=-270, out=0] (14.center) to (2.center);
		\draw [style=morphism-edge, in=-270, out=-180] (14.center) to (15.center);
		\draw [style=morphism-edge] (10.center) to (6.center);
		\draw [style=morphism-edge] (19.center) to (22.center);
		\draw [style=morphism-edge] (26.center) to (25.center);
		\draw [style=morphism-edge] (24.center) to (21.center);
		\draw [style=morphism-edge] (27.center) to (20.center);
		\draw [style=morphism-edge, in=270, out=0] (32.center) to (31.center);
		\draw [style=morphism-edge, in=270, out=180] (32.center) to (33.center);
		\draw [style=morphism-edge] (33.center) to (34.center);
		\draw [style=morphism-edge] (31.center) to (15.center);
		\draw [style=morphism-edge] (2.center) to (5.center);
		\draw [style=morphism-edge] (36) to (35.center);
		\draw [style=morphism-edge] (39.center) to (37.center);
		\draw [style=morphism-edge, in=180, out=90] (10.center) to (38.center);
		\draw [style=morphism-edge, in=90, out=0] (38.center) to (37.center);
		\draw [style=morphism-edge] (40.center) to (38.center);
	\end{pgfonlayer}
    \end{tikzpicture}
    \]
\end{proof}

\begin{proposition}
    For $\mathtt{I}$ a representable ideal in $\mathcal{C}(\blank,A)$, we have $\mathbb{SR}(\mathtt{I}) = \mathtt{I}$.
\end{proposition}

\begin{proof}
    One obtains that $\mathbb{SR}(\mathtt{I}) \subseteq \mathtt{I}$ by the following calculation,
    together with the fact that
    \(\scalebox{0.6}{\begin{tikzpicture}[tikzfig]
	\begin{pgfonlayer}{nodelayer}
		\node [style=box] (17) at (0.25, 0) {$\hspace{0.25em}g\hspace{0.25em}$};
		\node [style=none] (19) at (0.5, -0.25) {};
		\node [style=none] (22) at (0.5, -2) {};
		\node [style=none] (25) at (0, -0.25) {};
		\node [style=none] (26) at (-0.5, -1) {};
		\node [style=blackdot] (27) at (-1, -0.25) {};
		\node [style=none] (30) at (0.25, 1.5) {};
		\node [style=box] (31) at (-5, 0) {$\hspace{0.25em}g\hspace{0.25em}$};
		\node [style=none] (32) at (-4.75, -0.25) {};
		\node [style=none] (33) at (-4.75, -2) {};
		\node [style=none] (34) at (-5.25, -0.25) {};
		\node [style=none] (35) at (-5.25, -2) {};
		\node [style=none] (37) at (-5, 1.5) {};
		\node [style=none] (38) at (-2.6, 0) {$\in\mathtt{I}\implies$};
		\node [style=none] (39) at (1.7, 0) {$\in\mathtt{I}$};
	\end{pgfonlayer}
	\begin{pgfonlayer}{edgelayer}
		\draw [style=morphism-edge] (19.center) to (22.center);
		\draw [style=morphism-edge, in=0, out=-90] (25.center) to (26.center);
		\draw [style=morphism-edge, in=270, out=180] (26.center) to (27);
		\draw [style=morphism-edge] (17) to (30.center);
		\draw [style=morphism-edge] (32.center) to (33.center);
		\draw [style=morphism-edge, in=90, out=-90] (34.center) to (35.center);
		\draw [style=morphism-edge] (31) to (37.center);
	\end{pgfonlayer}
      \end{tikzpicture}%
    }\),
    as $\mathtt{I}$ is a subpresheaf of $\cat{C}(\blank, A)$:
    \begin{align*}
        &\mathbb{SR}(\mathtt{I})(P) =
        \{\, (\pi \otimes A) \circ f^1 \mid f \in \mathbb{R}(\mathtt{I})(P,Q), \pi\colon Q \to \mathbb{1} \,\}\\
        &= \left\{ \scalebox{0.8}{\begin{tikzpicture}[tikzfig]
	\begin{pgfonlayer}{nodelayer}
		\node [style=box] (17) at (0, 0) {$\hspace{0.75em}f\hspace{0.75em}$};
		\node [style=none] (19) at (-0.25, -0.25) {};
		\node [style=none] (20) at (-0.25, 0.25) {};
		\node [style=none] (21) at (0.25, 0.25) {};
		\node [style=none] (22) at (-0.25, -2) {};
		\node [style=none] (24) at (0.25, 1.75) {};
		\node [style=none] (25) at (0.25, -0.25) {};
		\node [style=blackdot] (26) at (0.25, -1) {};
		\node [style=blackdot] (27) at (-0.25, 1) {};
	\end{pgfonlayer}
	\begin{pgfonlayer}{edgelayer}
		\draw [style=morphism-edge] (19.center) to (22.center);
		\draw [style=morphism-edge] (26) to (25.center);
		\draw [style=morphism-edge] (24.center) to (21.center);
		\draw [style=morphism-edge] (27) to (20.center);
	\end{pgfonlayer}
      \end{tikzpicture}%
          }
    \;\Bigg|\;
    f=\scalebox{0.8}{\begin{tikzpicture}[tikzfig]
	\begin{pgfonlayer}{nodelayer}
		\node [style=box] (17) at (0, 0) {$\hspace{0.25em}g\hspace{0.25em}$};
		\node [style=none] (19) at (0.25, -0.25) {};
		\node [style=none] (22) at (0.25, -2) {};
		\node [style=none] (24) at (0.5, 1) {};
		\node [style=none] (25) at (-0.25, -0.25) {};
		\node [style=none] (26) at (-0.75, -1) {};
		\node [style=none] (27) at (-1.25, -0.25) {};
		\node [style=none] (28) at (-1.25, 1.75) {};
		\node [style=none] (29) at (1, 0) {};
		\node [style=none] (30) at (0.5, 1.75) {};
		\node [style=none] (31) at (1, -2) {};
	\end{pgfonlayer}
	\begin{pgfonlayer}{edgelayer}
		\draw [style=morphism-edge] (19.center) to (22.center);
		\draw [style=morphism-edge, in=0, out=-90] (25.center) to (26.center);
		\draw [style=morphism-edge, in=270, out=180] (26.center) to (27.center);
		\draw [style=morphism-edge] (28.center) to (27.center);
		\draw [style=morphism-edge, in=180, out=90] (17) to (24.center);
		\draw [style=morphism-edge, in=360, out=90] (29.center) to (24.center);
		\draw [style=morphism-edge] (31.center) to (29.center);
		\draw [style=morphism-edge] (30.center) to (24.center);
	\end{pgfonlayer}
      \end{tikzpicture}%
          }
    \,,\ g \in \mathtt{I}(\ld{Q} \otimes P)
    \right\}
    = \left\{ \scalebox{0.8}{\begin{tikzpicture}[tikzfig]
	\begin{pgfonlayer}{nodelayer}
		\node [style=box] (17) at (0, 0) {$\hspace{0.25em}g\hspace{0.25em}$};
		\node [style=none] (19) at (0.25, -0.25) {};
		\node [style=none] (22) at (0.25, -2) {};
		\node [style=none] (24) at (0.5, 1) {};
		\node [style=none] (25) at (-0.25, -0.25) {};
		\node [style=none] (26) at (-0.75, -1) {};
		\node [style=none] (27) at (-1.25, -0.25) {};
		\node [style=blackdot] (28) at (-1.25, 1) {};
		\node [style=none] (29) at (1, 0) {};
		\node [style=none] (30) at (0.5, 1.75) {};
		\node [style=blackdot] (31) at (1, -1.25) {};
	\end{pgfonlayer}
	\begin{pgfonlayer}{edgelayer}
		\draw [style=morphism-edge] (19.center) to (22.center);
		\draw [style=morphism-edge, in=0, out=-90] (25.center) to (26.center);
		\draw [style=morphism-edge, in=270, out=180] (26.center) to (27.center);
		\draw [style=morphism-edge] (28) to (27.center);
		\draw [style=morphism-edge, in=180, out=90] (17) to (24.center);
		\draw [style=morphism-edge, in=360, out=90] (29.center) to (24.center);
		\draw [style=morphism-edge] (31) to (29.center);
		\draw [style=morphism-edge] (30.center) to (24.center);
	\end{pgfonlayer}
      \end{tikzpicture}%
          }
    \;\Bigg|\;
    g \in \mathtt{I}(\ld{Q} \otimes P)
    \right\}\\
    &= \left\{ \scalebox{0.8}{\begin{tikzpicture}[tikzfig]
	\begin{pgfonlayer}{nodelayer}
		\node [style=box] (17) at (0, 0) {$\hspace{0.25em}g\hspace{0.25em}$};
		\node [style=none] (19) at (0.25, -0.25) {};
		\node [style=none] (22) at (0.25, -2) {};
		\node [style=none] (25) at (-0.25, -0.25) {};
		\node [style=none] (26) at (-0.75, -1) {};
		\node [style=blackdot] (27) at (-1.25, -0.25) {};
		\node [style=none] (30) at (0, 1.5) {};
	\end{pgfonlayer}
	\begin{pgfonlayer}{edgelayer}
		\draw [style=morphism-edge] (19.center) to (22.center);
		\draw [style=morphism-edge, in=0, out=-90] (25.center) to (26.center);
		\draw [style=morphism-edge, in=270, out=180] (26.center) to (27);
		\draw [style=morphism-edge] (17) to (30.center);
	\end{pgfonlayer}
      \end{tikzpicture}%
      }
    =
    g \circ (\pi \otimes \ld{Q} \otimes P) \circ (\coev_{{Q}} \otimes P)
    \right\}.
    \end{align*}

    For the other direction, suppose that $g\colon P \to A \in \mathtt{I}$.
    Using the above characterisation of $\mathbb{SR}(\mathtt{I})$,
    we have that
    \[
    \begin{tikzpicture}[tikzfig]
	\begin{pgfonlayer}{nodelayer}
		\node [style=box] (31) at (-2, -0.25) {\scriptsize$\hspace{0.25em}g\hspace{0.25em}$};
		\node [style=none] (32) at (-2, -0.5) {};
		\node [style=none] (33) at (-2, -2.25) {};
		\node [style=none] (37) at (-2, 1.75) {};
		\node [style=none] (38) at (-0.9, -0.25) {\scriptsize$=$};
		\node [style=box] (39) at (1.75, 0.25) {\scriptsize$\hspace{0.25em}g\hspace{0.25em}$};
		\node [style=none] (40) at (1.75, 0) {};
		\node [style=none] (41) at (1.25, -0.75) {};
		\node [style=none] (42) at (1.75, 1.75) {};
		\node [style=none] (43) at (3.5, -0.5) {};
		\node [style=none] (44) at (2.5, -0.5) {};
		\node [style=none] (45) at (3, 0.5) {};
		\node [style=none] (46) at (0.75, 0) {};
		\node [style=none] (47) at (-0.25, 0.75) {};
		\node [style=none] (48) at (0.75, 0.75) {};
		\node [style=none] (49) at (0.25, 1.5) {};
		\node [style=none] (50) at (1.25, -1.5) {};
		\node [style=none] (51) at (3.5, -2.25) {};
	\end{pgfonlayer}
	\begin{pgfonlayer}{edgelayer}
		\draw [style=morphism-edge] (32.center) to (33.center);
		\draw [style=morphism-edge] (31) to (37.center);
		\draw [style=morphism-edge, in=0, out=-90] (40.center) to (41.center);
		\draw [style=morphism-edge] (39) to (42.center);
		\draw [style=morphism-edge, in=180, out=90] (44.center) to (45.center);
		\draw [style=morphism-edge, in=360, out=90] (43.center) to (45.center);
		\draw [style=morphism-edge, in=180, out=-90] (46.center) to (41.center);
		\draw [style=morphism-edge] (46.center) to (48.center);
		\draw [style=morphism-edge, in=90, out=-180] (49.center) to (47.center);
		\draw [style=morphism-edge, in=90, out=0] (49.center) to (48.center);
		\draw [style=morphism-edge, in=270, out=0] (50.center) to (44.center);
		\draw [style=morphism-edge, in=270, out=180] (50.center) to (47.center);
		\draw [style=morphism-edge] (51.center) to (43.center);
	\end{pgfonlayer}
    \end{tikzpicture}
    \]
    Since $\ev_{{P}}$ is an epimorphism
    and
    $g \otimes \ev_{{P}} \in \mathtt{I}$, as $\mathtt{I}$ is a subpresheaf of $\cat{C}(\blank, A)$,
    we have that $g \in \mathbb{SR}(\mathtt{I})$,
    hence $\mathtt{I} \subseteq \mathbb{SR}(\mathtt{I})$.
\end{proof}

We should additionally observe that the ideal correspondence we establish is a lattice isomorphism. Given representable ideals $\mathtt{I,J}$ in $\mathcal{C}(\blank,A)$, we write $\mathtt{I} \subseteq \mathtt{J}$ if $\mathtt{I}(P) \subseteq \mathtt{J}(P)$ for all $P \in \cp$. By \autoref{lexinj}, this is equivalent to $\mathtt{I}(V) \subseteq \mathtt{J}(V)$ for all $V \in \mathcal{C}$, which, since limits in presheaf categories are pointwise, is equivalent to $\mathtt{I}$ being a subobject of $\mathtt{J}$ in $\presh$. Finally, since the Yoneda embedding $\yo$ is fully faithful, this is equivalent to $I$ being a subobject of $J$, where $I,J$ are the objects of $\mathcal{C}$ representing $\mathtt{I,J}$.

The following is an immediate consequence of \autoref{RDefined} and \autoref{SDefined}:

\begin{lemma}
    Given representable ideals $\mathtt{I,J}$ in $\mathcal{C}(\blank,A)$ such that $\mathtt{I} \subseteq \mathtt{J}$, we have $\mathbb{R}(\mathtt{I}) \subseteq \mathbb{R}(\mathtt{J})$. Similarly, given $\cp$-stable ideals $\mathfrak{I,J}$ in $\pfa$ such that $\mathfrak{I} \subseteq \mathfrak{J}$, we have $\mathbb{S}(\mathfrak{I}) \subseteq \mathbb{S}(\mathfrak{J})$.
\end{lemma}

Since there is a lattice isomorphism between ideals in $A$ and representable ideals in $\mathcal{C}(\blank,A)$, we find the following:
\begin{theorem}\label{RSlattice}
 The maps $\mathbb{R}$ and $\mathbb{S}$ induce mutually inverse lattice isomorphisms between the lattice of ideals in $A$ and the lattice of $\,\cp$-stable ideals in $\pfa$.
\end{theorem}

\section{Nilpotent ideals}\label{sec:nilpotent}

\subsection{Products of ideals}
\begin{definition}\label{def:productideal}
   Let $\mathcal{D}$ be an abelian monoidal category whose tensor product is right exact in both variables. Let $(B,\mu_{B},\eta_{B})$ be an algebra object in $\mathcal{D}$ and let $I,J$ be ideals in $B$, with respective inclusion morphisms $\iota_{I}, \iota_{J}$ into $B$. We define the product ideal $IJ$ by
   \[
   IJ = \on{Im}(I \otimes J \xrightarrow{\iota_{I} \otimes \iota_{J}} B  \otimes B \xrightarrow{\mu_{B}} B).
   \]
   It is not difficult to verify that $IJ$ indeed defines an ideal.
\end{definition}

  In the case $\mathcal{D} = \presh$, for ideals $\euler{I,J}$ in $\mathcal{C}(\blank,A)$, the ideal $\euler{IJ}$ is given as the image of the map
  \[
  \euler{I} \oast \euler{J} \rightarrow \mathcal{C}(\blank,A) \oast \mathcal{C}(\blank,A) \cong \mathcal{C}(\blank,A\otimes A) \xrightarrow{\mathcal{C}(\blank,\mu)} \mathcal{C}(\blank,A).
  \]
 As such, following the description of Day convolution in \autoref{DayMultiply}, we find that
 \[
 \euler{IJ}(U) = \setj{\mu \circ (g \otimes h) \circ \varphi \mid g \in \euler{I}(V), f \in \euler{J}(V') \text{ and } \varphi \in \mathcal{C}(U,V\otimes V')}.
 \]
 However, there is a technical subtlety here. Let $I,J$ be ideals in $A$.
 Then in general the presheaves $\yo(I)\yo(J)$ and $\yo(IJ)$ need not be the same. Indeed, since $\yo$ does generally not preserve epimorphisms, we find the following:
 \[
 \begin{aligned}
 \yo(I)\yo(J)
   & = \mathcal{C}(\blank,I)\mathcal{C}(\blank,J)
     = \on{Im}(\mu_{\mathcal{C}(\blank,A)} \circ (\iota_{\mathcal{C}(\blank,I)}\oast \iota_{\mathcal{C}(\blank,J)})) \\
   & \cong \on{Coker}\on{Ker}(\yo(\mu_{A} \circ (\iota_{I} \otimes \iota_{J})))
   \cong \on{Coker}(\yo(\on{Ker}((\mu_{A} \circ (\iota_{I} \otimes \iota_{J}))))) \\
   & \not\cong \yo(\on{Coker}\on{Ker}(\mu_{A} \circ (\iota_{I} \otimes \iota_{J})))
      \cong \yo(\on{Im}(\mu_{A} \circ (\iota_{I} \otimes \iota_{J})))
      = \yo(IJ).
 \end{aligned}
 \]
 The first isomorphism above uses the identification $\on{Im}(f) = \on{Coker}(\on{Ker}(f)) = \on{Ker}(\on{Coker}(f))$ which can be made in an abelian category.
 However, the two presheaves agree on projectives:
 \begin{lemma}\label{lem:IJdescription}
     Let $I,J$ be ideals in $A$. Then ${\mathcal{C}(\blank,IJ)}_{|\cp}$ is given by
 \[
 \mathcal{C}(P,IJ) = \setj{\mu \circ (g \otimes h) \circ \varphi \mid g \in \mathcal{C}(Q,I), f \in \mathcal{C}(Q',J) \text{ and } \varphi \in \mathcal{C}(P,Q\otimes Q')}.
 \]
 \end{lemma}

\begin{proof}
   By~\cite[Theorem~4.27]{St2}, the composite functor $\mathcal{C} \xrightarrow{\yo} \presh \xrightarrow{{(\blank)}_{|\cp}} \preshcp$ is a monoidal equivalence, and so
   ${\yo(IJ)}_{|\cp} \cong {\yo(I)}_{|\cp}{\yo(J)}_{|\cp}$, and the latter is obtained by Day convolution, analogously to \autoref{DayMultiply}.
\end{proof}

Continuing the correspondence between ideals in $A$ and mixed subfunctors of $\mathcal{C}(\blank,A)$, we find:
\begin{definition}\label{multiplymixedsubfunctors}
    Let $\mathtt{I,J}$ be mixed subfunctors of $\mathcal{C}(\blank,A)$. We define the mixed subfunctor $\mathtt{IJ}$ by
    \[
    \mathtt{IJ}(P) = \setj{\mu \circ (g \otimes h) \circ \varphi \mid g \in \mathtt{I}(Q), f \in \mathtt{J}(Q') \text{ and } \varphi \in \mathcal{C}(P,Q\otimes Q')}
    \]
\end{definition}

The following is the special case of \autoref{def:productideal}, where $\mathcal{D} = \mathbf{Prof}(\mathcal{A,A})$ is the category of pro-functors and $A = \mathcal{A}(\blank,\blank)$:
\begin{definition}
    Let $\csym{I,J}$ be a pair of ideals in a $\Bbbk$-linear category $\mathcal{A}$. We define the ideal $\csym{IJ}$ by
    \[
    \csym{IJ}(X,Y) = \on{Span}(\setj{g \circ f \mid f \in \csym{J}(X,X'), g \in \csym{I}(X',Y) \text{ and } X' \in \mathcal{A}}),
    \]
    for $X,Y \in \mathcal{A}$.
\end{definition}

While we give a proof of \autoref{productstableideals} below, we remark that it also is a special case of \autoref{def:productideal}, specialised to the monoidal category of \emph{Tambara modules on $\mathcal{C}$} (see~\cite[Proposition~10.16]{St1}):
\begin{lemma}\label{productstableideals}
    If $\csym{I,J}$ are ideals in a module category $\mathcal{M}$ over a monoidal category $\mathcal{D}$, and both $\csym{I}$ and $\csym{J}$ are $\mathcal{D}$-stable, then $\csym{IJ}$ is $\mathcal{D}$-stable.
\end{lemma}

\begin{proof}
  Let $\sum_{i} \lambda_{i} g_{i} \circ f_{i} \in \csym{IJ}$, where
  $f_{1},\ldots,f_{n} \in \csym{J}, g_{1},\ldots,g_{n} \in \csym{I}$ and $\lambda_{1},\ldots,\lambda_{n} \in \Bbbk$. Further, let $V \in \mathcal{D}$.
  Then $V \triangleright \sum_{i} \lambda_{i} g_{i} \circ f_{i} = \sum_{i} \lambda_{i} (V \triangleright g_{i}) \circ (V \triangleright f_{i}) \in \csym{IJ}$, since $V \triangleright g_{i} \in \csym{I}$ and $V \triangleright f_{i} \in \csym{J}$ by $\mathcal{D}$-stability of $\csym{I}$ and $\csym{J}$ respectively.
\end{proof}

\begin{proposition}\label{Prop:Sprod}
 Let $\mathfrak{I,J}$ be a pair of $\cp$-stable ideals in $\pfa$. Then $\mathbb{S}(\mathfrak{I})\mathbb{S}(\mathfrak{J}) = \mathbb{S}(\mathfrak{IJ})$.
\end{proposition}

\begin{proof}
We first show the ``$\subseteq$'' direction. Let $\varphi \in \cp(P,Q\otimes Q'),g \in \mathfrak{I}(Q, Q'')$ and let $f \in \mathfrak{J}(Q', Q''')$, with epimorphisms $Q'' \xrightarrow{q} \mathbb{1}$ and $Q'''\xrightarrow{q'}\mathbb{1}$.
Let $\mathbb{S}(g) \defeq (q \otimes A) \circ g \circ (Q \otimes \eta)$ and similarly let $\mathbb{S}(f) \defeq (q' \otimes A) \circ f \circ (Q' \otimes A)$.
Following \autoref{multiplymixedsubfunctors}, a general element of $\mathbb{S}(\mathfrak{I})\mathbb{S}(\mathfrak{J})$ is of the form $\mu \circ \mathbb{S}(g) \otimes \mathbb{S}(f) \circ \varphi$.
   The ``$\subseteq$'' direction follows by
   
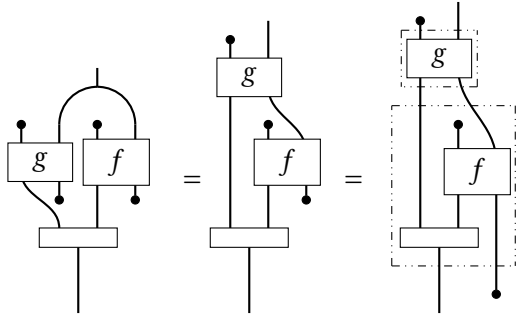
\[\begin{tikzpicture}[tikzfig]
	\begin{pgfonlayer}{nodelayer}
		\node [style=none] (0) at (-5, -3.5) {};
		\node [style=box] (1) at (-5, -1.5) {$\hspace{1em}\hspace{1em}$};
		\node [style=none] (3) at (-5, -1.75) {};
		\node [style=box] (5) at (-6, 0.5) {$\hspace{0.5em}g\hspace{0.5em}$};
		\node [style=none] (6) at (-6.5, 0.25) {};
		\node [style=none] (8) at (-5.5, -1.25) {};
		\node [style=none] (10) at (-6.5, 0.75) {};
		\node [style=none] (11) at (-5.5, 0.75) {};
		\node [style=blackdot] (12) at (-6.5, 1.5) {};
		\node [style=none] (13) at (-5.5, 1.5) {};
		\node [style=box] (14) at (-4, 0.5) {$\hspace{0.5em}f\hspace{0.5em}$};
		\node [style=none] (15) at (-4.5, 0.25) {};
		\node [style=none] (16) at (-4.5, -1.25) {};
		\node [style=none] (17) at (-4.5, 0.75) {};
		\node [style=none] (18) at (-3.5, 0.75) {};
		\node [style=blackdot] (19) at (-4.5, 1.5) {};
		\node [style=none] (20) at (-3.5, 1.5) {};
		\node [style=none] (21) at (-4.5, 2.5) {};
		\node [style=none] (22) at (-4.5, 3) {};
		\node [style=none] (23) at (-5.5, 0.25) {};
		\node [style=blackdot] (24) at (-5.5, -0.5) {};
		\node [style=none] (25) at (-3.5, 0.25) {};
		\node [style=blackdot] (26) at (-3.5, -0.5) {};
		\node [style=none] (27) at (-2, 0) {$=$};
		\node [style=none] (59) at (-0.5, -3.5) {};
		\node [style=box] (60) at (-0.5, -1.5) {$\hspace{1em}\hspace{1em}$};
		\node [style=none] (61) at (-0.5, -1.75) {};
		\node [style=box] (62) at (-0.5, 2.75) {$\hspace{0.5em}g\hspace{0.5em}$};
		\node [style=none] (63) at (-1, 2.5) {};
		\node [style=none] (64) at (-1, -1.25) {};
		\node [style=none] (65) at (-1, 3) {};
		\node [style=none] (66) at (0, 3) {};
		\node [style=blackdot] (67) at (-1, 3.75) {};
		\node [style=box] (68) at (0.5, 0.5) {$\hspace{0.5em}f\hspace{0.5em}$};
		\node [style=none] (69) at (0, 0.25) {};
		\node [style=none] (70) at (0, -1.25) {};
		\node [style=none] (71) at (0, 0.75) {};
		\node [style=none] (72) at (1, 0.75) {};
		\node [style=blackdot] (73) at (0, 1.5) {};
		\node [style=none] (74) at (0, 2.5) {};
		\node [style=none] (75) at (1, 0.25) {};
		\node [style=blackdot] (76) at (1, -0.5) {};
		\node [style=none] (78) at (0, 4.25) {};
		\node [style=none] (79) at (2.25, 0) {$=$};
		\node [style=none] (80) at (4.5, -3.5) {};
		\node [style=box] (81) at (4.5, -1.5) {$\hspace{1em}\hspace{1em}$};
		\node [style=none] (82) at (4.5, -1.75) {};
		\node [style=box] (83) at (4.5, 3.25) {$\hspace{0.5em}g\hspace{0.5em}$};
		\node [style=none] (84) at (4, 3) {};
		\node [style=none] (85) at (4, -1.25) {};
		\node [style=none] (86) at (4, 3.5) {};
		\node [style=none] (87) at (5, 3.5) {};
		\node [style=blackdot] (88) at (4, 4.25) {};
		\node [style=box] (89) at (5.5, 0.25) {$\hspace{0.5em}f\hspace{0.5em}$};
		\node [style=none] (90) at (5, 0) {};
		\node [style=none] (91) at (5, -1.25) {};
		\node [style=none] (92) at (5, 0.5) {};
		\node [style=none] (93) at (6, 0.5) {};
		\node [style=blackdot] (94) at (5, 1.5) {};
		\node [style=none] (95) at (5, 3) {};
		\node [style=none] (96) at (6, 0) {};
		\node [style=blackdot] (97) at (6, -3) {};
		\node [style=none] (98) at (5, 4.75) {};
		\node [style=none] (99) at (3.5, 4) {};
		\node [style=none] (100) at (5.5, 4) {};
		\node [style=none] (101) at (3.5, 2.5) {};
		\node [style=none] (102) at (5.5, 2.5) {};
		\node [style=none] (103) at (3.25, -2.25) {};
		\node [style=none] (104) at (6.5, -2.25) {};
		\node [style=none] (105) at (3.25, 2) {};
		\node [style=none] (106) at (6.5, 2) {};
	\end{pgfonlayer}
	\begin{pgfonlayer}{edgelayer}
		\draw [style=morphism-edge] (0.center) to (3.center);
		\draw [style=morphism-edge, in=270, out=90] (8.center) to (6.center);
		\draw [style=morphism-edge] (12) to (10.center);
		\draw [style=morphism-edge] (13.center) to (11.center);
		\draw [style=morphism-edge, in=-90, out=90] (16.center) to (15.center);
		\draw [style=morphism-edge] (19) to (17.center);
		\draw [style=morphism-edge] (20.center) to (18.center);
		\draw [style=morphism-edge, in=180, out=90] (13.center) to (21.center);
		\draw [style=morphism-edge, in=90, out=0] (21.center) to (20.center);
		\draw [style=morphism-edge] (21.center) to (22.center);
		\draw [style=morphism-edge] (26) to (25.center);
		\draw [style=morphism-edge] (24) to (23.center);
		\draw [style=morphism-edge] (59.center) to (61.center);
		\draw [style=morphism-edge, in=270, out=90] (64.center) to (63.center);
		\draw [style=morphism-edge] (67) to (65.center);
		\draw [style=morphism-edge, in=-90, out=90] (70.center) to (69.center);
		\draw [style=morphism-edge] (73) to (71.center);
		\draw [style=morphism-edge] (76) to (75.center);
		\draw [style=morphism-edge] (78.center) to (66.center);
		\draw [style=morphism-edge, in=270, out=90] (72.center) to (74.center);
		\draw [style=morphism-edge] (80.center) to (82.center);
		\draw [style=morphism-edge, in=270, out=90] (85.center) to (84.center);
		\draw [style=morphism-edge] (88) to (86.center);
		\draw [style=morphism-edge, in=-90, out=90] (91.center) to (90.center);
		\draw [style=morphism-edge] (94) to (92.center);
		\draw [style=morphism-edge] (97) to (96.center);
		\draw [style=morphism-edge] (98.center) to (87.center);
		\draw [style=morphism-edge, in=270, out=90] (93.center) to (95.center);
		\draw [style=ddd] (102.center) to (100.center);
		\draw [style=ddd] (99.center) to (100.center);
		\draw [style=ddd] (99.center) to (101.center);
		\draw [style=ddd] (102.center) to (101.center);
		\draw [style=ddd] (106.center) to (104.center);
		\draw [style=ddd] (103.center) to (104.center);
		\draw [style=ddd] (103.center) to (105.center);
		\draw [style=ddd] (105.center) to (106.center);
	\end{pgfonlayer}
      \end{tikzpicture}
    \]
    where the unlabelled solid box represents the morphism $\varphi$, the morphism enclosed by the upper dotted box
    lies in $\mathfrak{I}$, and the morphism enclosed by the lower dotted box lies in $\mathfrak{J}$ since $\mathfrak{J}$ is a $\cp$-stable ideal. Thus, the boxes indicate how to realise the right-hand side as a composite of a morphism in $\mathfrak{I}$ with a morphism in $\mathfrak{J}$, and hence an element of $\mathbb{S}(\mathfrak{IJ})$. The first equality holds since $g$ is a right $A$-module morphism.

    We now show the  ``$\supseteq$'' direction. If we assume $Q = Q'''$, and remove the assumption that there is an epimorphism $q'\colon Q''' \to \mathbb{1}$, we can write a general element of $\mathbb{S}(\mathfrak{IJ})$ as
    \[
    \mathbb{S}(g \circ f) \defeq (q \otimes A) \circ (g \circ f) \circ (Q' \otimes \eta).
    \]
    The ``$\supseteq$'' direction follows by
   
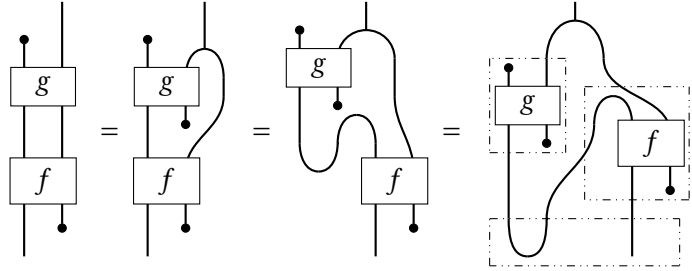
\[\begin{tikzpicture}[tikzfig]
	\begin{pgfonlayer}{nodelayer}
		\node [style=none] (0) at (-2.5, 0) {};
		\node [style=box] (1) at (-2, 2) {$\hspace{0.5em}f\hspace{0.5em}$};
		\node [style=blackdot] (2) at (-1.5, 0.75) {};
		\node [style=none] (3) at (-2.5, 1.75) {};
		\node [style=none] (4) at (-1.5, 1.75) {};
		\node [style=box] (5) at (-2, 4.5) {$\hspace{0.5em}g\hspace{0.5em}$};
		\node [style=none] (6) at (-2.5, 4.25) {};
		\node [style=none] (7) at (-1.5, 4.25) {};
		\node [style=none] (8) at (-2.5, 2.25) {};
		\node [style=none] (9) at (-1.5, 2.25) {};
		\node [style=none] (10) at (-2.5, 4.75) {};
		\node [style=none] (11) at (-1.5, 4.75) {};
		\node [style=blackdot] (12) at (-2.5, 5.75) {};
		\node [style=none] (13) at (-1.5, 6.75) {};
		\node [style=none] (14) at (-0.25, 3.25) {$=$};
		\node [style=none] (15) at (0.75, 0) {};
		\node [style=box] (16) at (1.25, 2) {$\hspace{0.5em}f\hspace{0.5em}$};
		\node [style=blackdot] (17) at (1.75, 0.75) {};
		\node [style=none] (18) at (0.75, 1.75) {};
		\node [style=none] (19) at (1.75, 1.75) {};
		\node [style=box] (20) at (1.25, 4.5) {$\hspace{0.5em}g\hspace{0.5em}$};
		\node [style=none] (21) at (0.75, 4.25) {};
		\node [style=none] (22) at (1.75, 4.25) {};
		\node [style=none] (23) at (0.75, 2.25) {};
		\node [style=blackdot] (24) at (1.75, 3.5) {};
		\node [style=none] (25) at (0.75, 4.75) {};
		\node [style=none] (26) at (1.75, 4.75) {};
		\node [style=blackdot] (27) at (0.75, 5.75) {};
		\node [style=none] (28) at (2.25, 5.5) {};
		\node [style=none] (29) at (1.75, 2.25) {};
		\node [style=none] (30) at (2.75, 4.25) {};
		\node [style=none] (31) at (2.75, 4.75) {};
		\node [style=none] (32) at (2.25, 6.75) {};
		\node [style=none] (33) at (3.75, 3.25) {$=$};
		\node [style=none] (34) at (6.75, 0) {};
		\node [style=box] (35) at (7.25, 2) {$\hspace{0.5em}f\hspace{0.5em}$};
		\node [style=blackdot] (36) at (7.75, 0.75) {};
		\node [style=none] (37) at (6.75, 1.75) {};
		\node [style=none] (38) at (7.75, 1.75) {};
		\node [style=box] (39) at (5.25, 5) {$\hspace{0.5em}g\hspace{0.5em}$};
		\node [style=none] (40) at (4.75, 4.75) {};
		\node [style=none] (41) at (5.75, 4.75) {};
		\node [style=none] (42) at (6.75, 2.25) {};
		\node [style=blackdot] (43) at (5.75, 4) {};
		\node [style=none] (44) at (4.75, 5.25) {};
		\node [style=none] (45) at (5.75, 5.25) {};
		\node [style=blackdot] (46) at (4.75, 6) {};
		\node [style=none] (47) at (6.5, 6) {};
		\node [style=none] (48) at (7.75, 2.25) {};
		\node [style=none] (49) at (7.25, 4.25) {};
		\node [style=none] (50) at (7.25, 5.25) {};
		\node [style=none] (51) at (6.5, 6.75) {};
		\node [style=none] (52) at (6.75, 3) {};
		\node [style=none] (53) at (6.25, 3.75) {};
		\node [style=none] (54) at (5.75, 3) {};
		\node [style=none] (55) at (5.75, 3) {};
		\node [style=none] (56) at (5.75, 3) {};
		\node [style=none] (57) at (5.25, 2.25) {};
		\node [style=none] (58) at (4.75, 3) {};
		\node [style=none] (59) at (4.75, 4.75) {};
		\node [style=none] (60) at (13.5, 0) {};
		\node [style=box] (61) at (14, 3) {$\hspace{0.5em}f\hspace{0.5em}$};
		\node [style=blackdot] (62) at (14.5, 1.75) {};
		\node [style=none] (63) at (13.5, 2.75) {};
		\node [style=none] (64) at (14.5, 2.75) {};
		\node [style=box] (65) at (10.75, 4) {$\hspace{0.5em}g\hspace{0.5em}$};
		\node [style=none] (66) at (10.25, 3.75) {};
		\node [style=none] (67) at (11.25, 3.75) {};
		\node [style=none] (68) at (13.5, 3.25) {};
		\node [style=blackdot] (69) at (11.25, 3) {};
		\node [style=none] (70) at (10.25, 4.25) {};
		\node [style=none] (71) at (11.25, 5.25) {};
		\node [style=blackdot] (72) at (10.25, 5) {};
		\node [style=none] (73) at (12, 6.25) {};
		\node [style=none] (74) at (14.5, 3.25) {};
		\node [style=none] (76) at (12.75, 5.25) {};
		\node [style=none] (77) at (12, 6.75) {};
		\node [style=none] (78) at (13.5, 3.5) {};
		\node [style=none] (79) at (13, 4.25) {};
		\node [style=none] (80) at (11.25, 1) {};
		\node [style=none] (81) at (11.25, 1) {};
		\node [style=none] (82) at (11.25, 1) {};
		\node [style=none] (83) at (10.75, 0) {};
		\node [style=none] (84) at (10.25, 1) {};
		\node [style=none] (85) at (10.25, 3.75) {};
		\node [style=none] (86) at (12.5, 3.5) {};
		\node [style=none] (87) at (9.75, 5.25) {};
		\node [style=none] (88) at (11.75, 5.25) {};
		\node [style=none] (89) at (9.75, 2.75) {};
		\node [style=none] (90) at (11.75, 2.75) {};
		\node [style=none] (91) at (12.25, 1.5) {};
		\node [style=none] (92) at (15, 1.5) {};
		\node [style=none] (93) at (11.25, 4.25) {};
		\node [style=none] (94) at (12.25, 4.5) {};
		\node [style=none] (95) at (15, 4.5) {};
		\node [style=none] (96) at (9.75, 1) {};
		\node [style=none] (97) at (14.75, 1) {};
		\node [style=none] (98) at (14.75, -0.25) {};
		\node [style=none] (99) at (9.75, -0.25) {};
		\node [style=none] (100) at (8.75, 3.25) {$=$};
	\end{pgfonlayer}
	\begin{pgfonlayer}{edgelayer}
		\draw [style=morphism-edge] (0.center) to (3.center);
		\draw [style=morphism-edge] (2) to (4.center);
		\draw [style=morphism-edge] (9.center) to (7.center);
		\draw [style=morphism-edge] (8.center) to (6.center);
		\draw [style=morphism-edge] (12) to (10.center);
		\draw [style=morphism-edge] (13.center) to (11.center);
		\draw [style=morphism-edge] (15.center) to (18.center);
		\draw [style=morphism-edge] (17) to (19.center);
		\draw [style=morphism-edge] (24) to (22.center);
		\draw [style=morphism-edge] (23.center) to (21.center);
		\draw [style=morphism-edge] (27) to (25.center);
		\draw [style=morphism-edge, in=90, out=-180] (28.center) to (26.center);
		\draw [style=morphism-edge, in=270, out=90] (29.center) to (30.center);
		\draw [style=morphism-edge] (31.center) to (30.center);
		\draw [style=morphism-edge, in=360, out=90] (31.center) to (28.center);
		\draw [style=morphism-edge] (28.center) to (32.center);
		\draw [style=morphism-edge] (34.center) to (37.center);
		\draw [style=morphism-edge] (36) to (38.center);
		\draw [style=morphism-edge] (43) to (41.center);
		\draw [style=morphism-edge] (46) to (44.center);
		\draw [style=morphism-edge, in=90, out=-180] (47.center) to (45.center);
		\draw [style=morphism-edge, in=270, out=90] (48.center) to (49.center);
		\draw [style=morphism-edge] (50.center) to (49.center);
		\draw [style=morphism-edge, in=360, out=90] (50.center) to (47.center);
		\draw [style=morphism-edge] (47.center) to (51.center);
		\draw [style=morphism-edge] (42.center) to (52.center);
		\draw [style=morphism-edge, in=360, out=90] (52.center) to (53.center);
		\draw [style=morphism-edge, in=90, out=-180] (53.center) to (54.center);
		\draw [style=morphism-edge] (54.center) to (55.center);
		\draw [style=morphism-edge] (55.center) to (56.center);
		\draw [style=morphism-edge, in=0, out=-90] (56.center) to (57.center);
		\draw [style=morphism-edge, in=270, out=180] (57.center) to (58.center);
		\draw [style=morphism-edge] (58.center) to (59.center);
		\draw [style=morphism-edge] (60.center) to (63.center);
		\draw [style=morphism-edge] (62) to (64.center);
		\draw [style=morphism-edge] (69) to (67.center);
		\draw [style=morphism-edge] (72) to (70.center);
		\draw [style=morphism-edge, in=90, out=-180] (73.center) to (71.center);
		\draw [style=morphism-edge, in=360, out=90] (76.center) to (73.center);
		\draw [style=morphism-edge] (73.center) to (77.center);
		\draw [style=morphism-edge] (68.center) to (78.center);
		\draw [style=morphism-edge, in=360, out=90] (78.center) to (79.center);
		\draw [style=morphism-edge] (80.center) to (81.center);
		\draw [style=morphism-edge] (81.center) to (82.center);
		\draw [style=morphism-edge, in=0, out=-90] (82.center) to (83.center);
		\draw [style=morphism-edge, in=270, out=180] (83.center) to (84.center);
		\draw [style=morphism-edge] (84.center) to (85.center);
		\draw [style=morphism-edge, in=180, out=90] (86.center) to (79.center);
		\draw [style=morphism-edge, in=90, out=-90] (86.center) to (82.center);
		\draw [style=ddd] (90.center) to (88.center);
		\draw [style=ddd] (88.center) to (87.center);
		\draw [style=ddd] (87.center) to (89.center);
		\draw [style=ddd] (89.center) to (90.center);
		\draw [style=ddd] (92.center) to (91.center);
		\draw [style=morphism-edge, in=270, out=90] (93.center) to (71.center);
		\draw [style=ddd] (95.center) to (92.center);
		\draw [style=ddd] (91.center) to (94.center);
		\draw [style=ddd] (94.center) to (95.center);
		\draw [style=morphism-edge, in=90, out=-90] (76.center) to (74.center);
		\draw [style=ddd] (99.center) to (96.center);
		\draw [style=ddd] (96.center) to (97.center);
		\draw [style=ddd] (97.center) to (98.center);
		\draw [style=ddd] (98.center) to (99.center);
	\end{pgfonlayer}
      \end{tikzpicture}
    \]
    where the upper morphisms enclosed by dotted boxes lie in $\mathbb{S}(\mathfrak{I})$ and $\mathbb{S}(\mathfrak{J})$, and the morphism in the lower dotted box plays the role of $\varphi$ in the description of $\mathbb{S}(\mathfrak{I})\mathbb{S}(\mathfrak{J})$. Thus, left-hand side above presents a general element of $\mathbb{S}(\mathfrak{IJ})$, and right-hand side illustrates how to realise it as an element of $\mathbb{S}(\mathfrak{I})\mathbb{S}(\mathfrak{J})$, proving the result.
\end{proof}

\begin{definition}
    An ideal $I$ in $A$ is said to be \emph{nilpotent} if there is $k$ such that $I^{k} = 0$.

    Similarly, a $\cp$-stable ideal $\mathfrak{J}$ in $\pfa$ is said to be nilpotent if $\mathfrak{J}^{k} = 0$ for some $k$.
\end{definition}

\begin{corollary}
 An ideal $I$ of $A$ is nilpotent if and only if so is the $\cp$-stable ideal $\mathfrak{I}$ in $\pfa$ corresponding to $I$ under the isomorphism of \autoref{RSlattice}.
\end{corollary}

\begin{proof}
    This follows directly from $\mathbb{S}(\mathfrak{I}^{k}) = {\mathbb{S}(\mathfrak{I})}^{k}$ by \autoref{Prop:Sprod}.
\end{proof}

\subsection{The \texorpdfstring{$\mathcal{C}$}{C}-module radical and exact algebras}

Recall that for an additive $\Bbbk$-linear category $\mathcal{A}$, its \emph{Jacobson radical} $\on{Rad}_{\mathcal{A}}$ (see e.g.\ \cite[Appendix~A.3]{ASS},\ \cite[Section~2]{Kr}) is the unique two-sided ideal $\csym{J}$ in $\mathcal{A}$ such that for any $X \in \mathcal{A}$, the ideal $\csym{J}(X,X) \subseteq \on{End}_{\mathcal{A}}(X)$ is the Jacobson radical. In the case when $\mathcal{A}$ is $\on{Hom}$-finite, it is the greatest ideal in $\mathcal{A}$ which does not contain any non-zero idempotents.

\begin{definition}\label{def:cmoduleradical}
 We define the \emph{$\mathcal{C}$-module radical $\radca$ in $\pfa$} by setting, for $P,P' \in \cp$:
 \begin{equation}\label{eqmoduleradical}
 \radca(P,P')
 \defeq
 \left\{
   f \in \homa (P\triangleright A,P' \triangleright A)
 \ \Bigg|\ %
   \begin{aligned}
     &Q\triangleright f \in \on{Rad}_{\pfa}(Q\otimes P\triangleright A,Q\otimes P'\triangleright A) \\
     &\text{for all } Q \in \cp
   \end{aligned}
 \right\}.
 \end{equation}
 We denote the ideal in $A$ representing $\mathbb{S}(\radca)$ by $\iradc(A)$, and refer to it as the $\mathcal{C}$-module radical of $A$.
\end{definition}

\begin{lemma}
  The $\mathcal{C}$-module radical is a  $\cp$-stable ideal in $\pfa$.
\end{lemma}

\begin{proof}
   If $g \in \radca$ and $f,h$ are morphisms of $\pfa$ such that $h \circ g \circ f$ is defined, then the equality $P \triangleright (g \circ f) = (P\triangleright g) \circ (P \triangleright f)$ implies that $g \circ f \in \radca$ and similarly for $h \circ g$.
   Thus $\radca$ is an ideal in $\pfa$. To see that it is $\cp$-stable, note that for a fixed $Q\in \cp$, the equality $P\triangleright (Q \triangleright f) = (P \otimes Q)\triangleright f$ implies that $P \triangleright (Q \triangleright f) \in \on{Rad}_{\pfa}$ for all $P$, so $Q \triangleright f \in \radca$.
\end{proof}

\begin{lemma}\label{radcarad}
 The $\mathcal{C}$-module radical is contained in the $\Bbbk$-linear radical. In other words, for all $P,P' \in \cp$, we have $\radca(P,P') \subseteq \on{Rad}_{\pfa}(P, P')$.
\end{lemma}

\begin{proof}
 It suffices to observe that for $f \in \radca(P,P')$, we have
 \[
 f = (\ev_{\rd{P}} \otimes P' \otimes A) \circ (P \otimes \rd{P} \otimes f) \circ P \otimes \coev_{\rd{P}} \otimes A,
 \]
 and by definition of $\radca$ we have $P \otimes \rd{P} \otimes f \in \on{Rad}_{\pfa}$, and so the right-hand side lies in $\on{Rad}_{\pfa}$ and thus so does $f$.
\end{proof}

\begin{proposition}\label{prop:greatestnilpotent}
  The $\mathcal{C}$-module radical $\on{Rad}_{A}^{\mathcal{C}}$ is the greatest nilpotent $\cp$-stable ideal in $\pfa$.
  In particular, $\iradc(A)$ is the greatest nilpotent ideal object in $A$.
\end{proposition}

\begin{proof}
 Assume that $\mathfrak{J}$ is a nilpotent $\cp$-stable ideal in $\pfa$. Since $\mathfrak{J}$ is nilpotent, we have an inclusion $\mathfrak{J} \subseteq \on{Rad}_{\pfa}$ of ideals in $\pfa$.
 Then for any $f \in \mathfrak{J}$ and any $Q \in \cp$, since $\mathfrak{J}$ is $\cp$-stable, we must have $Q \triangleright f \in \mathfrak{J}$ and hence also $Q \triangleright f \in \on{Rad}_{\pfa}$. Thus $f \in \radca$, and since $f$ was arbitrary, we also have $\mathfrak{J} \subseteq \radca$.
\end{proof}

Recall that we denote by ${\left(\radca\right)}^{\mathsf{c}}$ the unique extension of $\radca$ to an ideal of $\projca$, following \autoref{lem:CauchyCorr}.
\begin{proposition}\label{laststep}
Let $X \in \modca$ and let $Q \in \cp$. Let $R_{1} \xrightarrow{r_{1}} R_{0} \xtwoheadrightarrow{r_{0}} Q \triangleright X$ be exact, where $r_{0}$ is a projective cover of $Q \triangleright X$.
 Then $r_{1} \in {\left(\radca\right)}^{\mathsf{c}}$.
\end{proposition}

\begin{proof}
    We want to first apply~\cite[Lemma~4.16]{St2}. Using the notation of~\cite{St2}, we set $\csym{S} \defeq \cp$ and $\mathbf{M} \defeq \projca$. In order to apply the result, we first observe that the assumption about transitivity of $\mathbf{M}$ made therein is used only to conclude that $\setj{P \triangleright X \mid P \in \csym{S}} \neq \setj{0}$ for $X \in \mathbf{M}$ non-zero. This latter claim is clearly true in our case, so the claim of~\cite[Lemma~4.16]{St2} applies and we may use~\cite[Lemma~4.11]{St2} to conclude that for any $Q' \in \cp$, the morphism $Q' \triangleright r_{0}$
    is a projective cover of $Q' \triangleright (Q\triangleright X)$.
    In particular, by~\cite[Lemma~4.13]{St2},
    we have that $Q' \triangleright r_{1} \in \on{Rad}_{\projca}$
    for all $Q' \in \cp$, establishing the result.
\end{proof}

\begin{remark}\label{idempotentexplain}
 In our case,~\cite[Lemma~4.16]{St2} produces an idempotent $\mathbf{e}$ in the semiring $\mathbb{R} \otimes_{\mathbb{Z}} {[\cp]}_{\oplus}$, where ${[\cp]}_{\oplus}$ denotes the split Grothendieck semiring.
 The space $\mathbb{R} \otimes_{\mathbb{Z}} {[\projca]}_{\oplus}$ becomes a semiring module with a $\mathbb{Z}_{\geq 0}$-basis given by $\on{Indec}(\projca)$, the indecomposable objects of $\projca$.
 For~\cite[Lemma~4.11]{St2}, it is crucial that for a non-zero non-negative linear combination $v$ in $\on{Indec}(\projca)$, the element $\mathbf{e} \triangleright v$ again is non-zero. However, since $\cp$ has finitely many indecomposables, $\mathbf{e}$ by construction is a positive linear combination in $\on{Indec}(\cp)$, hence the condition is satisfied if and only if $\setj{P \triangleright R \mid P \in \cp} \neq \setj{0}$ for all non-zero $R \in \projca$.
\end{remark}

\begin{remark}\label{infiniteexplain}
 By replacing $[\cp^{\on{op}},\Vecc]$ with its full subcategory of presheaves which vanish on all but finitely many isomorphism classes of indecomposable projectives, our approach is also applicable to tensor categories which are not finite but have enough projectives such as those coming from Lie superalgebras and infinite fusion categories, with one exception: \autoref{laststep} requires finiteness.
  Moreover, while the use of the Perron--Frobenius theorem in obtaining the idempotent $\mathbf{e}$ of \autoref{idempotentexplain} could perhaps in some cases be replaced by suitable generalisations to infinite-dimensional spaces, it is easy to find tensor categories whose Grothendieck semirings of projectives do not admit idempotents.

 For example, for $e \in {[\mathfrak{sl}_{2}(\mathbb{C})\text{-}\mathrm{mod}]}_{\oplus} \cong \mathbb{Z}\{\, V_{i} \mid i = 1,2,\ldots \,\}$, letting
 \[
 m(e) = \max\,\{\, n \mid \text{coefficient of }V_{n} \text{ in }e\text{ is non-zero}\,\},
 \]
  we find $m(e^{2}) = 2m(e)-1$, so for $e$ idempotent we must have $m = 1$, showing that the only idempotent is the unit element.
\end{remark}

\begin{theorem}\label{thm:semisimplecharacterization}
 The algebra $A$ is exact if and only if $\,\iradc(A) = 0$.
\end{theorem}

\begin{proof}
 If $\iradc(A) = 0$, then also $\radca = 0$ and hence also ${\left(\radca\right)}^{\mathsf{c}} = 0$. By \autoref{laststep}, for $Q \in \cp$ and $X \in \modca$, the projective cover of $Q \triangleright X$ must then be zero, and so $Q \triangleright X$ must be projective, which establishes the exactness of $A$.

 Conversely, if $A$ is exact, then for a morphism $f\colon R\to R'$ of $\projca$ and $P \in \cp$, we find that $P \triangleright \on{Coker}(f)$ is projective, and hence $P\triangleright R' \twoheadrightarrow P \triangleright \on{Coker}(f)$ is split.
 Thus, $P \triangleright f$ is either zero, or not radical. Since the former implies $f = 0$, and the latter implies $f \not\in \radca$, we find $\radca = 0$.
\end{proof}

\section{Main result}\label{sec:mainresult}

We are now able to state our main result, answering~\cite[Conjecture~B.6]{EOf} affirmatively.

\begin{theorem}\label{thm:mainresult}
 Let $A$ be an algebra object in a finite tensor category $\mathcal{C}$. The following are equivalent:
\begin{enumerate}[label=(\roman*)]
 \item\label{Aexact} $A$ is exact.
 \item\label{RadA0} $\iradc(A) = 0$.
 \item\label{NoNilpotent} $A$ has no non-zero nilpotent ideals.
 \item\label{Semisimple} $A$ is a finite direct product of simple algebras.
\end{enumerate}
\end{theorem}

\begin{proof}
   The equivalence of~\ref{Aexact} and~\ref{RadA0} is given in \autoref{thm:semisimplecharacterization}. The equivalence of~\ref{RadA0} and~\ref{NoNilpotent} follows immediately from \autoref{prop:greatestnilpotent}. That~\ref{Aexact} implies~\ref{Semisimple} is well known, however, for completeness, we include a self-contained proof below in \autoref{prop:splitalg}.
   To conclude the proof, we will show that~\ref{Semisimple} implies~\ref{NoNilpotent}.

   Let $A = A_{1} \times A_{2}\times \cdots \times A_{n}$ be a finite direct product of simple algebras. The above direct product decomposition is also a decomposition of $A$-$A$-bimodules, and so ideals of $A$ are of the form $I = I_{1} \times\cdots \times I_{n}$, where $I_{k}$ is an ideal in $A_{k}$. Further, it is easy to verify that $$(I_{1} \times\cdots \times I_{n})(J_{1} \times\cdots \times J_{n}) = I_{1}J_{1} \times \cdots \times I_{n}J_{n}.$$ Thus, $I$ is nilpotent if and only if $I_{k}$ is nilpotent, for $k=1,\ldots,n$. But, by the simplicity assumption, for all $k$, the only nilpotent ideal in $A_{k}$ is $0$. Thus the only nilpotent ideal in $A$ is $0 \times \cdots \times 0 = 0$.
\end{proof}

The rest of this section consists of a self-contained proof of the implication~\ref{Aexact} $\implies$~\ref{Semisimple} in \autoref{thm:mainresult}.

\begin{lemma}\label{idealsarethick}
 Let $A$ be an exact algebra object. Let $\csym{J}$ be an ideal in $\projca$, let $R,R' \in \projca$ be indecomposable. If $\csym{J}(R,R') \neq 0$, then $\on{id}_{R}, \on{id}_{R'} \in \csym{J}$.
\end{lemma}

\begin{proof}
   Assume $\csym{J}(R,R') \neq 0$, and let $f \in \csym{J}(R,R')$ be non-zero.
   Let $Q\in \cp$ be such that there is a split monomorphism $\iota\colon R \hookrightarrow Q \triangleright A$. Let $\pi$ be a retraction for $\iota$. Consider the following commutative diagram:
\[\begin{tikzcd}[column sep=large]
	R \\
	{Q\triangleright A} & {R'} \\
	{(Q \otimes \ld{Q} \otimes Q) \triangleright A} & {Q \otimes \ld{Q} \triangleright R'}
	\arrow["\iota", from=1-1, to=2-1]
	\arrow["f", from=1-1, to=2-2]
	\arrow["{f \circ \pi}", from=2-1, to=2-2]
	\arrow["{(\coev_{{Q}} \otimes Q) \triangleright A}"', from=2-1, to=3-1]
	\arrow["{\coev_{{Q}} \triangleright R'}", from=2-2, to=3-2]
	\arrow["{(Q\otimes \ld{Q}) \triangleright (f\circ \pi)}"', from=3-1, to=3-2]
\end{tikzcd}\]
 Since $\coev_{{Q}} \otimes Q$ is a split mono,
 and the bottom arrow is also split by $Q$ being projective and $A$ being exact,
 the outer path of the above diagram defines a split morphism.
 On the other hand, since $\coev_{{Q}} \triangleright R'$ is monic, and $f$ is non-zero, the composite $(\coev_{{Q}}\triangleright R') \circ f$ is a non-zero split morphism with an indecomposable domain, and thus a split monomorphism. Since $f \in \csym{J}$, we obtain $(\coev_{{Q}}\triangleright R') \circ f \in \csym{J}$ and thus $\on{id}_{R} = \rho \circ (\coev_{{Q}}\triangleright R') \circ f$, for a retraction $\rho$, also lies in $\csym{J}$. Similarly one shows that $\on{id}_{R'}$ lies in $\csym{J}$.
\end{proof}

\begin{corollary}\label{cor:RR'}
  Let $A$ be an exact algebra object. Let $\csym{J}$ be an ideal in $\projca$ and $f \in \csym{J}(R,R')$. For any indecomposable $R,R' \in \projca$, we have
  \[
  \csym{J}(R,R') =
   \begin{cases}
  \homa(R,R')\text{ if }\on{id}_{R}\text{ or }\on{id}_{R'} \text{ lies in }\csym{J} \\
  0 \text{ if }\on{id}_{R}\text{ or }\on{id}_{R'} \text{ doesn't lie in }\csym{J}.
   \end{cases}
  \]
 In particular, $\homa(R,R') = 0$ if exactly one of $\on{id}_{R'}, \on{id}_{R}$ lies in $\csym{J}$.
\end{corollary}

\begin{proof}
 The first case is trivial and holds for any ideal. The latter follows immediately from \autoref{idealsarethick}.
\end{proof}

We now give a module-theoretic proof of a well-known result about exact algebra objects:

\begin{proposition}\label{prop:splitalg}
  An exact algebra object is a finite direct product of simple exact algebras.
\end{proposition}

\begin{proof}
 Let $A$ be an exact algebra object. Let $J$ be an ideal in $A$ and $\csym{J}$ be the corresponding ideal in $\projca$. Let $\mathcal{P} = \setj{R \mid R \in \projca \text{ is indecomposable and } \on{id}_{R} \in \csym{J}}$, and similarly let $\overline{\mathcal{P}} = \setj{R \mid R \in \projca \text{ is indecomposable and } \on{id}_{R} \not\in \csym{J}}$. Then $\on{add}\mathcal{P}$ and $\on{add}\overline{\mathcal{P}}$ both define $\mathcal{C}$-module subcategories of $\projca$, such that $\homa(R,R') = 0$ if $R \in \mathcal{P}, R' \in \overline{\mathcal{P}}$ or vice-versa, by \autoref{cor:RR'}.
 Let $R,R'$ be indecomposable objects in $\projca$, and set $L \defeq \topp(R), L' \defeq \topp(R')$. We have
 \begin{align*}
   \dim\homa(Q \triangleright R, R')
    \geq \dim\homa(Q \triangleright R, L')
     = \dim\homa(R, \rd{Q} \triangleright L')
    = [\rd{Q} \triangleright L' : L].
 \end{align*}
 Thus, for $R \in \mathcal{P}$ and $R' \in \overline{\mathcal{P}}$,
 there is no $Q \in \cp$ such that $L'$ is a subquotient of $Q \triangleright L$ and vice-versa. This implies the same statement for $Q \in \mathcal{C}$, and thus partitions $\on{Irr}(\modca)$ into two equivalence classes with respect to the relation of~\cite[Lemma~7.6.6]{EGNO}, which yields a splitting of $\mathcal{C}$-module categories $\modca = {(\modca)}_{1} \oplus {(\modca)}_{2}$ by~\cite[Proposition~7.6.7]{EGNO}, and in turn also a splitting $\modca = \modw{A_{1}}\oplus \modw{A_{2}}$.
 This yields a decomposition $A \cong A_{1}\times A_{2}$.
 Since $\on{Irr}(\modca)$ is a finite set, the result follows.
\end{proof}

\begin{remark}
   In the setting of the proof of \autoref{prop:splitalg} we may in fact describe the splitting $A = A_{1} \times A_{2}$ explicitly. Let $\mathfrak{I}$ be the ideal in $\pfa$ determined by $\on{Hom}_{\on{add}\overline{\mathcal{P}}}(-,-)$, let $\mathtt{I}$ be the corresponding mixed subfunctor of ${\mathcal{C}(-,A)}_{\cp}$, and $I$ the corresponding ideal in $A$. Similarly, let $\mathtt{J}$ be the mixed subfunctor corresponding to $J$. There is an isomorphism $\mathtt{I} \simeq {\mathcal{C}(-,A)}_{\cp}/\mathtt{J}$, which yields an isomorphism $I \simeq A/J$ of $A$-$A$-bimodules, and thus we may set $A_{1} = A/I$ and $A_{2} = A/J$ as the algebra splitting.
\end{remark}

If the base field $\bk$ is perfect, then the tensor product of two simple finite-dimensional $\bk$-algebras is always semisimple, in other words, the tensor product of two exact algebras in $\Vecc$ remains exact. This does not extend to general finite symmetric tensor categories:

\begin{example}
    Let $\bk$ be perfect and $G$ an infinitesimal group scheme over $\bk$. Then the coordinate ring $\bk[G]$ is a simple (commutative) algebra in $\cC=\mathrm{Rep} G$ for the left regular $G$-action, it is in fact $\cF(\mathrm{Rep} G)$ in the notation of \autoref{SecSymm}. However, the tensor product $\bk[G]\otimes \bk[G]$ is not an exact algebra. For a concrete example, consider the affine group scheme $\mu_2$ of second roots of unity, which is infinitesimal if and only if $\mathrm{char}(\bk)=2$. Now $\cC$ is the category of $\mathbb{Z}/2$-graded vector spaces and we consider the simple algebra $A=\bk[x]/x^2$ with $x$ in odd degree. Then $A\otimes A$ is a direct product of simple algebras ($A\times A$) if and only if $\mathrm{char}(\bk)\not=2$.
\end{example}

\section{The radical of a module object and the maximal semisimple quotient}\label{sec:radmod2}

The following definition is analogous to \autoref{def:productideal}
\begin{definition}\label{def:MI}
   Let $\mathcal{D}$ be an abelian monoidal category whose tensor product is right exact in both variables. Let $(B,\mu_{B},\nabla_{B})$ be an algebra object in $\mathcal{D}$, let $(M,\nabla_M) \in \on{mod}_{\mathcal{D}}(B)$, and let $I$ be an ideal in $B$, with inclusion morphism $\iota_{I}$ into $B$. We define the module $MI$ by
   \[
   MI = \on{Im}(M \otimes I \xrightarrow{M \otimes \iota_{J}} M  \otimes B \xrightarrow{\nabla_{M}} M).
   \]
\end{definition}

Analogously to \autoref{lem:IJdescription}, for $M \in \modca$ and $I$ an ideal in $A$, we obtain a description of ${\mathcal{C}(-,MI)}_{|\cp}$ by Day convolution in $[\cp^{\on{op}}, \Vecc]$:

\begin{lemma}\label{lem:MIdescription}
     Let $M \in \modca$ and let $I$ be an ideal in $A$. Then ${\mathcal{C}(\blank,MI)}_{|\cp}$ is given by
 \[
 \mathcal{C}(P,MI) = \setj{\nabla_{M} \circ (g \otimes h) \circ \varphi \mid g \in \mathcal{C}(Q,M), h \in \mathcal{C}(Q',I) \text{ and } \varphi \in \mathcal{C}(P,Q\otimes Q')}.
 \]
\end{lemma}

The proof technique for \autoref{prop:DayRadicalModule} is analogous to that for \autoref{Prop:Sprod}.
\begin{proposition}\label{prop:DayRadicalModule}
    We have
    \begin{align*}
      (M\iradc(A))(P)
      &= \left\{\nabla \circ (g \otimes ((q \otimes A)\circ h^{1})) \circ \varphi
      \ \Bigg|\ %
      \begin{aligned}
        &g \in \mathcal{C}(Q,M),\, q \colon Q''\twoheadrightarrow \mathbb{1}, \\
        &h \in \radca(Q',Q''),\, \varphi \in \mathcal{C}(P,Q\otimes Q')
      \end{aligned}
      \right\}\\
      &= \{\,{(g \circ h)}^{1} \mid g \in \homa(Q \otimes A, M), h \in \radca(P,Q)\,\}.
    \end{align*}
\end{proposition}

\begin{proof}
    The first equality follows immediately by combining \autoref{def:Smap} with \autoref{lem:MIdescription}.
    The "$\subseteq$" direction of the second equality follows by the left diagram below,
    and the "$\supseteq$" direction is a consequence of the right one.
    
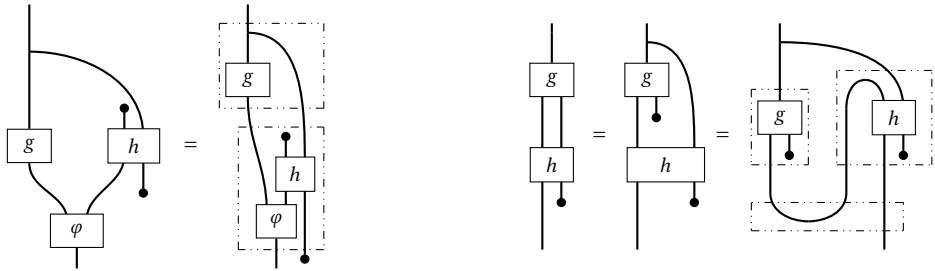
\[\begin{tikzpicture}[tikzfig]
	\begin{pgfonlayer}{nodelayer}
		\node [style=none] (0) at (-2, -2.75) {};
		\node [style=none] (1) at (-2, 0.5) {};
		\node [style=box] (2) at (-1.75, 0.75) {\scriptsize$\hspace{0.3em}h\hspace{0.3em}$};
		\node [style=none] (3) at (-1.5, 0.5) {};
		\node [style=none] (4) at (-2, 1) {};
		\node [style=none] (5) at (-1.5, 1) {};
		\node [style=none] (6) at (-2.5, 1.75) {};
		\node [style=none] (7) at (-3, 1) {};
		\node [style=none] (8) at (-3, -1.25) {};
		\node [style=blackdot] (10) at (-1.5, -0.25) {};
		\node [style=none] (11) at (-4, -2) {};
		\node [style=none] (12) at (-5, -1.25) {};
		\node [style=box] (13) at (-4.75, 0.75) {\scriptsize$\hspace{0.3em}g\hspace{0.3em}$};
		\node [style=none] (14) at (-5, 0.5) {};
		\node [style=none] (15) at (-4.5, 0.5) {};
		\node [style=blackdot] (16) at (-4.5, -0.25) {};
		\node [style=none] (17) at (-4.75, 2.75) {};
		\node [style=none] (18) at (-4.75, 3.25) {};
		\node [style=none] (19) at (-1.5, -1.5) {};
		\node [style=none] (20) at (-5.5, -1.5) {};
		\node [style=none] (21) at (-5.5, -2.25) {};
		\node [style=none] (22) at (-1.5, -2.25) {};
		\node [style=none] (23) at (-4, 1.5) {};
		\node [style=none] (24) at (-5.5, 1.5) {};
		\node [style=none] (25) at (-5.5, -0.5) {};
		\node [style=none] (26) at (-4, -0.5) {};
		\node [style=none] (27) at (-0.75, 2) {};
		\node [style=none] (28) at (-3.25, 2) {};
		\node [style=none] (29) at (-3.25, -0.5) {};
		\node [style=none] (30) at (-0.75, -0.5) {};
		\node [style=none] (31) at (-6.25, 0.25) {\scriptsize$=$};
		\node [style=none] (32) at (-8.5, -2.75) {};
		\node [style=blackdot] (33) at (-7, -1.5) {};
		\node [style=box] (34) at (-7.75, -0.5) {\scriptsize$\hspace{1.1em}h\hspace{1.1em}$};
		\node [style=none] (35) at (-8.5, -0.75) {};
		\node [style=none] (36) at (-7, -0.75) {};
		\node [style=none] (37) at (-8.5, -0.25) {};
		\node [style=none] (38) at (-7, -0.25) {};
		\node [style=none] (39) at (-8.5, 1.5) {};
		\node [style=none] (40) at (-8, 1.5) {};
		\node [style=blackdot] (41) at (-8, 0.75) {};
		\node [style=none] (42) at (-8.25, 2.75) {};
		\node [style=none] (43) at (-8.25, 3.25) {};
		\node [style=box] (44) at (-8.25, 1.75) {\scriptsize$\hspace{0.3em}g\hspace{0.3em}$};
		\node [style=none] (45) at (-9.5, 0.25) {\scriptsize$=$};
		\node [style=none] (46) at (-11, -2.75) {};
		\node [style=blackdot] (47) at (-10.5, -1.5) {};
		\node [style=box] (48) at (-10.75, -0.5) {\scriptsize$\hspace{0.3em}h\hspace{0.3em}$};
		\node [style=none] (49) at (-11, -0.75) {};
		\node [style=none] (50) at (-10.5, -0.75) {};
		\node [style=none] (51) at (-11, -0.25) {};
		\node [style=none] (52) at (-10.5, -0.25) {};
		\node [style=none] (53) at (-11, 1.5) {};
		\node [style=none] (54) at (-10.5, 1.5) {};
		\node [style=none] (56) at (-10.75, 3) {};
		\node [style=none] (57) at (-10.75, 3.25) {};
		\node [style=box] (58) at (-10.75, 1.75) {\scriptsize$\hspace{0.3em}g\hspace{0.3em}$};
		\node [style=none] (59) at (-23.25, -3.25) {};
		\node [style=box] (60) at (-23.25, -2.25) {\scriptsize$\hspace{0.5em}\varphi\hspace{0.5em}$};
		\node [style=box] (61) at (-24.5, 0) {\scriptsize$\hspace{0.3em}g\hspace{0.3em}$};
		\node [style=box] (62) at (-21.75, 0) {\scriptsize$\hspace{0.5em}h\hspace{0.5em}$};
		\node [style=none] (63) at (-23.5, -2) {};
		\node [style=none] (64) at (-23, -2) {};
		\node [style=none] (65) at (-22, -0.25) {};
		\node [style=none] (66) at (-21.5, -0.25) {};
		\node [style=none] (67) at (-22, 0.25) {};
		\node [style=none] (68) at (-21.5, 0.25) {};
		\node [style=blackdot] (69) at (-21.5, -1.25) {};
		\node [style=blackdot] (70) at (-22, 1) {};
		\node [style=none] (71) at (-24.5, 3.75) {};
		\node [style=none] (72) at (-20.25, 0) {\scriptsize$=$};
		\node [style=none] (73) at (-18, -3.25) {};
		\node [style=box] (74) at (-18, -2) {\scriptsize$\hspace{0.2em}\varphi\hspace{0.2em}$};
		\node [style=box] (75) at (-18.75, 1.75) {\scriptsize$\hspace{0.3em}g\hspace{0.3em}$};
		\node [style=box] (76) at (-17.5, -0.75) {\scriptsize$\hspace{0.2em}h\hspace{0.2em}$};
		\node [style=none] (77) at (-18.25, -1.75) {};
		\node [style=none] (78) at (-17.75, -1.75) {};
		\node [style=none] (79) at (-17.75, -1) {};
		\node [style=none] (80) at (-17.25, -1) {};
		\node [style=none] (81) at (-17.75, -0.5) {};
		\node [style=none] (82) at (-17.25, -0.5) {};
		\node [style=blackdot] (83) at (-17.25, -3) {};
		\node [style=blackdot] (84) at (-17.75, 0.25) {};
		\node [style=none] (85) at (-18.75, 3) {};
		\node [style=none] (86) at (-24.5, 2.5) {};
		\node [style=none] (87) at (-18.75, 3.75) {};
		\node [style=none] (88) at (-19, -2.75) {};
		\node [style=none] (89) at (-16.75, -2.75) {};
		\node [style=none] (90) at (-19, 0.5) {};
		\node [style=none] (91) at (-16.75, 0.5) {};
		\node [style=none] (92) at (-19.5, 3.25) {};
		\node [style=none] (93) at (-16.75, 3.25) {};
		\node [style=none] (94) at (-16.75, 1) {};
		\node [style=none] (95) at (-19.5, 1) {};
	\end{pgfonlayer}
	\begin{pgfonlayer}{edgelayer}
		\draw [style=morphism-edge] (10) to (3.center);
		\draw [style=morphism-edge] (1.center) to (0.center);
		\draw [style=morphism-edge, in=360, out=90] (4.center) to (6.center);
		\draw [style=morphism-edge, in=90, out=-180] (6.center) to (7.center);
		\draw [style=morphism-edge, in=90, out=-90] (7.center) to (8.center);
		\draw [style=morphism-edge, in=180, out=-90] (12.center) to (11.center);
		\draw [style=morphism-edge, in=270, out=0] (11.center) to (8.center);
		\draw [style=morphism-edge] (16) to (15.center);
		\draw [style=morphism-edge] (12.center) to (14.center);
		\draw [style=morphism-edge] (13) to (17.center);
		\draw [style=morphism-edge, in=90, out=0] (17.center) to (5.center);
		\draw [style=morphism-edge] (17.center) to (18.center);
		\draw [style=ddd] (19.center) to (20.center);
		\draw [style=ddd] (20.center) to (21.center);
		\draw [style=ddd] (21.center) to (22.center);
		\draw [style=ddd] (22.center) to (19.center);
		\draw [style=ddd] (23.center) to (24.center);
		\draw [style=ddd] (24.center) to (25.center);
		\draw [style=ddd] (25.center) to (26.center);
		\draw [style=ddd] (26.center) to (23.center);
		\draw [style=ddd] (27.center) to (28.center);
		\draw [style=ddd] (28.center) to (29.center);
		\draw [style=ddd] (29.center) to (30.center);
		\draw [style=ddd] (30.center) to (27.center);
		\draw [style=morphism-edge] (33) to (36.center);
		\draw [style=morphism-edge] (32.center) to (35.center);
		\draw [style=morphism-edge] (37.center) to (39.center);
		\draw [style=morphism-edge] (40.center) to (41);
		\draw [style=morphism-edge, in=360, out=90] (38.center) to (42.center);
		\draw [style=morphism-edge] (44) to (42.center);
		\draw [style=morphism-edge] (42.center) to (43.center);
		\draw [style=morphism-edge] (47) to (50.center);
		\draw [style=morphism-edge] (46.center) to (49.center);
		\draw [style=morphism-edge] (51.center) to (53.center);
		\draw [style=morphism-edge] (58) to (56.center);
		\draw [style=morphism-edge] (56.center) to (57.center);
		\draw [style=morphism-edge] (52.center) to (54.center);
		\draw [style=morphism-edge] (59.center) to (60);
		\draw [style=morphism-edge, in=270, out=90] (63.center) to (61);
		\draw [style=morphism-edge, in=270, out=90] (64.center) to (65.center);
		\draw [style=morphism-edge] (67.center) to (70);
		\draw [style=morphism-edge] (66.center) to (69);
		\draw [style=morphism-edge] (73.center) to (74);
		\draw [style=morphism-edge, in=270, out=90] (77.center) to (75);
		\draw [style=morphism-edge, in=270, out=90] (78.center) to (79.center);
		\draw [style=morphism-edge] (81.center) to (84);
		\draw [style=morphism-edge] (80.center) to (83);
		\draw [style=morphism-edge] (86.center) to (61);
		\draw [style=morphism-edge, in=360, out=90] (68.center) to (86.center);
		\draw [style=morphism-edge] (86.center) to (71.center);
		\draw [style=morphism-edge] (75) to (85.center);
		\draw [style=morphism-edge, in=90, out=0] (85.center) to (82.center);
		\draw [style=morphism-edge] (85.center) to (87.center);
		\draw [style=ddd] (91.center) to (89.center);
		\draw [style=ddd] (89.center) to (88.center);
		\draw [style=ddd] (88.center) to (90.center);
		\draw [style=ddd] (90.center) to (91.center);
		\draw [style=ddd, in=270, out=90] (95.center) to (92.center);
		\draw [style=ddd] (92.center) to (93.center);
		\draw [style=ddd] (93.center) to (94.center);
		\draw [style=ddd] (94.center) to (95.center);
	\end{pgfonlayer}
    \end{tikzpicture}
    \]
\end{proof}

It is easy to verify the following claims:

\begin{proposition}\label{prop:MI0}
    Let $I$ be an ideal in $A$. Restriction $\modw{A/I} \to \modca$ along the algebra morphism $A \twoheadrightarrow A/I$ realises $\modw{A/I}$ as a full $\mathcal{C}$-module subcategory of $\modca$ which is closed under subquotients. An $A$-module object $M$ lies in the image of this functor if and only if $MI = 0$.
\end{proposition}

\begin{proposition}\label{prop:mijmij}
    Let $I,J$ be ideals in $A$ and let $M \in \modca$. Then $(MI)J = M(IJ)$.
\end{proposition}

While \autoref{lem:Qradical} below is a direct consequence of \autoref{prop:MI0}, we present it separately with proof, because of its crucial role in the proof of \autoref{thm:semisimplequotient}.

\begin{lemma}\label{lem:Qradical}
 For $Q'' \in \cp$ and $M \in \modca$, if $\,(Q''\triangleright M)I \neq 0$, then $MI \neq 0$.
\end{lemma}

\begin{proof}
 Let $g \in \mathcal{C}(Q,Q''\triangleright M)$, $h \in \mathcal{C}(Q',I)$, and $\varphi \in \mathcal{C}(P,Q\otimes Q')$.
 Consider the following diagram,
 with a non-zero element $\nabla_{Q'' \triangleright M} \circ (g \otimes h) \circ \varphi \in ((Q''\triangleright M)I)(P)$
 as its left side:
 
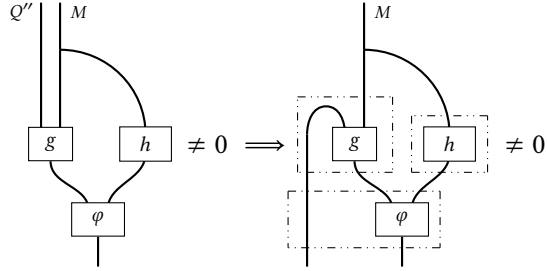
\[\begin{tikzpicture}[tikzfig]
	\begin{pgfonlayer}{nodelayer}
		\node [style=none] (0) at (-2.5, -3.25) {};
		\node [style=box] (1) at (-2.5, -2) {\scriptsize$\hspace{0.5em}\varphi\hspace{0.5em}$};
		\node [style=box] (2) at (-3.75, 0) {\scriptsize$\hspace{0.3em}g\hspace{0.3em}$};
		\node [style=box] (3) at (-1.25, 0) {\scriptsize$\hspace{0.5em}h\hspace{0.5em}$};
		\node [style=none] (4) at (-2.75, -1.75) {};
		\node [style=none] (5) at (-2.25, -1.75) {};
		\node [style=none] (6) at (-1.25, -0.25) {};
		\node [style=none] (9) at (-1.25, 0.25) {};
		\node [style=none] (12) at (-3.5, 3.75) {};
		\node [style=none] (13) at (-3.5, 2.5) {};
		\node [style=none] (14) at (1.25, 0) {$\neq 0 \implies$};
		\node [style=none] (15) at (-4, 0.25) {};
		\node [style=none] (16) at (-3.5, 0.25) {};
		\node [style=none] (17) at (-4, 3.75) {};
		\node [style=none] (18) at (-3, 3.5) {\tiny$M$};
		\node [style=none] (19) at (-4.5, 3.5) {\tiny$Q''$};
		\node [style=none] (20) at (5.5, -3.25) {};
		\node [style=box] (21) at (5.5, -2) {\scriptsize$\hspace{0.5em}\varphi\hspace{0.5em}$};
		\node [style=box] (22) at (4.25, 0) {\scriptsize$\hspace{0.3em}g\hspace{0.3em}$};
		\node [style=box] (23) at (6.75, 0) {\scriptsize$\hspace{0.5em}h\hspace{0.5em}$};
		\node [style=none] (24) at (5.25, -1.75) {};
		\node [style=none] (25) at (5.75, -1.75) {};
		\node [style=none] (26) at (6.75, -0.25) {};
		\node [style=none] (27) at (6.75, 0.25) {};
		\node [style=none] (28) at (4.5, 3.75) {};
		\node [style=none] (29) at (4.5, 2.5) {};
		\node [style=none] (30) at (4, 0.25) {};
		\node [style=none] (31) at (4.5, 0.25) {};
		\node [style=none] (32) at (3.5, 1) {};
		\node [style=none] (33) at (5, 3.5) {\tiny$M$};
		\node [style=none] (34) at (3, 0.25) {};
		\node [style=none] (35) at (3, -3.25) {};
		\node [style=none] (36) at (8.75, 0) {$\neq 0$};
		\node [style=none] (37) at (2.5, -2.75) {};
		\node [style=none] (38) at (6.5, -2.75) {};
		\node [style=none] (39) at (6.5, -1.25) {};
		\node [style=none] (40) at (2.5, -1.25) {};
		\node [style=none] (41) at (5.75, 0.75) {};
		\node [style=none] (42) at (7.75, 0.75) {};
		\node [style=none] (43) at (5.75, -0.75) {};
		\node [style=none] (44) at (7.75, -0.75) {};
		\node [style=none] (45) at (2.75, -0.75) {};
		\node [style=none] (46) at (5.25, -0.75) {};
		\node [style=none] (47) at (5.25, 1.25) {};
		\node [style=none] (48) at (2.75, 1.25) {};
	\end{pgfonlayer}
	\begin{pgfonlayer}{edgelayer}
		\draw [style=morphism-edge] (0.center) to (1);
		\draw [style=morphism-edge, in=270, out=90] (4.center) to (2);
		\draw [style=morphism-edge, in=270, out=90] (5.center) to (6.center);
		\draw [style=morphism-edge, in=360, out=90] (9.center) to (13.center);
		\draw [style=morphism-edge] (13.center) to (12.center);
		\draw [style=morphism-edge] (16.center) to (13.center);
		\draw [style=morphism-edge] (17.center) to (15.center);
		\draw [style=morphism-edge] (20.center) to (21);
		\draw [style=morphism-edge, in=270, out=90] (24.center) to (22);
		\draw [style=morphism-edge, in=270, out=90] (25.center) to (26.center);
		\draw [style=morphism-edge, in=360, out=90] (27.center) to (29.center);
		\draw [style=morphism-edge] (29.center) to (28.center);
		\draw [style=morphism-edge] (31.center) to (29.center);
		\draw [style=morphism-edge, in=90, out=0] (32.center) to (30.center);
		\draw [style=morphism-edge] (35.center) to (34.center);
		\draw [style=morphism-edge, in=180, out=90] (34.center) to (32.center);
		\draw [style=ddd] (40.center) to (39.center);
		\draw [style=ddd] (39.center) to (38.center);
		\draw [style=ddd] (38.center) to (37.center);
		\draw [style=ddd] (40.center) to (37.center);
		\draw [style=ddd] (44.center) to (42.center);
		\draw [style=ddd] (42.center) to (41.center);
		\draw [style=ddd] (41.center) to (43.center);
		\draw [style=ddd] (43.center) to (44.center);
		\draw [style=ddd] (48.center) to (47.center);
		\draw [style=ddd] (47.center) to (46.center);
		\draw [style=ddd] (46.center) to (45.center);
		\draw [style=ddd] (45.center) to (48.center);
	\end{pgfonlayer}
    \end{tikzpicture}
    \]
    The boxes on the right side indicate that the mate of the given non-zero element of $((Q''\triangleright M)I)(P)$ defines a non-zero element in $(MI)(\ld{Q''} \otimes P)$.
\end{proof}

\begin{lemma}\label{lem:projlifts}
 Let $\mathcal{A}$ be an abelian category, let $P_{1}^{M} \xrightarrow{p_{1}^{N}} P_{0}^{N} \xtwoheadrightarrow{p_0^N} N \rightarrow 0$ be an exact sequence in $\mathcal{A}$ and $M \xtwoheadrightarrow{g} N$ an epimorphism in $\mathcal{A}$.

 The epimorphism $g$ is split if and only if there is a lift $\widehat{g}\from P_{0}^{N} \rightarrow M$ of $g$ such that $\widehat{g} \circ p_{1}^{N} = 0$.
\end{lemma}

\begin{proof}
  If $g$ is split, then a lift $\widehat{g}$ as specified can be easily obtained from horseshoe lemma. If there is a left $\widehat{g}$ such that $\widehat{g} \circ p_{1}^{N}$, then, since $N$ is the cokernel of $p_{1}^{N}$, there is a unique map $h\from N \rightarrow M$, such that $h \circ p_{0}^{N} = \widehat{g}$. We then find $g \circ h \circ p_{0}^{N} = g \circ \widehat{g} = p_{0}^{N}$, and so $g \circ h = \on{id}_{N}$, since $p_{0}^{N}$ is an epimorphism.
\end{proof}

\begin{lemma}\label{lem:mradneq0}
 For $M \in \modca$ non-zero, we have  $M\iradc(A) \neq M$.
\end{lemma}

\begin{proof}
 Let $k$ be such that ${\iradc(A)}^{k} = 0$. Thus, $M{\iradc(A)}^{k} = 0$. On the other hand, by \autoref{prop:mijmij}, if $M\iradc(A) = M$, then $M{\iradc(A)}^{k} = M$, which is a contradiction.
\end{proof}

\begin{theorem}\label{thm:semisimplequotient}
  The category $\modw{A/\iradc(A)}$, viewed as a full subcategory of $\modca$, consists of the subquotients of objects of the form $Q \triangleright L$, for $Q \in \cp$ and $L$ a semisimple object in $\modca$.
\end{theorem}

\begin{proof}
    First, by \autoref{lem:mradneq0}, for any simple object $L$ of $\modca$, we have $L\iradc(A) = 0$. Thus the semisimple objects in $\modca$ lie in $\modw{A/\iradc(A)}$, and so do the objects of the form $Q \triangleright L$ for $L$ semisimple, as well as their subquotients.

    It remains to show that any object of $\modw{A/\iradc(A)}$ is of that form.
    Fix an epimorphism $Q \xtwoheadrightarrow{q} \mathbb{1}$ with $Q \in \cp$, let $M \in \modw{A/\iradc(A)}$, and let
    \[
    0 \to N \xrightarrow{f} M \xrightarrow{g} L \to 0
    \]
    be a short exact sequence with $L$ non-zero semisimple. We now show that the epimorphism $Q \triangleright g$ is split.
    If $Q \triangleright L$ is projective, then the epimorphism $Q\triangleright g$ is split.
    Otherwise, using \autoref{laststep}, the minimal projective presentation $r_{1}^{N}\from R_{1}^{Q\triangleright L} \rightarrow R_{0}^{Q\triangleright L}$ is non-zero and lies in ${\left(\radca\right)}^{\mathsf{c}}$. Choose splittings
$\begin{tikzcd}[ampersand replacement=\&]
	{R_{1}^{Q \triangleright L}} \& {Q^{1}\triangleright A}
	\arrow["{\iota_{1}}", shift left, hook, from=1-1, to=1-2]
	\arrow["{\pi_{1}}", shift left, two heads, from=1-2, to=1-1]
\end{tikzcd}$
   and
$\begin{tikzcd}[ampersand replacement=\&]
	{R_{0}^{Q \triangleright L}} \& {Q^{0}\triangleright A,}
	\arrow["{\iota_{0}}", shift left, hook, from=1-1, to=1-2]
	\arrow["{\pi_{0}}", shift left, two heads, from=1-2, to=1-1]
\end{tikzcd}$\!\! with $Q_0, Q_1 \in \cp$.
Using the projectivity of $R_{0}^{Q\triangleright L}$, let $R_{0}^{Q\triangleright L} \xrightarrow{\widehat{Q\triangleright g}} Q\triangleright M$ be a lift of $Q \triangleright g$. Assume that $Q\triangleright g$ is not split. Then by \autoref{lem:projlifts}, we have $\widehat{Q \triangleright g} \circ r_{1}^{N}\neq 0$. By \autoref{prop:DayRadicalModule}, the morphism ${(\widehat{Q\triangleright g} \circ r_{1} \circ \pi_{1})}^{1}$ then defines a non-zero element of $((Q\triangleright M)\iradc(A))(Q^{1})$. Thus, $((Q\triangleright M)\iradc(A))\neq 0$ and, by \autoref{lem:Qradical}, $M\iradc(A) \neq 0$, which contradicts our assumption. Thus, $Q\triangleright g$ is split.

    By induction on a semisimple filtration for $M$, we obtain $Q \triangleright M = \bigoplus_{L \in \on{Irr}(\modca)} {(Q\triangleright L)}^{\oplus [M:L]}$. Finally, $M$ is the image of $Q \triangleright M$ via the epimorphism $q \triangleright M$.
\end{proof}

Finally, we show that $A/\iradc(A)$ is exact, and hence it can be viewed as the appropriate analogue of the \emph{maximal semisimple quotient} $B/\on{Rad}(B)$ for a finite-dimensional $\Bbbk$-algebra $B$.

\begin{definition}\label{defliftideal}
 Let $I$ be an ideal in $A$, and let $J$ be an ideal in $A/I$. We denote by $J_{I}$ the lift of $J$ to an ideal of $A$, obtained by the indicated pullback in the category of $A$-$A$-bimodules:
\[\begin{tikzcd}
	A & {A/I} \\
	{J_{I}} & J
	\arrow["\pi", two heads, from=1-1, to=1-2]
	\arrow[dashed, hook, from=2-1, to=1-1]
	\arrow["\lrcorner"{anchor=center, pos=0.125, rotate=90}, draw=none, from=2-1, to=1-2]
	\arrow[dashed, two heads, from=2-1, to=2-2]
	\arrow[hook, from=2-2, to=1-2]
\end{tikzcd}\]
\end{definition}

Directly from \autoref{defliftideal}, and the preservation of monomorphisms under pullbacks, we deduce the following property:

\begin{lemma}\label{productlift}
 For ideals $J,J'$ in $A/I$, there is an inclusion $J_{I}J'_{I} \hookrightarrow {(JJ')}_{I}$.
\end{lemma}

\begin{corollary}\label{nilpotentlifts}
    If $I$ and $J$ are nilpotent, then $J_{I}$ is nilpotent.
\end{corollary}

\begin{proof}
    Let $n$ be such that $J^{n} = 0$. Iterating \autoref{productlift}, we find a monomorphism ${(J_{I})}^{n} \hookrightarrow {(J^{n})}_{I} = I$. The result follows by observing that, generally, for $I' \hookrightarrow I$, we have $JI' \hookrightarrow JI$. Alternatively, note that $I \hookrightarrow \iradc(A)$ and so $J_{I}^{n} \hookrightarrow \iradc(A)$, so $J_{I}^{n}$ is nilpotent and thus so is $J_{I}$.
\end{proof}

\begin{proposition}
 The algebra $A/\iradc(A)$ has no nilpotent ideals.
\end{proposition}

\begin{proof}
    Let $J$ be a nilpotent ideal in $A/\iradc(A)$. Then, by \autoref{nilpotentlifts}, $J_{\iradc(A)}$ is nilpotent, and so the composite $J_{\iradc(A)} \xhookrightarrow{\iota_{\iradc(A)}} A \xtwoheadrightarrow{\pi} A/\iradc(A)$ is zero. Since epimorphisms in an abelian category are preserved under pullbacks, the morphism $\widetilde{\pi}$ in the following pullback square is epic:
\[\begin{tikzcd}
	A & {A/\iradc(A)} \\
	{J_{\iradc(A)}} & J
	\arrow["\pi", two heads, from=1-1, to=1-2]
	\arrow["{\iota_{\iradc(A)}}", hook, from=2-1, to=1-1]
	\arrow["{\widetilde{\pi}}", two heads, from=2-1, to=2-2]
	\arrow["{\iota_{J}}"', hook, from=2-2, to=1-2]
\end{tikzcd}\]
Thus $0 = \pi \circ \iota_{\iradc(A)} = \iota_{J} \circ \widetilde{\pi}$ implies $\iota_{J} = 0$, which in turns implies $J = 0$.
\end{proof}

\appendix

\section{\texorpdfstring{$\mathcal{C}$}{C}-projective objects and Skryabin's theorem}\label{sec:Skryabin}

Given a $\mathcal{C}$-module category $\mathcal{M}$ and an object $X \in \mathcal{M}$, we denote the \emph{internal $\on{Hom}$ functor} associated to $X$, which by definition is the right adjoint of $\blank\triangleright X\colon \mathcal{C} \to \mathcal{M}$, by $\hom{X,\blank}$. See~\cite[Section~7.9]{EGNO} for details on internal $\on{Hom}$, but note that we use a notation different from $\underline{\on{Hom}}(\blank,\blank)$ used therein.

\begin{definition}
    We say that an object $X$ of a $\mathcal{C}$-module category $\mathcal{M}$ is \emph{$\mathcal{C}$-projective} if $P \triangleright X$ is projective in $\mathcal{M}$ for any $P \in \cp$. We say that $X$ is $\mathcal{C}$-non-projective if it is not $\mathcal{C}$-projective.
\end{definition}

The following is implied by~\cite[Theorem~7.10.1]{EGNO}, as well as~\cite[Proposition~4.3]{SZ}:
\begin{proposition}
 Let $\mathcal{M}$ be a finite $\mathcal{C}$-module category such that $\blank \triangleright \blank$ is right exact in both variables. Then $X \in \mathcal{M}$ is $\mathcal{C}$-projective if and only if $\,\hom{X,\blank}$ is exact.
\end{proposition}

Our main result, Theorem~\ref{thm:mainresult}, is closely related to a result of Skryabin,~\cite[Theorem~3.5]{Sk}, which shows in particular that for a finite-dimensional Hopf algebra $H$ and an $H$-simple (co)module algebra $A$, every Hopf $A$-module is projective as an $A$-module.

Reformulating this into the language of tensor categories, the result implies that if the finite tensor category $\mathcal{C}$ admits a fibre functor $\omega\colon \mathcal{C} \to \Vecc$, then for a simple algebra object $A$ and $M \in \modca$, the module $\omega(M) \in \on{mod}\omega(A)$ is projective. On the other hand, Theorem~\ref{thm:mainresult} shows in particular that, without the requirement of a fibre functor, every $M \in \modca$ is $\mathcal{C}$-projective. The following observation connects the two results:

\begin{lemma}
  Assume that $\mathcal{C}$ admits a fibre functor $\omega\colon \mathcal{C} \to \Vecc$. Then $M \in \modca$ is $\mathcal{C}$-projective if and only if $\,\omega(M)$ is projective in $\on{mod}\omega(A)$.
\end{lemma}

\begin{proof}For notational convenience, we write the proof for the category of left $A$-modules in $\cC$, which is a right $\cC$-module category.

The fibre functor allows us to interpret $\cC$ as the category of finite dimensional left modules over a finite dimensional Hopf algebra $H$, see~\cite[Chapter 5]{EGNO}.
Since $\omega(A)$ is a $\bk$-algebra equipped with a left Hopf action of $H$, we can define the smash product (also called crossed product) $A\sharp H$. This is a $\bk$-algebra with $\omega(A)\otimes H$ as underlying vector space and multiplication given by
$$(a\otimes h)(a'\otimes h')\;=\; a(h_1a')\otimes h_2h'.$$
It then follows that the category of left $A$-modules in $\cC$ is equivalent with the category of finite dimensional left $A\sharp H$-modules.

We can interpret $\omega(A)$ as a subalgebra of $A\sharp H$ and, for every $A$-module $M$ in $\cC$, we have an isomorphism
$$(A\sharp H)\otimes_{\omega(A)}\omega(M)\;\to\; M\otimes H,\quad a\otimes h\otimes v\mapsto a(h_1v)\otimes h_2$$
of $A\sharp H$-modules.

Now, by definition $M$ is $\cC$-projective if and only if the target in this isomorphism is projective. So if $\omega(M)$ is projective, it follows that $M$ is $\cC$-projective.
For the converse, assume that $M$ is $\cC$-projective. Then $\omega(M\otimes H)$ is projective as an $\omega(A)$-module, since $A\sharp H$ is projective as an $\omega(A)$-module. But $\omega(M)$ is a direct summand of $\omega(M\otimes H)$, so also projective.
\end{proof}

\section{Incompressible finite symmetric tensor categories}\label{SecSymm}

In this section, we assume that $\bk$ is an \emph{algebraically closed} field and that $\cC$ is a finite \emph{symmetric} tensor category. A finite symmetric tensor category $\cD$ is \emph{incompressible} if every symmetric tensor functor from $\cD$ to another symmetric tensor category is fully faithful.

Via Tannakian reconstruction, the structure theory of symmetric tensor categories can be reduced to the classification and study of incompressible symmetric tensor categories, see~\cite{CEO}. All currently known incompressible categories are introduced in~\cite{BE, BEO, AbEnv}.

\begin{proposition}\label{prop:maxsimp}
    There exists a (unique) simple commutative algebra $\cF(\cC)$ in $\cC$ that contains all other simple commutative algebras.
\end{proposition}
\begin{proof}
    If $A$ and $B$ are simple commutative algebras, then they are included in a simple quotient algebra of $A\otimes B$. To conclude the proof we thus only need to show that any chain
    $$A_1\subset A_2\subset A_3\subset \cdots$$
    of simple commutative algebras in $\cC$ must stabilise. By \autoref{thm:mainresult}, the algebras are exact, which implies that their module categories are tensor categories and we obtain a chain of surjective (every object in the target is a subquotient of one in the essential image) symmetric tensor functors
    $$\cC\to\mathrm{mod}_{\cC}A_1\to\mathrm{mod}_{\cC}A_2\to \mathrm{mod}_{\cC}A_3\to\cdots,$$
    given by extension of scalars. By~\cite[Proposition~6.3.4]{EGNO}, the Frobenius--Perron dimension of the categories must decrease along the chain and all functors $\mathrm{mod}_{\cC}A_i\to\mathrm{mod}_{\cC}A_{i+1}$ must be equivalences, for $i\gg 1$. Since the right adjoints of the functors are given by restriction along $A_i\to A_{i+1}$ it follows that the latter is an isomorphism.
    \end{proof}

\begin{example}\label{ex:G}\hfill
\begin{enumerate}
    \item Let $G$ be a finite group and $\cC=\mathrm{Rep}_{\bk} G$, then $\cF(\cC)=\cF(G)$ is the algebra of $\bk$-valued functions on $G$.
    \item More generally, if $\cC$ is the category of modules over a cocommutative finite dimensional Hopf algebra $H$, then $\cF(\cC)=H^\ast$.
    \item Let $G$ be a finite group with a central element $z$ of order $2$. Then for $\cC=\mathrm{Rep}(G,z)$ as in~\cite[Examples~0.4]{Del}, we have $\cF(\cC)=\cF(G/\langle z\rangle)$.
\end{enumerate}

\end{example}

The following improves upon~\cite[Theorem~A]{CEO} for finite tensor categories.
    \begin{proposition}\label{prop:uniqueincomp}
        There is, up to equivalence, a unique incompressible symmetric tensor category to which $\cC$ admits a surjective tensor functor. This category can be realised as $\mathrm{mod}_{\cC}\cF(\cC)$.
    \end{proposition}
\begin{proof}
    It is well-known that surjective symmetric tensor functors out of finite tensor categories correspond precisely to extension of scalars along commutative exact algebras, see~\cite[\S~6.3]{CEO}. The result thus follows from \autoref{prop:maxsimp} and \autoref{thm:mainresult}.
\end{proof}

\begin{example}
    The following conditions are equivalent on a finite symmetric tensor category $\cC$:
    \begin{enumerate}
        \item $\cC$ is incompressible;
        \item $\unit=\cF(\cC)$;
        \item the only simple commutative algebra in $\cC$ is $\unit$.
    \end{enumerate}
\end{example}

We conclude this appendix with a conjectural partial description of $\cF(\cC)$. If the characteristic of $\bk$ is zero, then \autoref{ex:G} deals with all cases, by~\cite[Corollaire~0.7]{Del}. Although non-essential, we thus focus on $\mathrm{char}(\bk)=p>0$.

Let $F\colon\cC\to \cD$ be a surjective tensor functor to an incompressible symmetric tensor category. By \autoref{prop:uniqueincomp} we find, for every $X\in\cC$
$$\dim_{\bk}\cC(X,\cF(\cC))\;=\;\dim_{\bk}\cD(FX,\unit_{\cD}). $$
A natural property of an incompressible category is that it be \emph{maximally nilpotent}, as defined in~\cite{ComAlg}. For instance, if~\cite[Conjecture~3.2.3]{ComAlg} is true (as verified for $p=2$ in~\cite[Theorem~9.2.1]{CEO}) and~\cite[Conjecture~1.4]{BEO} is also true, then every incompressible category is maximally nilpotent. Under this assumption on $\cD$,~\cite[Theorem~7.2.2 and Remark~7.1.2]{CEO} show that
$$\dim_{\bk}\cD(FX,\unit_{\cD})\;=\;\limsup_{n\to\infty} \frac{\log (\FPdim \Sym^{\le n}FX)}{\log(n)},$$
where $\Sym^{\le n}$ stands for the direct sum of the first $n$
 symmetric powers of an object. Since Frobenius--Perron dimension of objects is preserved by tensor functors, see~\cite[Proposition~4.5.7]{EGNO}, we thus find a conjectural description of $\cF(\cC)$:

 \begin{conjecture}\label{conj:GK}
     For every $X\in\cC$,
     $$\dim_{\bk}\cC(X,\cF(\cC))\;=\;\limsup_{n\to\infty} \frac{\log (\FPdim \Sym^{\le n}X)}{\log(n)}. $$
 \end{conjecture}

 \begin{example}
If $\cC=\mathrm{Rep}_{\bk}G$ as in \autoref{ex:G}(1), then both sides in \autoref{conj:GK} give $\dim_{\bk}X$. If $\cC=\mathrm{Rep}_{\bk}(G,z)$ as in \autoref{ex:G}(3), then both sides in \autoref{conj:GK} give the dimension of the even subspace of the super vector space $X$.
 \end{example}

\end{document}